\documentclass[a4paper]{amsart}
\usepackage[T1]{fontenc}
\usepackage{amscd,amssymb}
\usepackage{hyperref}
\usepackage[all]{xy}

\allowdisplaybreaks[2]

\numberwithin{equation}{section}

\newcommand{\bC}{\mathbb{C}}
\newcommand{\bH}{\mathbb{H}}
\newcommand{\bQ}{\mathbb{Q}}
\newcommand{\bR}{\mathbb{R}}
\newcommand{\bZ}{\mathbb{Z}}
\newcommand{\RP}{\mathbb{RP}}
\newcommand{\C}{\mathbb{C}}
\newcommand{\R}{\mathbb{R}}
\newcommand{\Z}{\mathbb{Z}}
\renewcommand{\H}{\mathbb{H}}

\newcommand{\cA}{\mathcal{A}}
\newcommand{\cB}{\mathcal{B}}
\newcommand{\cD}{\mathcal{D}}
\newcommand{\cE}{\mathcal{E}}
\newcommand{\cF}{\mathcal{F}}
\newcommand{\cG}{\mathcal{G}}
\newcommand{\cM}{\mathcal{M}}
\newcommand{\cL}{\mathcal{L}}
\newcommand{\cN}{\mathcal{N}}

\newcommand{\cC}{\mathcal{C}}

\newcommand{\cP}{\mathcal{P}}

\newcommand{\Hol}{\mathrm{Hol}}
\newcommand{\Hom}{\mathrm{Hom}}
\newcommand{\End}{\mathrm{End}}
\newcommand{\Ext}{\mathrm{Ext}}
\newcommand{\Gal}{\mathrm{Gal}}
\newcommand{\Id}{\mathrm{Id}}
\newcommand{\ev}{\mathrm{ev}}
\newcommand{\Map}{\mathrm{Map}}
\newcommand{\rk}{\mathrm{rk}}
\renewcommand{\deg}{\mathrm{deg}}

\renewcommand{\mod}{\mathrm{mod\ }}
\newcommand{\gr}{\mathrm{gr}}
\renewcommand{\Map}{\mathrm{Map}}

\newcommand{\bad}{\overline{d}}
\newcommand{\baP}{\overline{P}}
\newcommand{\baQ}{\overline{Q}}

\newcommand{\U}{\mathbf{U}}
\renewcommand{\O}{\mathbf{O}}
\newcommand{\GL}{\mathbf{GL}}
\newcommand{\BU}{B\U}
\newcommand{\BO}{B\O}
\newcommand{\SU}{\mathbf{SU}}
\newcommand{\SO}{\mathbf{SO}}
\newcommand{\Sp}{\mathbf{Sp}}
\newcommand{\BSp}{B\Sp}

\newcommand{\fm}{\mathfrak{m}}
\newcommand{\fBun}{\mathfrak{Bun}}

\newcommand{\si}{\sigma}
\newcommand{\sig}{\sigma}
\newcommand{\Si}{\Sigma}
\renewcommand{\phi}{\varphi}

\newcommand{\tHol}{\widetilde{\Hol}}

\newcommand{\hg}{\hat{g}}

\newcommand{\vp}{\vec{p}}
\newcommand{\vw}{\vec{w}}
\newcommand{\vx}{\vec{x}}

\newcommand{\ralpha}{\mathring{\alpha}}
\newcommand{\rbeta}{\mathring{\beta}}
\newcommand{\rgamma}{\mathring{\gamma}}
\newcommand{\rdelta}{\mathring{\delta}}

\newcommand{\Msi}{M^{\si}}
\newcommand{\Ms}{M^{\si}}

\newcommand{\Css}{\cC_{ss}}
\newcommand{\Cs}{\cC_s}
\newcommand{\Cmu}{\cC_\mu}
\newcommand{\Mod}{\mathcal{M}_{M,\si}^{\, r,d}}

\newcommand{\cxG}{\cG_{\C}}
\newcommand{\quot}{/\negthickspace /}
\newcommand{\Ird}{\mathrm{I}_{r,d}}
\newcommand{\Irdp}{\mathrm{I}_{r,d}\setminus\{\mu_{ss}\}}
\newcommand{\dmu}{d_{\mu}}
\newcommand{\tauR}{\tau_{\R}}
\newcommand{\tauH}{\tau_{\H}}
\newcommand{\Ctau}{\cC^{\tau}}
\newcommand{\Cmut}{\Cmu^{\tau}}
\newcommand{\Irdtau}{\Ird^{\, \tau}}
\newcommand{\Irdtaup}{\Irdtau \setminus \{\mu_{ss}\}}
\newcommand{\Csst}{\Css^{\, \tau}}
\newcommand{\Cst}{\Cs^{\, \tau}}

\newcommand{\ov}[1]{\overline{#1}}
\newcommand{\os}[1]{\overline{\sigma^*{#1}}}

\renewcommand{\mod}[1]{\mathrm{mod\ #1}}

\newcommand{\sectionsofE}{\Omega^0(M;E)}
\newcommand{\oneformsinE}{\Omega^1(M;E)}
\newcommand{\onezeroformsinE}{\Omega^{1,0}(M;E)}
\newcommand{\zerooneformsinE}{\Omega^{0,1}(M;E)}
\newcommand{\fibre}{F^{-1}(\{\mu_{ss}\})}
\newcommand{\antiHermoneforms}{\Omega^1\big(M;\mathfrak{u}(E)\big)}
\newcommand{\antiHermtwoforms}{\Omega^2\big(M;\mathfrak{u}(E)\big)}
\newcommand{\YMconn}{\mathcal{A}_{\mathrm{min}}}

\newcommand{\Cmup}{\cC_{\mu'}}
\newcommand{\Umu}{U_{\mu}}
\newcommand{\Nmu}{\mathrm{N}_\mu}

\newcommand{\euler}{e_{\cG_E}(\Nmu)}
\newcommand{\Kmu}{\mathbf{K}_\mu}
\newcommand{\Fmu}{\cF_\mu}
\newcommand{\Bmu}{\cB_\mu}
\newcommand{\cGmu}{\cG_\mu}
\newcommand{\cGt}{\cG_E^{\, \tau}}
\newcommand{\Nmut}{\Nmu^{\tau}}
\newcommand{\Umut}{\Umu^{\tau}}
\newcommand{\Fmut}{\Fmu^{\tau}}
\newcommand{\Bmut}{\Bmu^{\tau}}
\newcommand{\Gmut}{\cGmu^{\tau}}

\newcommand{\lra}{\longrightarrow}

\newcommand{\Gt}{G_{(n,a)}^{\ \tau}(r)}
\newcommand{\GtR}{G_{(n,a)}^{\ \tauR}(r)}
\newcommand{\GtH}{G_{(n,a)}^{\ \tauH}(r)}
\newcommand{\BGt}{B\big(\Gt \big)}
\newcommand{\EGt}{E\big(\Gt\big)} 
\newcommand{\Wt}{W_{(g,n,a)}^{\ \tau}(r,d)}
\newcommand{\Vt}{V_{(g,n,a)}^{\ \tau}(r,d)}

\newcommand{\Bun}{\mathcal{B}un_{ss,\mu}}
\newcommand{\BunR}{\mathcal{B}un_{ss,\mu}^{\ \R}}
\newcommand{\BunH}{\mathcal{B}un_{ss,\mu}^{\ \H}}
\newcommand{\Ker}{\mathrm{Ker}\,}
\renewcommand{\Im}{\mathrm{Im}\,}
\newcommand{\Mods}{\mathcal{N}_{M,\si}^{\, r,d}}

\newcommand{\p}{\mathsf{p}}

\newtheorem{dummy}{dummy}[section]
\newtheorem{lemma}[dummy]{Lemma}
\newtheorem{theorem}[dummy]{Theorem}

\newtheorem{corollary}[dummy]{Corollary}
\newtheorem{proposition}[dummy]{Proposition}
\newtheorem{definition}[dummy]{Definition}

\theoremstyle{definition}
\newtheorem*{acknowledgements}{Acknowledgments}

\newtheorem{remark}[dummy]{Remark}

\begin{document}

\title{The Yang-Mills Equations over Klein Surfaces}
\author{Chiu-Chu Melissa Liu}
\address{Department of Mathematics, Columbia University, New York City, NY, USA.}
\email{ccliu@math.columbia.edu}
\author{Florent Schaffhauser}
\address{Departamento de Matem\'aticas, Universidad de Los Andes, Bogot\'a, Colombia.}
\email{florent@uniandes.edu.co}

\date{February 28, 2013}

\begin{abstract}
Moduli spaces of semi-stable real and quaternionic vector bundles of a fixed topological type admit a presentation as Lagrangian quotients and can be embedded into the symplectic quotient corresponding to the moduli variety of semi-stable holomorphic vector bundles of fixed rank and degree on a smooth complex projective curve. From the algebraic point of view, these Lagrangian quotients are connected sets of real points inside a complex moduli variety endowed with a real structure; when the rank and the degree are coprime, they are in fact the connected components of the fixed-point set of the real structure. This presentation as a quotient enables us to generalize the methods of Atiyah and Bott to a setting with involutions and compute the mod 2 Poincar\'{e} polynomials of these moduli spaces in the coprime case. We also compute the mod 2 Poincar\'{e} series of moduli stacks of all real and quaternionic vector bundles of a fixed topological type. As an application of our computations, we give new examples of maximal real algebraic varieties.
\end{abstract}

\maketitle

\tableofcontents

\section{Introduction}\label{intro}

\subsection{Klein surfaces}
A \textbf{Klein surface} $M/\si$ is the quotient of a (connected) Riemann surface $M$ by an anti-holomorphic involution $\si$ (\cite{AG}). It is sometimes better to think of it as the pair $(M,\sigma)$, to which there is associated a real algebraic curve $X/\R$, whose set of closed points is $|X| = M/\sigma$ (a real surface which either is non-orientable or has non-empty boundary, possibly both, but orientable surfaces without boundary are excluded). The topological classification of compact Klein surfaces was first obtained by Felix Klein (\cite{Klein2})~: $(M,\si)$ is topologically classified by the triple $(g,n,a)$, where

\begin{enumerate}
\item $g$ is the genus of $M$,
\item $n$ is the number of connected components of $M^{\si}$ (the fixed-point set of $\si$ in $M$),
\item $a$ is the index of orientability of $M/\si$~: $a=0$ if $M/\si$ is orientable and $a=1$ if $M/\si$ is non-orientable (equivalently, $a=2\ -$ the number of connected components of the complement of $M^{\si}$ in $M$).
\end{enumerate}
This means that there exists a homeomorphism $\phi:(M,\si) \to (M',\si')$ such that $\si' = \phi \si \phi^{-1}$ if and only if $(g,n,a) = (g',n',a')$. 
We shall call $(g,n,a)$ the \textbf{topological type} of $(M,\sigma)$. Sometimes, we also write $X(\R)$ for $\Msi$ and $X(\C)$ for $M$. 
The topological type $(g,n,a)$ of a Klein surface $(M,\si)$ satisfies

\begin{enumerate}
\item $0\leq n \leq g+1$ (Harnack's theorem),
\item if $n=0$, then $a=1$,
\item if $n=g+1$, then $a=0$,
\item if $a=0$, then $n \equiv (g+1)\ \mod{2}$,
\end{enumerate}
and for each triple $(g,n,a)$ satisfying these conditions, there exists a Klein surface of topological type $(g,n,a)$.

\subsection{Topology of moduli spaces of holomorphic vector bundles} 
Given a compact connected Klein surface $X \leftrightarrow(M,\si)$ of genus $g\geq 2$, we denote $\Mod$  
the moduli scheme parametrizing $S$-equivalence classes of semi-stable holomorphic vector bundles of rank $r$ and degree $d$ on $M=X(\C)$,
and denote $\Mods$ the open dense sub-scheme of $\Mod$ parametrizing the isomorphism classes of stable holomorphic vector bundles
of rank $r$ and degree $d$ on $M$. Then $\Mod(\bC)$ is a complex projective variety and $\Mods(\bC)$ is a non-singular complex variety.

If $E$ is a fixed, smooth complex vector bundle of rank $r$ and degree $d$ on $M$, we denote $\cC$ the set of holomorphic structures (Dolbeault operators) on $E$, 
$\Css$ (resp.\ $\Cs$) the set of semi-stable (resp.\ stable) holomorphic structures on $E$ and $\cxG$ the group of all complex linear automorphisms of $E$ 
(the complex gauge group). Then $\cxG$ acts on $\cC$ and the generic stabilizer is $\mathcal{Z}(\cG_\C) \simeq \bC^*$. The quotient stack $[\cC/\cxG]$ of this action can 
be identified with the moduli stack $\fBun_{M,\si}^{r,d}$ of {\em all} holomorphic structures on $E$:
$$
\fBun_{M,\si}^{r,d} = [\cC/\cG_\bC]. 
$$
Let $\ov{\cG_\C} = \cxG/\C^*$. Then $\ov{\cG_\bC}$ acts freely on $\Cs$ and there is a homeomorphism
$$
\Mods(\bC) \simeq \Cs/\ov{\cG_\bC}. 
$$
When $r$ and $d$ are coprime, one has
\begin{enumerate}
\item $\Css=\Cs$,
\item $\Mod=\Mods$,
\item $\Mod(\C)$ is a smooth projective variety of complex dimension $r^2(g-1) + 1$.
\end{enumerate}
Moreover, Atiyah and Bott have shown that, \textit{when} $\mathit{r\wedge d=1}$, the map $$H^*(B\cG_\C;\bQ) \longrightarrow H^*(B\C^*;\bQ)\, ,$$ induced by the inclusion of the centre of $\cG_\C$, is surjective (\cite[p.577]{AB}, see also p.545 for the analogous result with \textit{integral} coefficients): a subsequent application of the Leray-Hirsch theorem to the top row fibration of the following commutative diagram
$$
\begin{CD}
B\C^*@>>> E\cG_\C \times_{\cG_\C} X @>>> E\ov{\cG_\C} \times_{\ov{\cG_\C}} X \\
@|  @VVV  @VVV \\
B\C^* @>>> B\cG_\C @>>> B\ov{\cG_\C}
\end{CD}
$$
where $X$ is any $\ov{\cG_\C}$-space, then shows that $$H^*_{\cG_\C}(X;\bQ) \simeq H^*(B\C^*;\bQ) \otimes H^*_{\ov{\cG_\C}}(X;\bQ)$$ in this case.
In particular, when $r\wedge d=1$, one has 
$$
P_t^{\cG_\bC}(\Css;\bQ) = P_t(B\bC^*;\bQ)P_t^{\ov{\cG_\bC}}(\Css;\bQ) =\frac{1}{1-t^2} P_t(\Mod(\bC);\bQ).
$$ 

\noindent Let  
$$
P_g(r,d):= P_t^{\cG_\bC}(\Css;\bQ)
$$
be the rational $\cG_\bC$-equivariant Poincar\'{e} series of $\Css$, the set of semi-stable holomorphic structures on 
a fixed smooth complex vector bundle of rank $r$ and degree $d$ on a Riemann surface of genus $g\geq 2$ (this series does indeed only depend on $g$, $r$ and $d$, and not on the complex structure of $M$, which can be seen as a consequence of the Narasimhan-Seshadri theorem, \cite{NS65}). 
In \cite{AB}, Atiyah and Bott computed $P_g(r,d)$ for any $g\geq 2$, $r\geq 1$ and $d\in\Z$.
In particular, when $r\wedge d=1$, they compute the rational Poincar\'{e} polynomial of the smooth projective variety
$\Mod(\bC)=\Mods(\bC)$. The following are the main ingredients of their approach: 
\subsubsection*{1. Poincar\'{e} series of the classifying space of the gauge group} 
$\cC$ is contractible, so 
$$
P_t^{\cG_\bC}(\cC;\bQ) = P_t(B\cG_\bC;\bQ).
$$ 
Denote $Q_g(r)$ the above series, which can also be interpreted as the rational Poincar\'{e} series
of the moduli stack $\fBun_{M,\si}^{r,d}$ (see \cite{HS}):
$$
Q_g(r) = P_t^{\cG_\bC}(\cC;\bQ) =  P_t([\cC/\cG_\bC];\bQ) = P_t(\fBun_{M,\si}^{r,d};\bQ).
$$
Then
\begin{equation}\label{eqn:Q}
Q_g(r) =\frac{\prod_{j=1}^r (1+t^{2j-1})^{2g}}{\prod_{j=1}^{r-1}(1-t^{2j}) \prod_{j=1}^r (1-t^{2j}) }\cdot
\end{equation}
In particular, it only depends on $g$ and $r$.

\subsubsection*{2. Equivariantly perfect stratification}
Let $\cC_\mu\subset \cC$ denote the set of holomorphic structures
of Harder-Narasimhan type
\begin{equation}\label{eqn:mu}
\mu=\Big(\underbrace{\frac{d_1}{r_1}, \cdots, \frac{d_1}{r_1}}_{r_1\ \mathrm{times}}, \cdots, 
\underbrace{\frac{d_l}{r_l}, \cdots, \frac{d_l}{r_l}}_{r_l\ \mathrm{times}}\Big)
\end{equation} 
on $E$, where 
$$
\frac{d_1}{r_1} > \frac{d_2}{r_2} > \cdots > \frac{d_l}{r_l}.
$$ 
In particular, $\Css = \cC_{\mu_{ss}}$, where 
$$
\mu_{ss} = \Big(\frac{d}{r}, \cdots, \frac{d}{r}\Big).
$$ 
The set of all Harder-Narasimhan types of holomorphic structures on $E$ is denoted $\Ird$. 
The complex codimension of $\Cmu$ in $\cC$ is finite and equal to 
\begin{equation}\label{eqn:dmu}
\dmu = \sum_{1 \leq i < j \leq l} r_i r_j \big(\mu_i - \mu_j + (g-1)\big).
\end{equation}
In particular, $\Css$ is open in $\cC$.  The set 
$$
\{ \Cmu : \mu\in\Ird\}
$$ 
is a $\cG_\bC$-equivariantly perfect stratification of $\cC$ over the field $\bQ$. In particular,

\begin{equation*}
P_t^{\cG_\bC}(\cC;\bQ) = \sum_{\mu\in \Ird } t^{2 d_\mu} P_t^{\cG_\bC}(\cC_\mu;\bQ),
\end{equation*}

\noindent and the \textit{Kirwan map}

\begin{equation*}
H^*(B\cG_\bC;\bQ) \longrightarrow H^*_{\cG_\bC}(\Css;\bQ)
\end{equation*}

\noindent is surjective.

\subsubsection*{3. Equivariant Poincar\'{e} series of positive-codimensional strata}
Let $\mu\in \Irdp$ be as in \eqref{eqn:mu}. Then
$$
P_t^{\cG_\bC}(\Cmu;\bQ) =\prod_{i=1}^l P_g(r_i, d_i). 
$$ 

\bigskip

\noindent The above three ingredients give the following formula, which computes $P_g(r,d)$ recursively in terms of $Q_g(r)$:
\begin{theorem}[Atiyah-Bott recursive formula, \cite{AB}]
$$
P_g(r,d) \quad = \quad Q_g(r)\quad -\sum_{\mu\in \Ird\setminus \{\mu_{ss}\}} t^{2d_\mu} \prod_{i=1}^l P_g(r_i,d_i),
$$
where $Q_g(r)$ is given by \eqref{eqn:Q} and $d_\mu$ is given by \eqref{eqn:dmu}.
\end{theorem}

\noindent Zagier derived a closed formula which solves the Atiyah-Bott recursive formula.

\begin{theorem}[Zagier's closed formula, \cite{Zagier}] \label{thm:zagier}
\begin{eqnarray*}
P_g(r,d)&=&  \sum_{l=1}^r \sum_{\substack{r_1,\ldots, r_l\in \bZ_{>0}\\ \sum r_i =r}} 
(-1)^{l-1} \frac{t^{2(\sum_{i=1}^{l-1}(r_i+r_{i+1})\langle (r_1+\cdots + r_i)\frac{d}{r}\rangle + (g-1)\sum_{i<j} r_ir_j) } }{\prod_{i=1}^{l-1} (1-t^{2(r_i+r_{i+1})}) }\\
&& \qquad \qquad \qquad \quad
 \prod_{i=1}^l\frac{\prod_{j=1}^{r_i}(1+t^{2j-1})^{2g}}{
\prod_{j=1}^{r_i-1}(1-t^{2j})\prod_{j=1}^{r_i}(1-t^{2j})}
\end{eqnarray*}
where $\langle x\rangle = [x] + 1 -x$ denotes, for a real number $x$, the unique $t\in (0,1]$ with $x+t\in \bZ$.
\end{theorem}

\subsection{Topology of moduli spaces of real and quaternionic vector bundles}
\label{topology_moduli_real_quaternionic}
A \textbf{real vector bundle}, in the sense of Atiyah (\cite{Atiyah_real_bundles}), on the Klein surface / real algebraic curve $(M,\si)$ is a pair $(\cE,\tau)$ where $\cE \to M$ is a holomorphic 
vector bundle and $\tau:\cE \to \cE$ is an anti-holomorphic map such that
\begin{enumerate}
\item the diagram
\begin{equation*}
\begin{CD}
\cE    @>\tau>> 	\cE \\
@VVV 			@VVV \\
M      @>\si>> 	        M
\end{CD}
\end{equation*}
is commutative,

\item the map $\tau$ is $\C$-anti-linear fibrewise,

\item $\tau^2 = \Id_{\cE}$.
\end{enumerate}
A \textbf{quaternionic vector bundle} on $(M,\si)$ is a pair $(\cE,\tau)$ satisfying conditions (1) and (2) and the modified condition (3') $\tau^2 = -\Id_{\cE}$. 
We observe that this definition makes sense in various categories, for instance the category of smooth Hermitian vector bundles with complex linear isometries 
as homomorphisms between them, in which case it is also required that $\tau$ be an isometry.  A real or quaternionic vector bundle $(\cE,\tau)$ is called \textbf{semi-stable} if $\cE$ satisfies 
the slope semi-stability condition 
$$
\mu(\cF) = \frac{\deg\, \cF}{\rk\, \cF} \leq \frac{\deg\, \cE}{\rk\, \cE} = \mu(\cE)
$$ 
for any non-trivial, $\tau$-invariant sub-bundle $\cF \subset \cE$. This turns out to be equivalent to slope 
semi-stability for the holomorphic vector bundle $\cE$ (see Proposition \ref{semi-stability_and_geometric_semi-stability}).
We say that a real or quaternionic vector bundle $(\cE,\tau)$ is {\bf geometrically stable} if $\cE$ is a stable holomorphic vector bundle;
it is strictly stronger than  being stable as a real (resp.\ quaternionic) bundle (see Definition \ref{stability_conditions} and 
Proposition \ref{stability_in_the_real_sense}).

\noindent

Exactly as in the Atiyah-Bott approach, we fix a real (resp.\ quaternionic) $C^{\infty}$ vector 
bundle $(E,\tau)$ of rank $r$ and degree $d$  (complete numerical invariants for such bundles were found by Biswas, Huisman and Hurtubise in \cite{BHH}, 
and are recalled in Theorem \ref{top_type_of_bundles}). 
We consider the set $\Csst$ (resp.\ $\Cst$) of all $\tau$-compatible semi-stable (resp.\ stable) holomorphic structures on $(E,\tau)$. 
This means that $\tau$ is an anti-holomorphic map with respect to this holomorphic structure, turning the associated holomorphic bundle $\cE$ 
into a real (resp.\ quaternionic) bundle smoothly isomorphic to $(E,\tau)$. We also denote $\cG_{\C}^{\tau}$ the group of all complex linear automorphisms of $E$ 
that commute to $\tau$. Then $\cG_{\C}^{\tau}$ acts on $\Ctau$, the set of all $\tau$-compatible holomorphic structures on $(E,\tau)$, 
and the generic stabilizer for this action is the subgroup $\mathcal{Z}(\cG_\C^\tau) =\mathcal{Z}(\cG_\C)\cap\cG_\C^\tau\simeq (\C^*)^{\tau}$ of scalar automorphisms of $E$ 
that commute to the $\C$-anti-linear map $\tau$, so $(\C^*)^{\tau} \simeq \R^*$.  The quotient stack $[\cC^\tau/\cG^\tau_\bC]$ can
be identified with the moduli stack $\fBun_{M,\si}^{r,d,\tau}$ of {\em all} $\tau$-compatible holomorphic structures on $(E,\tau)$:
$$
\fBun_{M,\si}^{r,d,\tau}=[\cC^\tau/\cG_{\bC}^\tau].
$$
Let $\ov{\cG_\C^\tau} = \cG_\bC^\tau/\R^*$. Then $\ov{\cG_\bC^\tau}$ acts freely on $\Cst$.
Let $\cM_{M,\si}^{r,d,\tau}$ be the space of real (resp.\ quaternionic) $S$-equivalence classes of semi-stable real (resp.\ quaternionic) vector bundles
that are smoothly isomorphic to $(E,\tau)$ and let $\cN_{M,\si}^{r,d,\tau}$ be the space of isomorphism classes of geometrically stable
real (resp.\ quaternionic) vector bundles that are smoothly isomorphic to $(E,\tau)$ (the precise definitions will be given
in Section \ref{real_and_quaternionic_bundles}, Definitions \ref{stability_def} and \ref{real_quat_S-equiv}). Then
$$
\cN_{M,\si}^{r,d,\tau} = \Cst/\ov{\cG_\bC^\tau}.
$$
When $r\wedge d=1$, we shall see that:
\begin{enumerate}
\item $\Csst=\Cst$,
\item $\cM_{M,\si}^{r,d,\tau} =\cN_{M,\si}^{r,d,\tau}$,
\item $\cM_{M,\si}^{r,d,\tau}$ is a connected component of the smooth manifold $\Mod(\R)$, which is compact and of real dimension $r^2(g-1) + 1$,
\end{enumerate}
and we shall compute the $\mod{2}$ equivariant Poincar\'{e} series
$$
P^{\ \tau}_{(g,n,a)}(r,d):= P_t^{\cG_\bC^\tau}(\Csst;\Z/2\Z)
$$
for all $r$ and $d$, not necessarily coprime. As will be clear from the final expression, this series does only depend on $g$, $r$, $d$ and the topological invariants of $\si$ and $\tau$, and not on the complex structure of $M$, a fact that is also explained by the comparison theorem between the equivariant cohomology of $\Csst$ and that of a certain representation variety attached to it (Theorem \ref{finite-dim_description}). In order to relate our computation of the equivariant Poincar\'e  series of $\Csst$  to the actual Poincar\'e series of the moduli space $\cM_{M,\si}^{r,d,\tau}$, we shall prove that, \textit{when} $\mathit{r\wedge d=1}$, the map $$H^*(B\cG_\C^\tau;\Z/2\Z) \longrightarrow H^*(B\R^*;\Z/2\Z)\, ,$$ induced by the inclusion of the centre of $\cG_\C^\tau$, is surjective: a subsequent application of the Leray-Hirsch theorem to the top row fibration of the commutative diagram
$$
\begin{CD}
B\R^*@>>> E\cG_\C^\tau \times_{\cG_\C^\tau} X @>>> E\ov{\cG_\C^\tau} \times_{\ov{\cG_\C^\tau}} X \\
@|  @VVV  @VVV \\
B\R^* @>>> B\cG_\C^\tau @>>> B\ov{\cG_\C^\tau}
\end{CD}
$$
where $X$ is any $\ov{\cG_\C^\tau}$-space, will show (Theorem \ref{comparison_of_Poincare_series_in_good_cases}) that $$H^*_{\cG_\C^\tau}(X;\Z/2\Z) \simeq H^*(B\R^*;\Z/2\Z) \otimes H^*_{\ov{\cG_\C^\tau}}(X;\Z/2\Z)$$ in this case.
In particular, when $r\wedge d =1$, $\cM_{M,\si}^{r,d,\tau}=\cN_{M,\si}^{r,d,\tau}$ is a smooth compact connected manifold and we have the following equalities (Corollary \ref{cohom_of_moduli_space}):
$$
P_t^{\cG_\bC^\tau}(\Csst;\Z/2\Z) = P_t(B\bR^*;\bZ/2\Z)P_t^{\ov{\cG_\C^\tau}}(\Csst;\Z/2\Z) =\frac{1}{1-t}P_t(\cM_{M,\si}^{r,d,\tau};\bZ/2\Z).
$$

Our strategy to compute $P^{\ \tau}_{(g,n,a)}(r,d)$ is to follow the three main steps of Atiyah and Bott's computation of $P_{g}(r,d)$. We denote $\Irdtau$ the set of real (resp.\ quaternionic) Harder-Narasimhan types of $\tau$-compatible holomorphic structures on $E$ (see Definition \ref{real_HN_types}). In particular, if $\mu\in\Irdtau$,  there is, associated to it, a uniquely defined holomorphic Harder-Narasimhan type, which we also denote $\mu$, and which satisfies $\Cmut := \cC^{\tau} \cap \Cmu\neq\emptyset$. The equivariant normal bundle to $\Cmu^\tau$ in $\cC^\tau$ is a real vector bundle in the usual sense which is not orientable in general (see Subsection \ref{orientability}), so we are forced to consider cohomology with mod 2 coefficients. As a comment on our choice of notation, let us mention here that we shall define, for each choice of a real (resp.\ quaternionic) bundle $(E,\tau)$ of rank $r$ and degree $d$, an (affine) involution of $\cC$ preserving $\Css$ and all of the other $\Cmu$, as well as an involution of $\cG_{\C}$, both induced by $\tau$, in such a way that $\Csst$, $\Cmut$ and $\cG_{\C}^{\tau}$ are precisely the fixed-point sets of these involutions. We note that, if we consider vector bundles of degree $0$ on a closed non-orientable surface  $M/\si$, the involution considered in the present paper is different from the involution in \cite{HL, HLR}. We may then summarize our results as follows.

\subsubsection*{1. Poincar\'{e} series of the classifying space of the gauge group}
$\cC^\tau$ is contractible, so
$$
P_t^{\cG_{\C}^{\tau}}(\cC^\tau;\Z/2\Z) = P_t(B\cG^\tau_\C;\Z/2\Z).
$$ 
It turns out that, if $(E,\tau)$ is a real (resp.\ quaternionic) smooth vector bundle of rank $r$ and degree $d$ on a Klein surface of topological type $(g,n,a)$, $P_t(B\cG^\tau_\C;\Z/2\Z)$ depends on $(g,n,a)$ and $r$, but not on $d$. We denote $Q_{(g,n,a)}^{\  \tau}(r)$ this Poincar\'e series, which can also be
interpreted as the mod 2 Poincar\'{e} series of the moduli stack $\fBun_{M,\si}^{r,d,\tau}$:
$$
Q_{(g,n,a)}^{\ \tau}(r)= P_t^{\cG_{\C}^{\tau}}(\cC^\tau;\Z/2\Z)= P_t([\cC^\tau/\cG^\tau_\bC];\Z/2\Z) = P_t(\fBun_{M,\si}^{r,d,\tau};\bZ/2\bZ).
$$
\begin{theorem}  \label{classifying_space}
Let $(M,\si)$ be a Klein surface of topological type $(g,n,a)$. 
\begin{enumerate}
\item Let $(E,\tauR)$ be a real smooth vector bundle of rank $r$ and degree $d$ over
$(M,\si)$. Then
\begin{equation}
Q_{(g,n,a)}^{\ \tauR}(r) \ \ = \quad \frac{\prod_{j=1}^r(1+t^{2j-1})^{g-n+1} \prod_{j=1}^{r-1}(1+t^j)^n 
\prod_{j=1}^r (1+t^j)^n}{\prod_{j=1}^{r-1}(1-t^{2j})\prod_{j=1}^{r}(1-t^{2j})}\cdot
\end{equation}
\item Let $(E,\tauH)$ be a quaternionic smooth vector bundle of rank $r$ and degree $d$
over $(M,\si)$. Then
\begin{enumerate}
\item[(i)] if $(g,n,a)=(g,0,1)$, 
\begin{equation}
Q_{(g,0,1)}^{\ \tauH}(r) \ \ = \quad \frac{ \prod_{j=1}^r (1+t^{2j-1})^{g+1} }{ \prod_{j=1}^{r-1} (1-t^{2j}) \prod_{j=1}^r (1-t^{2j})}\cdot
\end{equation} 
This coincides with $Q_r^{\tauR}(g,0,1)$.
\item[(ii)] if $(g,n,a)$ satisfies $n>0$, in which case the rank of a quaternionic vector bundle is necessarily even, 
\begin{equation}
Q_{(g,n,a)}^{\ \tauH}(r) \ \ = \quad\frac{ \prod_{j=1}^r (1+t^{2j-1})^g}{ \prod_{j=1}^{r/2-1} (1-t^{4j}) \prod_{j=1}^{r/2} (1-t^{4j}) }\cdot
\end{equation}
\end{enumerate}
\end{enumerate}
\end{theorem}
We compute $Q_{(g,n,a)}^{\tau}(r)$ using mod 2 cohomological Leray-Hirsch spectral sequences associated
to various fibrations; these spectral sequences do not collapse at the $E_2$-page in general.
See \cite{Baird} for an alternative approach to part (1) of Theorem \ref{classifying_space}.

\subsubsection*{2. Equivariantly perfect stratification}
\begin{theorem}\label{stratification_thm} One has:
\begin{enumerate}
\item The real codimension of $\Cmut$ in $\cC^\tau$ is finite and equal to $\dmu$, the complex codimension of $\Cmu$ in $\cC$. 
In particular, $\Csst$ is open in $\Ctau$.
\item The set 
$$
\{ \Cmut : \mu\in\Irdtau\}
$$ is a $\cG_\C^{\, \tau}$-equivariantly perfect stratification of $\Ctau$ over the field $\Z/2\Z$.
In particular,
\begin{equation*}
P^{\cG_\C^{\tau}}_t(\Ctau;\Z/2\Z)\quad  = \quad \sum_{\mu \in \Irdtau} t^{\dmu} P^{\cG_\C^{\tau}}_t(\Cmut;\Z/2\Z).
\end{equation*}

\noindent and the \textit{real Kirwan map}

\begin{equation*}
H^*(B\cG_\bC^{\tau};\Z/2\Z) \longrightarrow H^*_{\cG_\bC^{\tau}}(\Csst;\Z/2\Z)
\end{equation*}

\noindent is surjective.

\end{enumerate}
\end{theorem}

\subsubsection*{3. Equivariant Poincar\'{e} series of positive-codimensional strata}
Let $P^{\ \tau}_{(g,n,a)}(r,d)$ be the $\cG_\C^{\tau}$-equivariant Poincar\'e series of the set $\Csst$ of $\tau$-compatible, semi-stable holomorphic structures on $(E,\tau)$. 
\begin{theorem} Let
$$
\mu=\Big(\underbrace{\frac{d_1}{r_1}, \cdots, \frac{d_1}{r_1}}_{r_1\ \mathrm{times}}, \cdots, 
\underbrace{\frac{d_l}{r_l}, \cdots, \frac{d_l}{r_l}}_{r_l\ \mathrm{times}}\Big) \in \Irdtaup. 
$$
Then 
$$
P_t^{\cG_{\C}^{\tau}}(\Cmu^\tau;\bZ/2\Z) =\prod_{i=1}^l P^{\ \tau}_{(g,n,a)}(r_i, d_i). 
$$ 
\end{theorem}

\noindent Combining the above three steps, we obtain the following recursive formula, which computes $P^{\ \tau}_{(g,n,a)}(r,d)$ 
in terms of $Q_{(g,n,a)}^{\ \tau}(r)$: 
\begin{theorem}[Recursive formula] 
\begin{equation}\label{rec_formula}
P^{\ \tau}_{(g,n,a)}(r,d) \quad =\quad  Q_{(g,n,a)}^{\ \tau}(r)\quad - \sum_{\mu \in \Irdtaup} t^{\dmu} \prod_{i=1}^l P^{\ \tau}_{(g,n,a)} (r_i,d_i).
\end{equation}
where $Q_{(g,n,a)}^{\ \tau}(r)$ is given by Theorem \ref{classifying_space}, 
and $d_\mu$ is given by \eqref{eqn:dmu}.
\end{theorem}

\noindent Finally, we derive the following closed formulae.
\begin{theorem}[Closed formulae] \label{thm:closed-formula} 
\begin{eqnarray*}
& & P^{\ \tauR}_{(g,0,1)}(r,2d)\\
&=& \sum_{l=1}^r\sum_{\substack{r_1,\ldots, r_l\in\bZ_{>0}\\ \sum r_i=r}}
(-1)^{l-1}  \frac{t^{2(\sum_{i=1}^{l-1} (r_i+r_{i+1}) \langle (r_1+\cdots + r_i)(\frac{d}{r})\rangle)} }
{\prod_{i=1}^{l-1}(1-t^{2(r_i+r_{i+1})} ) } t^{(g-1)\sum_{i<j} r_i r_j}
\\
& & \qquad \qquad \qquad \quad \prod_{i=1}^l \frac{\prod_{j=1}^{r_i} (1+t^{2j-1})^{g+1} }
{ \prod_{j=1}^{r_i-1}(1-t^{2j})\prod_{j=1}^{r_i}(1-t^{2j}) }
\end{eqnarray*}
\begin{eqnarray*}
& & P^{\ \tauH}_{(2g'-1,0,1)}(r,2d) \\
& = & \sum_{l=1}^r\sum_{\substack{r_1,\ldots, r_l\in\bZ_{>0}\\ \sum r_i=r}}
(-1)^{l-1} \frac{t^{2\sum_{i=1}^{l-1} (r_i+r_{i+1}) \langle (r_1+\cdots + r_i)(\frac{d}{r})\rangle} }
{\prod_{i=1}^{l-1}(1-t^{2(r_i+r_{i+1})} ) } t^{ (2g'-2)\sum_{i<j} r_i r_j}\\
& & \qquad \qquad \qquad \quad  \prod_{i=1}^l \frac{\prod_{j=1}^{r_i} (1+t^{2j-1})^{2g'} }
{ \prod_{j=1}^{r_i-1}(1-t^{2j})\prod_{j=1}^{r_i} (1-t^{2j}) }
\end{eqnarray*}
\begin{eqnarray*}
& & P^{\ \tauH}_{(2g',0,1)}(r,2d+r) \\
& = & \sum_{l=1}^r\sum_{\substack{r_1,\ldots, r_l\in\bZ_{>0}\\ \sum r_i=r}}
(-1)^{l-1} \frac{t^{2\sum_{i=1}^{l-1} (r_i+r_{i+1}) \langle (r_1+\cdots + r_i)(\frac{d}{r})\rangle} }
{\prod_{i=1}^{l-1}(1-t^{2(r_i+r_{i+1})} ) } t^{ (2g'-1)\sum_{i<j} r_i r_j}\\
& & \qquad \qquad \qquad \quad  \prod_{i=1}^l \frac{\prod_{j=1}^{r_i} (1+t^{2j-1})^{2g'+1} }
{ \prod_{j=1}^{r_i-1}(1-t^{2j})\prod_{j=1}^{2r_i} (1-t^{2j}) }
\end{eqnarray*}

Suppose that $n>0$. Then
\begin{eqnarray*}
&& P^{\ \tauR}_{(g,n,a)}(r,d)\\
&=& \sum_{l=1}^r\sum_{\substack{r_1,\ldots, r_l\in\bZ_{>0}\\ \sum r_i=r}}
(-1)^{l-1}\frac{t^{\sum_{i=1}^{l-1} (r_i+r_{i+1}) \langle (r_1+\cdots + r_i)(\frac{d}{r})\rangle}}
{\prod_{i=1}^{l-1}(1-t^{r_i+r_{i+1}} ) } t^{(g-1)\sum_{i<j} r_i r_j}
\\
& & \qquad \quad   2^{(n-1)(l-1)} \prod_{i=1}^l \frac{\prod_{j=1}^{r_i} (1+t^{2j-1})^{g-n+1} \prod_{j=1}^{r_i-1}(1+t^j)^n\prod_{j=1}^{r_i}(1+t^j)^n }
{ \prod_{j=1}^{r_i-1}(1-t^{2j})\prod_{j=1}^{r_i}(1-t^{2j}) }
\end{eqnarray*} and
\begin{eqnarray*}
&& P^{\ \tau_\bH}_{(g,n,a)}(2r,2d)\\
&=& \sum_{l=1}^{r}\sum_{\substack{r_1,\ldots, r_l\in\bZ_{>0}\\ \sum r_i=r}}
(-1)^{l-1} \frac{t^{4\sum_{i=1}^{l-1} (r_i+r_{i+1}) \langle (r_1+\cdots + r_i)(\frac{d}{r})\rangle} }
{\prod_{i=1}^{l-1}(1-t^{4(r_i+r_{i+1})} ) } t^{4(g-1)\sum_{i<j} r_i r_j}\\
& & \qquad \qquad \qquad \quad  \prod_{i=1}^l \frac{\prod_{j=1}^{2r_i} (1+t^{2j-1})^g }
{ \prod_{j=1}^{r_i-1}(1-t^{4j})\prod_{j=1}^{r_i}(1-t^{4j}) }
\end{eqnarray*}
\end{theorem}

\subsection{Notation}\label{notation}

In the remainder of the paper, we often denote $\Z_2$ the field $\Z/2\Z$. We also work throughout with a fixed Hermitian metric on the smooth complex vector bundle $E$ and we denote $\cG_E$ the group of unitary automorphisms of $E$ (the unitary gauge group). $\cxG$ is in a natural way the complexification of $\cG_E$. If $E$ is endowed with a real (resp.\ quaternionic) Hermitian structure $\tau$, the involution of $\cG_\C$ induced by $\tau$ preserves $\cG_E$ and we denote $\cG_E^{\, \tau} = \cG_E \cap \cxG^{\tau}$ the fixed-point set of the resulting involution of $\cG_E$. The deformation retract $\GL_r(\bC)\leadsto \U(r)$ induces a deformation retract $\cG_\bC \leadsto \cG_E$, which
restricts to a deformation retract $\cG_\bC^\tau \leadsto \cG_E^{\, \tau}$.
As a consequence, the $\cxG$-equivariant cohomology of $\Css$ is the same as its $\cG_E$-equivariant cohomology and the $\cxG^{\tau}$-equivariant 
cohomology of $\Css^{\, \tau}$ is the same as its $\cG_E^{\, \tau}$-equivariant cohomology. Also, the classifying spaces $B\cxG$ and $B\cG_E$ have the 
same homotopy type and so do $B\cxG^{\tau}$ and $B\cG_E^{\, \tau}$. Since we have chosen a Hermitian metric on $E$, we may think of a holomorphic structure on $E$ as a unitary connection 
$$
d_A : \sectionsofE \longrightarrow \oneformsinE = \onezeroformsinE \oplus \zerooneformsinE.
$$ 
If $(E,\tau)$ is a real (resp.\ quaternionic) Hermitian vector bundle, then $\Omega^k(M;E)$ has a real 
(resp.\ quaternionic) structure given by pulling back the differential form then applying the real (resp.\ quaternionic) structure
\begin{equation}\label{real_or_quat_structure_on_k-forms}
\eta \longmapsto \ov{\eta} := \tau \circ \eta \circ \sigma.
\end{equation} 
In this case, the holomorphic structure $d_A$ is called \textbf{real} (resp.\ \textbf{quaternionic}) if it commutes 
with the real (resp.\ quaternionic) structures of $\sectionsofE$ and $\oneformsinE$. This is the exact necessary and sufficient 
condition for $\tau$ to induce a real (resp.\ quaternionic) structure on the space $\ker(d_A^{0,1})$ of holomorphic sections of $\cE:=(E,d_A)$, 
turning the holomorphic vector bundle $(\cE,\tau) $ into a real (resp.\ quaternionic) holomorphic bundle. 
To avoid having to continuously distinguish between real and quaternionic connections on a real or quaternionic Hermitian bundle $(E,\tau)$, we simply call them $\mathbf{\tau}$\textbf{-compatible}. 

Given a topological space $Y$ with a continuous action by a topological group $G$, let
$Y_{hG}=EG\times_G Y$ denote the homotopy orbit space and let $[Y/G]$ denote the quotient stack. Then
$H^*([Y/G];R) := H^*_G(Y;R) = H^*(Y_{hG};R)$ for any coefficient ring $R$.

\begin{acknowledgements}
We would like to thank Alejandro \'Adem, Thomas Baird, Erwan Brugall\'e,
Jacques Hurtubise, Robert Lipshitz, Walter Neumann, Bernardo Uribe, and
Richard Wentworth for helpful conversations. We thank the referee for helpful
comments on the first version of this paper.
\end{acknowledgements}

\section{Real and quaternionic structures and their moduli}
\label{real_and_quaternionic_bundles}

In this section, we summarize the results of \cite{Sch_JSG}. We give precise definitions of
the moduli spaces $\cM_{M,\si}^{r,d,\tau}$  and $\cN_{M,\si}^{r,d,\tau}$ which appeared
in Section \ref{topology_moduli_real_quaternionic} and identify
$\cM_{M,\si}^{r,d,\tau}$ with a Lagrangian quotient $\cL_\tau$ (Theorem \ref{Lag_quotient_as_moduli_space}). 

Recall that a moduli problem for geometric objects (=topological spaces with an additional geometric structure) typically has  the following two aspects~:

\begin{enumerate}
\item a topological classification,
\item the construction of a moduli space for objects of each topological type, by which we mean, here, a manifold whose points are in bijection with certain equivalence classes (ideally, but not always, isomorphism classes) of the objects for which one seeks moduli.
\end{enumerate}

\noindent A more refined notion of moduli space is obtained by searching for universal families, but we shall not touch upon that aspect here.

The topological classification of real and quaternionic vector bundles on a Klein surface $(M,\si)$ of topological type $(g,n,a)$ was obtained by Biswas, Huisman and Hurtubise in \cite{BHH} and we recall their result in Subsection \ref{top_types_of_bundles}. As for the second aspect of the moduli problem, it has been known since the work of Mumford on Geometric Invariant Theory, that topologically and geometrically well-behaved moduli spaces may only be obtained if one imposes a certain stability condition on the objects that one wishes to classify. In the context of vector bundles on curves, slope stability probably is the obvious choice. Nonetheless, some care should be taken when it comes to the kind of sub-bundle on which to test the slope stability condition.  We recall the definition of stability in the real and quaternionic sense and the differences with stability in the holomorphic sense, in Subsection \ref{stability_conditions}. We subsequently propose a geometric-invariant-theoretic and a gauge-theoretic construction of moduli spaces for real and quaternionic bundles (Subsections \ref{moduli_spaces} and \ref{gauge_viewpoint}) and explain how this construction fits into two-dimensional Yang-Mills theory (Subsection \ref{Yang-Mills_theory}).

\subsection{Topological types of real and quaternionic bundles}\label{top_types_of_bundles}

We collect the topological classification results of Biswas, Huisman and Hurtubise in the following theorem.

\begin{theorem}[Topological types of real and quaternionic bundles, \cite{BHH}]\label{top_type_of_bundles}
One has~:
\begin{enumerate}

\item For real bundles~:

\begin{enumerate}

\item[(i)] if $\Ms = \emptyset$, then real Hermitian bundles on $(M,\si)$ are topologically classified by their rank and degree. 
It is necessary and sufficient for a real Hermitian bundle of rank $r$ and degree $d$ to exist that 
$$
d \equiv 0\ (\mod{2}).
$$

\item[(ii)] if $\Ms \not= \emptyset$ and $(E,\tau)$ is real, then $(E^\tau \to \Ms)$ is a real vector bundle in the ordinary sense, on the disjoint union 
$$
\Ms = \gamma_1 \sqcup \cdots \sqcup \gamma_n
$$ 
of at most $(g+1)$ circles and we denote $$w^{(j)} := w_1(E^\tau\big|_{\gamma_j}) \in H^1(S^1; \Z / 2\Z)
\simeq \Z / 2\Z$$ the first Stiefel-Whitney class of 
$E^\tau\to \Ms$  restricted to $\gamma_j$.\\ Then real Hermitian bundles on $(M,\si)$ are topologically classified by their 
rank, their degree and the sequence  $\vec{w}:=(w^{(1)}, \cdots, w^{(n)})$. 
It is necessary and sufficient for a real Hermitian bundle with given invariants $r$, $d$ and 
$\vec{w}$ to exist that 
$$
w^{(1)} + \cdots + w^{(n)} \equiv d\ (\mod 2).
$$
\end{enumerate}

\item For quaternionic bundles~:\\ Quaternionic Hermitian bundles on $(M,\sig)$ are topologically classified by their rank and degree. 
It is necessary and sufficient for a topological quaternionic bundle of rank $r$ and degree $d$ to exist that 
$$
d + r(g-1) \equiv 0\ (\mod 2).
$$

\end{enumerate}

\end{theorem} 

\noindent Note that if $\Msi\neq\emptyset$ and $(E,\tau)$ is quaternionic, then $\rk\,E$ is even, because the fibres of $E|_{\Msi}\longrightarrow \Msi$ are left modules over the field of quaternions. Also, if $(E,\tau)$ is a real bundle over a Klein surface $(M,\si)$ of topological type $(g,n,a)$ with $n>0$, the sequence $\vec{w}:=(w^{(1)}, \cdots, w^{(n)})$ is of course none other than the first Stiefel-Whitney class of the vector bundle $E^\tau \longmapsto M^\si=\gamma_1\sqcup \cdots \sqcup \gamma_n$, which is a real vector bundle in the ordinary sense (with fibre $\R^r$) over the disconnected base $\Msi$.

\subsection{Stability of real and quaternionic bundles}\label{stability_conditions}

The slope of a non-zero holomorphic vector bundle $\cE$ is the ratio $$\mu(\cE) := \frac{\deg\,\cE}{\rk\,\cE}$$ of its degree by its rank.

\begin{definition}[Stability conditions for real and quaternionic bundles]\label{stability_def}
Let $(\cE,\tau)$ be a real (resp.\ quaternionic) holomorphic vector bundle on $(M,\sigma)$. We call a sub-bundle of $\cE$ non-trivial if it is distinct from $\{0\}$ and from $\cE$. 
\begin{enumerate}
\item We say $(\cE, \tau)$ is \textbf{stable} if, for any non-trivial $\tau$-invariant sub-bundle $\cF \subset \cE$, the slope stability condition $$\mu(\cF) < \mu(\cE)$$ is satisfied.
\item We say $(\cE, \tau)$ is \textbf{semi-stable} if, for any non-trivial  $\tau$-invariant sub-bundle $\cF \subset \cE$, one has $$\mu(\cF) \leq \mu(\cE).$$
\item We say $(\cE, \tau)$ is \textbf{geometrically stable} if the underlying holomorphic bundle $\cE$ is stable, that is, if, for any non-trivial sub-bundle $\cF \subset \cE$,  one has $$\mu(\cF) < \mu(\cE).$$
\item We say $(\cE, \tau)$ is\textbf{geometrically semi-stable}, if the underlying holomorphic bundle $\cE$ is semi-stable, that is, if for any non-trivial sub-bundle $\cF \subset \cE$,  one has $$\mu(\cF) \leq \mu(\cE).$$
\end{enumerate}
\end{definition}

\noindent We see that (3) $\Rightarrow$ (1) and (4) $\Rightarrow$ (2). We recall below that (2) $\Rightarrow$ (4), but (1) $\not\Rightarrow$ (3).

\begin{proposition}[\cite{Sch_JSG}]\label{semi-stability_and_geometric_semi-stability}
Let $(\cE,\tau)$ be a semi-stable real (resp.\ quaternionic) vector bundle on $(M,\sig)$. Then $(\cE,\tau)$ is geometrically semi-stable.
\end{proposition}

\noindent To see that (1) does not necessarily imply (3), we identify all bundles $(\cE,\tau)$ which are stable in the real (resp.\ quaternionic) sense. We note that when $\cF$ is any holomorphic vector bundle, there is a commutative diagram

$$\begin{CD}
\os{\cF}  @>\tau>>  \cF \\
@VVV  @VVV \\
M @>\sigma>> M
\end{CD}$$

\noindent where $\tau$ is an invertible, $\C$-antilinear map covering $\sig$ and such that $$\tau \circ \tau^{-1}=\Id_{\cF},\ \mathrm{and}\ \tau^{-1} \circ \tau = \Id_{\os{\cF}}.$$ Therefore, on $\cF \oplus\os{\cF}$, we may define $$\tau^+ = \begin{pmatrix} 0 & \tau \\ \tau^{-1} & 0 \end{pmatrix}\quad \mathrm{and}\quad \tau^- = \begin{pmatrix} 0 & -\tau \\ \tau^{-1} & 0 \end{pmatrix}\, .$$ $\tau^+$ and $\tau^-$ are $\C$-antilinear maps from $\cF\oplus\os{\cF}$ to itself, covering $\sig$, and satisfying $$\tau^+ \circ \tau^+ = \begin{pmatrix} \Id_{\cF} & 0 \\ 0 & \Id_{\os{\cF}} \end{pmatrix} = \Id_{\cF\oplus\os{\cF}} ,$$ and $$\tau^- \circ \tau^- = \begin{pmatrix} -\Id_{\cF} & 0 \\ 0 & -\Id_{\os{\cF}} \end{pmatrix} = -\Id_{\cF\oplus\os{\cF}}.$$ In other words, $(\cF\oplus\os{\cF},\tau^+)$ is a real bundle and $(\cF\oplus\os{\cF},\tau^-)$ is a quaternionic bundle. We also note that, if $(\cE,\tau)$ is any real (resp.\ quaternionic) bundle, the bundle $\End(\cE) \simeq \cE^*\otimes\cE$ of endomorphisms  of $\cE$ always has a \textit{real} structure given by $$\xi\otimes v \longmapsto (\ov{\xi\circ \tau^{-1}}) \otimes \tau(v).$$ If we still denote $\tau$ this real structure, the bundle of real (resp.\ quaternionic) endomorphisms of $(\cE,\tau)$ is the bundle $\big(\End(\cE)\big)^{\tau}$ of $\tau$-invariant elements of $\End(\cE)$.

\begin{proposition}[\cite{Sch_JSG}]\label{stability_in_the_real_sense}
Let $(\cE,\tau_{\cE})$ be a stable real (resp.\ quaternionic) vector bundle.
\begin{enumerate}
\item Then either $(\cE,\tau_{\cE})$ is geometrically stable, or there exists a holomorphic vector bundle $\cF$, stable in the holomorphic sense, such that $\cE= \cF \oplus \os{\cF}$. In the latter case, if $(\cE,\tau)$ is real then $\os{\cF}\neq \cF$ and $\tau_{\cE}=\tau^+$, and if $(\cE,\tau)$ is quaternionic, then $\tau_{\cE} = \tau^-$.
\item In the geometrically stable case, the set of real (resp.\ quaternionic) endomorphisms of $(\cE,\tau_{\cE})$ is $$\big(\End(\cE)\big)^{\tau_{\cE}} = \{\lambda \Id_{\cE} : \lambda\in\R\} \simeq_{\R} \R\, ,$$ and, if $\cE=\cF \oplus \os{\cF}$, then $$\big(\End(\cE)\big)^{\tau_{\cE}} = \{(\lambda\Id_{\cF},\ov{\lambda}\Id_{\cF}) : \lambda\in\C\} \simeq_{\R} \C\, .$$
\end{enumerate}
\end{proposition}

\noindent Note that the isomorphisms given in part (2) of the Proposition are isomorphisms of \textit{real} vector spaces. Also, a real (resp.\ quaternionic) bundle which is stable in the real (resp.\ quaternionic) sense but not geometrically stable, is necessarily of even rank.\\ The key feature into the moduli problem for real (resp.\ quaternionic) bundles on $(M,\si)$ is that there are enough real (resp.\ quaternionic) bundles which are stable in the real (resp.\ quaternionic) sense for a semi-stable real (resp.\ quaternionic) bundle $(\cE,\tau)$ to admit a real (resp.\ quaternionic) Jordan-H\"older filtration in the following sense.

\begin{definition}
Let $(\cE,\tau)$ be a real (resp.\ quaternionic) bundle. A \textbf{real (resp.\ quaternionic) Jordan-H\"older filtration}  of $(\cE,\tau)$ is a filtration $$\{0\} = \cE_0 \subset \cE_1 \subset \cdots \subset \cE_k=\cE$$ by $\tau$-invariant holomorphic sub-bundles, whose successive quotients are stable in the real (resp.\ quaternionic) sense.
\end{definition}

\begin{theorem}[\cite{Sch_JSG}]\label{Abelian_cat}
Let $\BunR$ (resp.\ $\BunH$) denote the category of semi-stable real (resp.\ quaternionic) bundles of slope $\mu$ on $(M,\sig)$. By Proposition \ref{semi-stability_and_geometric_semi-stability}, it is a strict sub-category of the category $\Bun$ of semi-stable holomorphic bundles of slope $\mu$. Moreover~:
\begin{enumerate}
\item If $u:(\cE_1,\tau_1) \longrightarrow (\cE_2,\tau_2)$ is a morphism of real (resp.\ quaternionic) bundles, then the bundles $\Ker u$ and $\Im u$ are semi-stable real (resp.\ quaternionic) bundles of slope $\mu$ and the isomorphism $\cE/\Ker u \simeq \Im u$ is an isomorphism of real (resp.\ quaternionic) bundles. As a consequence, $\BunR$ (resp.\ $\BunH$) is an Abelian category.
\item The Abelian category $\BunR$ (resp.\ $\BunH$) is Artinian, Noetherian and stable by extensions. If $(\cE,\tau)$ is stable in the real (resp.\ quaternionic) sense, then its endomorphism ring $(\End\, \cE)^{\tau}$ is a field which is an algebraic extension of $\R$, so it is either $\R$ or $\C$.
\item The simple objects of $\BunR$ (resp.\ $\BunH$) are the real (resp.\ quaternionic) bundles of slope $\mu$ on $(M,\sig)$ that are stable in the real (resp.\ quaternionic) sense. In particular, a semi-stable real (resp.\ quaternionic) bundle $(\cE,\tau)$ admits a real (resp.\ quaternionic) Jordan-H\"older filtration.
\end{enumerate}
\end{theorem}

\noindent Note that a real (resp.\ quaternionic) Jordan-H\"older filtration of a semi-stable real (resp.\ quaternionic) bundle$(\cE,\tau)$ is \emph{not} a Jordan-H\"older filtration of the underlying holomorphic bundle $\cE$ (for instance, if $(\cE,\tau) = (\cF \oplus\os{\cF},\tau^{\pm})$ with $\os{\cF}\not\simeq \cF$ and $\cF$ geometrically stable, then $(\cE,\tau)$ is stable as a real (resp.\ quaternionic) bundle so its real (resp.\ quaternionic) Jordan-H\"older filtrations have length one, while its holomorphic Jordan-H\"older filtrations have length two). The graded object associated to a real (resp.\ quaternionic) Jordan-H\"older filtration of a semi-stable real (resp.\ quaternionic) bundle $(\cE,\tau)$ is a poly-stable object in the sense of the following definition.

\begin{definition}[Poly-stable real and quaternionic bundles]\label{poly-stability}
A real (resp.\ quaternionic) vector bundle $(\cE,\tau)$ on $(M,\sig)$ is called \textbf{poly-stable} if there exist real (resp.\ quaternionic) 
bundles $(\cF_j,\tau_j)_{1\leq j\leq k}$ of equal slope, stable in the real (resp.\ quaternionic) sense, such that $$\cE \simeq \cF_1 \oplus\cdots \oplus \cF_k$$ and $$\tau = \tau_1 \oplus \cdots \oplus \tau_k.$$ 
\end{definition}

\noindent By Proposition \ref{stability_in_the_real_sense}, a poly-stable real (resp.\ quaternionic) bundle is poly-stable in the holomorphic sense (=a direct sum of stable holomorphic bundles of equal slope). We recall that the holomorphic $S$-equivalence class of a semi-stable holomorphic bundle $\cE$ is, by definition (\cite{Seshadri}), the graded isomorphism class of the poly-stable bundle $\gr(\cE)$ associated to any Jordan-H\"older filtration of $\cE$.

\begin{corollary}[\cite{Sch_JSG}]\label{real_embedding}
The $S$-equivalence class, as a holomorphic bundle, of a semi-stable real (resp.\ quaternionic) bundle $(\cE,\tau)$ contains a poly-stable real (resp.\ quaternionic) bundle in the sense of Definition \ref{poly-stability}. Any two such objects are isomorphic as real (resp.\ quaternionic) poly-stable bundles.
\end{corollary}

\noindent In particular, there is a well-defined notion of real (resp.\ quaternionic) $S$-equivalence class for a semi-stable real (resp.\ quaternionic) bundle $(\cE,\tau)$.

\begin{definition}[Real and quaternionic $S$-equivalence classes]\label{real_quat_S-equiv}
The graded isomorphism class, in the real (resp.\ quaternionic) sense, of the poly-stable real (resp.\ quaternionic) bundle $\gr(\cE,\tau)$ associated to any real (resp.\ quaternionic) Jordan-H\"older filtration of $(\cE,\tau)$, is called the \textbf{real (resp.\ quaternionic) $S$-equivalence class} of $(\cE,\tau)$.
\end{definition}
 
\noindent We point out that a same poly-stable object may admit, however, both a real and a quaternionic structure, showing that it belongs both to a real and to a quaternionic $S$-equivalence class (for instance, $\cF\oplus\os{\cF}$ admits the real structure $\tau^+$ and the quaternionic structure $\tau^-$). A final instructive example is given as follows. Let $(\cL,\tau)$ be a real (resp.\ quaternionic) line bundle on $(M,\sig)$. Then $\cL\oplus\cL$ admits two non-conjugate, non-stable, real (resp.\ quaternionic) structures, namely $$\tau \oplus \tau \quad \mathrm{and}\quad \tau^+ = \begin{pmatrix} 0 & \tau \\ \tau & 0 \end{pmatrix}.$$ We note that $\tau^+$ is indeed quaternionic when $\tau$ is quaternionic. These two non-conjugate poly-stable real (resp.\ quaternionic) structures $\tau\oplus\tau$ and $\tau^+$ are, however, $S$-equivalent in the real (resp.\ quaternionic) sense. Indeed, $$\gr(\cL\oplus\cL, \tau \oplus\tau) = (\cL,\tau) \oplus (\cL,\tau)$$ and $(\cL \oplus \cL, \tau^+)$ admits the real (resp.\ quaternionic) Jordan-H\"older filtration $$\{0\} \subset \cL_{\Delta} \subset \cL \oplus \cL\, ,$$ where $\cL_{\Delta}$ is the image of the diagonal embedding
 \begin{eqnarray*} 
 \cL & \longrightarrow & \cL \oplus \cL \\
u & \longmapsto & (u,u) \, .
\end{eqnarray*} 
In particular, $(\cL_{\Delta},\tau^+|_{\cL_{\Delta}})$ is isomorphic to $(\cL,\tau)$ as a real (resp.\ quaternionic) bundle. Moreover, the map 
$$ \begin{array}{ccc} 
 (\cL \oplus \cL) / \cL_{\Delta} & \longrightarrow & \cL \\
(v,w) & \longmapsto & i(v - w)
\end{array}$$ 
is an isomorphism of real (resp.\ quaternionic) bundles with respect to $\tau^+$ and $\tau$, so $$\gr(\cL\oplus\cL,\tau^+) \simeq (\cL,\tau) \oplus (\cL,\tau).$$

\subsection{Moduli of semi-stable real and quaternionic bundles}\label{moduli_spaces}

Motivated by the results of the previous subsection, we look for a  space whose points are in bijection with real (resp.\ quaternionic) $S$-equivalence classes of semi-stable real (resp.\ quaternionic) bundles of fixed topological type. Since the moduli variety $\Mod(\C)$ is the set of holomorphic $S$-equivalence classes of semi-stable holomorphic bundles of topological type $(r,d)$, it is natural to look for moduli spaces of real (resp.\ quaternionic) bundles that would be subspaces of $\Mod(\C)$. More specifically, since the functor $\cE \longmapsto \os{\cE}$ preserves the rank, the degree and the slope (semi-) stability of a holomorphic vector bundle, it takes a holomorphic Jordan-H\"older filtration $$\{0\} =\cE_0 \subset \cE_1 \subset \cdots \subset \cE_k =\cE$$ to the holomorphic Jordan-H\"older filtration $$\{0\} = \os{\cE_0} \subset \os{\cE_1} \subset \cdots \subset \os{\cE_k} = \os{\cE}\, ,$$ so it induces an anti-holomorphic involution $$[\cE]_S \longmapsto [\os{\cE}]_S$$ of $\Mod(\C)$, whose fixed-point set $\Mod(\R)$ contains holomorphic $S$-equivalence classes of semi-stable real (resp.\ quaternionic) bundles of rank $r$ and degree $d$. Unfortunately, a real point of $\Mod$ is not necessarily the real (resp.\ quaternionic) $S$-equivalence class of a real (resp.\ quaternionic) poly-stable bundle, as one can see by considering the direct sum $(\cE_1,\tau_1) \oplus (\cE_2,\tau_2)$ of a stable real bundle and a stable quaternionic bundle. The statement becomes true, however, if we restrict our attention to geometrically stable real (resp.\ quaternionic) bundles, as shown by the next Proposition.

\begin{proposition}[\cite{BHH}]\label{stable_and_Galois_invariant}
Assume that $\cE$ is a bundle on $M$, which is stable in the holomorphic sense and such that $\os{\cE} \simeq \cE$. Then $\cE$ is either real or quaternionic and it cannot be both.
\end{proposition}

\noindent If we denote $\Mods$ the open sub-scheme of $\Mod$ parametrizing isomorphism classes of stable holomorphic bundles of rank $r$ and degree $d$ on $M$, then $\Mods(\C)$ is a non-singular complex variety and is equal to $\Mod(\C)$ when $r$ and $d$ are coprime. We have the following description of $\Mods(\R)$.

\begin{theorem}[\cite{Sch_JSG}]
The points of $\Mods(\R)$ are in bijection with isomorphism classes of geometrically stable real and quaternionic bundles. Moreover, two geometrically stable real bundles belong to the same connected component of $\Mods(\R)$ if and only if they are smoothly isomorphic. Geometrically stable quaternionic bundles lie in a different, single connected component of $\Mods(\R)$.
\end{theorem}

\noindent This includes the case where $r\wedge d=1$, as in this case a stable real (resp.\ quaternionic) bundle necessarily is geometrically stable, but the statement of the theorem is no longer correct for $\Mod(\R)$ in general, for, as we have already noted, a poly-stable object of the form $\cF\oplus\os{\cF}$ admits both a real and a quaternionic structure. The Theorem also says that connected components of $\Mods(\R)$ are indexed by topological types of real and quaternionic bundles, which we can easily count using Theorem \ref{top_types_of_bundles} (for instance, it is no greater than $2^g + 1$, see \cite{Sch_JSG} for an exact count).  
Let $\cN_{M,\si}^{r,d,\tau}$ denote the space of real (resp.\ quaternionic) isomorphism classes of geometrically stable real (resp.\ quaternionic) bundles of a fixed topological type (determined by $r,d,\tau$) over a Klein surface $(M,\si)$. Then $\cN_{M,\si}^{r,d,\tau}$ is a connected component of $\cN_{M,\si}^{r,d}(\bR)$. 

The problem, when $r$ and $d$ are not coprime, is that real points of $\Mod$ do not represent \textit{real} (resp.\ \textit{quaternionic}) $S$-equivalence classes of semi-stable real (resp.\ quaternionic) bundles of rank $r$ and degree $d$. Fortunately, it is possible, using gauge theory, to produce topological spaces whose points are in bijection exactly with the elements of the set $\cM_{M,\si}^{r,d,\tau}$ of real (resp.\ quaternionic) $S$-equivalence classes of semi-stable real (resp.\ quaternionic) bundles of fixed topological type. When $r$ and $d$ are coprime, these topological spaces are smooth manifolds which embed onto the various connected components of $\Mod(\R)$.

\subsection{The gauge-theoretic point of view}\label{gauge_viewpoint}

We refer to Section 3.2 of \cite{Sch_JSG} for the explicit computations of this subsection. Let $(E,\tau)$ be a fixed real (resp.\ quaternionic) Hermitian bundle of rank $r$ and degree $d$ on $(M,\si)$. The affine space $\cC$ of holomorphic structures on $E$ is in bijection with the affine space $\mathcal{A}_E$ of unitary connections $$d_A: \sectionsofE \longrightarrow \oneformsinE = \onezeroformsinE \oplus \zerooneformsinE$$ via the map sending $d_A$ to its $(0,1)$-part, denoted $\ov{\partial}_A$. In the rest of the paper, we constanttly identify $\cC$ and $\mathcal{A}_E$ in this manner (which depends on the choice of the metric on $E$). The set of isomorphism  classes of holomorphic vector bundles of rank $r$ and degree $d$ is in bijection with the set 
$$
\cC / \cxG
$$ 
of orbits of unitary connections on $E$ under the action 
$$
g(A) = A - (\ov{\partial}_A g) g^{-1} + \big( (\ov{\partial}_A g) g^{-1}\big)^*
$$ 
of the complex gauge group of $E$. This restricts to the usual gauge action of $\cG_E \subset \cxG$ on $\cC$ given by $$u(A) = A - (d_A u)u^{-1}$$ for $u$ unitary. Atiyah and Bott have showed that the affine space $\cC$ has a natural K\"ahler structure, with complex structure induced by the Hodge star $*$ of $M$ and compatible symplectic structure given, on $T_A\cC \simeq \antiHermoneforms$, by $$\omega_A(a,b) = \int_M -\mathrm{tr}(a\wedge b)\, .$$ The remarkable property (\cite{AB}) is that the action of the unitary gauge group $\cG_E$ on $\cC$ is Hamiltonian with respect to this symplectic form, with the curvature map  
$$
F:
\begin{array}{ccl}
\cC & \longrightarrow & \antiHermtwoforms \simeq (\mathrm{Lie}\, \cG_E)^* \\
A & \longmapsto & F_A
\end{array}
$$
for a momentum map. A celebrated theorem of Donaldson then gives necessary and sufficient conditions for a holomorphic bundle $\cE$ to be stable in terms of the corresponding $\cxG$-orbit $O(\cE)$ of unitary connections (this gives a new proof of a theorem of Narasimhan-Seshadri: 
the stable holomorphic vector bundles over a compact Riemann surface are precisely those arising from irreducible projective unitary representations of the fundamental group \cite{NS65}). 

\begin{theorem}[Donaldson, \cite{Don_NS}]
A holomorphic vector bundle $\cE$ of rank $r$ and degree $d$ on $M$ is stable if, and only if, the corresponding $\cxG$-orbit $O(\cE)$ of unitary connections on $E$ contains an irreducible, minimal Yang-Mills connection, meaning a unitary connection $A$ such that~:
\begin{enumerate}
\item $\mathrm{Stab}_{\cxG}(A) = \C^*$,
\item $F_A = *i 2\pi \frac{d}{r} \mathrm{I}_r$.
\end{enumerate}
\noindent Moreover, such a connection is unique up to a unitary automorphism of $E$.
\end{theorem}

\noindent As a consequence, graded isomorphism classes of poly-stable holomorphic bundles of rank $r$ and degree $d$ are in bijection with $\cG_E$-orbits of minimal Yang-Mills connections. It will be convenient to have the following notation at our disposal~:
$$
\YMconn:= \fibre\, ,
$$ 
where $\mu_{ss} = \big(\frac{d}{r}, \cdots, \frac{d}{r})$ has been identified with $*i2\pi\frac{d}{r}\Id_E\in \antiHermtwoforms$, the notation $\YMconn$ being justified by the fact that $\YMconn$ is the set of absolute minima of the Yang-Mills functional 
$$
L_{YM}:
\begin{array}{ccc}
\cC & \longrightarrow & \R \\
A & \longmapsto & \int_M \|F_A\|^2
\end{array}
$$
for unitary connections on $E$. Donaldson's Theorem then implies that there is a homeomorphism $$\Mod(\C) = \Css\quot \cxG \simeq \fibre / \cG_E\, ,$$ where $\Css\quot\cxG$ designates the set of $S$-equivalence classes of semi-stable holomorphic structures on $E$.

It turns out that there is a similar presentation for the set 
$$
\cM_{M,\si}^{r,d,\tau}:= \Css^{\, \tau} \quot \cxG^\tau
$$ of real (resp.\ quaternionic) $S$-equivalence classes of semi-stable $\tau$-compatible structures on $(E,\tau)$. As a first step towards this, we note that $\tau$ induces an anti-symplectic, involutive isometry
$$
\alpha_\tau:
\begin{array}{ccc}
\cC & \longrightarrow & \cC \\
A & \longmapsto & \phi\, \os{A}\, \phi^{-1}
\end{array}
$$
where $\phi:\os{E} \overset{\simeq}{\longrightarrow} E$ is the bundle isomorphism corresponding to $\tau$ 
(one has $\os{\phi}=\phi^{-1}$ if $\tau$ is real and $\os{\phi} = -\phi^{-1}$ is $\tau$ is quaternionic; 
notice that, in either case, the transformation $\alpha_{\tau}$ is \textit{involutive}), as well as involutions
$$
\beta_\tau:
\begin{array}{ccc}
\cG_E & \longrightarrow & \cG_E \\
u & \longmapsto & \phi\, \os{u}\, \phi^{-1}
\end{array}
$$
and
$$
\beta_\tau:
\begin{array}{ccc}
\antiHermtwoforms & \longrightarrow & \antiHermtwoforms \\
R & \longmapsto & \phi\, \os{R}\, \phi^{-1}
\end{array}\quad .
$$
We denote both involutions by $\beta_\tau$ because the second one is induced by the first one under the identification 
$$
\antiHermtwoforms \simeq \big(Lie(\cG_E)\big)^*.
$$ 
It should be noted that the involution $\beta_\tau$ on $\cG_E \subset \cG_\C$ in fact comes from an involution $\beta_\tau: g \mapsto \phi\, \os{g}\, \phi^{-1}$ defined on the whole of $\cG_\C$. It is convenient to simply denote 
$$
\ov{A}:=\alpha_{\tau}(A),\quad  \ov{u}:=\beta_{\tau}(u),\quad \ov{g} := \beta_\tau(g),\quad \ov{R}:=\beta_{\tau}(R).
$$ 
We have the following compatibility relations (\cite{Sch_JSG}), 
$$
\ov{u(A)} = \ov{u} (\ov{A})\quad \mathrm{and}\quad F_{\ov{A}} = \ov{F_A},
$$ 
between the involution of $\cC$ and the gauge action and between the involution of $\cC$ and the momentum map of the gauge action. Similarly, $\cG_\C$ also acts on $\cC$ in a compatible way~: $$\ov{g(A)} = \ov{g} (\ov{A})$$ for all $g\in\cG_\C$.
These relations ensure that $\alpha_{\tau}$ preserves the minimal set $$\YMconn=\fibre,$$ and that $\cG_E^{\, \tau}$ (the group of fixed points 
of $\beta_{\tau}$ on $\cG_E$) acts on $\YMconn^{\, \tau},$ the fixed-point set of the restriction of $\alpha_{\tau}$ to $\YMconn$. As a consequence of all these compatibilities, we can form the Lagrangian quotient
\begin{equation}\label{Lagrangian}
\cL_\tau := \left(\fibre\right)^{\tau} / \cG_E^{\tau}\, .
\end{equation}
The group $\cG_E^{\, \tau}$ is exactly the group of unitary automorphisms of $E$ that commute to $\tau$ and we call it the \textbf{real} (resp.\ \textbf{quaternionic})
\textbf{gauge group} of $(E,\tau)$. When we want to emphasize the real or quaternionic nature of $\tau$, 
we write $\tauR$ or $\tauH$, respectively. Similarly, we write $\cG_E^{\tauR}$ and $\cG_E^{\tauH}$, respectively for the real and the quaternionic gauge group. 
Distinguishing between real and quaternionic structures will be of importance in Section \ref{topology_of_groups}, when we compute the Poincar\'e series 
of $B(\cG_E^{\tauR})$ and $B(\cG_E^{\tauH})$. Results of Section \ref{stratification_section}, in contrast, do not depend on the type of $\tau$.

Our second step is to notice that, for fixed $\tau$, a unitary connection $A$ on $(E,\tau)$ induces a $\tau$-compatible holomorphic structure if and only if $A=\ov{A}$. This is so simply because the covariant derivative 
$$
d_A: \sectionsofE \longrightarrow \oneformsinE
$$ 
commutes to the real (resp.\ quaternionic) structures of $\sectionsofE$ and $\oneformsinE$ if and only if 
$$
d_A\ov{s} = \ov{d_A s},
$$ 
which, because of the relation $d_{\ov{A}} s= \ov{d_A \ov{s}}$ defining $\ov{A}$, is equivalent to 
$$
d_{\ov{A}} = d_A.
$$ 
Since a semi-stable $\tau$-compatible holomorphic structure on $(E,\tau)$ gives rise to a semi-stable real (resp.\ quaternionic) holomorphic bundle $(\cE,\tau)$ which admits a real (resp.\ quaternionic) Jordan-H\"older filtration, the poly-stable real (resp.\ quaternionic) structure of the associated graded object $\gr(\cE,\tau)$ is defined by a unitary connection $$A\in \left(\fibre\right)^{\tau}\, .$$ In fact, this unitary connection is unique up to the action of the real (resp.\ quaternionic) gauge group, as shown by the next result.

\begin{proposition}[\cite{Sch_JSG}]\label{single_real_orbit}
Let $A,A'$ be two connections which satisfy $\ov{A} = A$ and $\ov{A'} = A'$, and assume that $A$ and $A'$ define poly-stable real (resp.\ quaternionic) structures. Then $A$ and $A'$ lie in the same $\cG_E$-orbit if, and only if, they lie in the same $\cG_E^{\, \tau}$-orbit.
\end{proposition}

\noindent We note that, since a poly-stable real (resp.\ quaternionic) bundle is poly-stable in the holomorphic sense and since the involution $\alpha_\tau$ on $\cC$ induces the involution $[\cE]_S \longmapsto [\os{\cE}]_S$ on $\Css\quot \cxG = \Mod(\C)$, one has a natural embedding 
$$
\Csst \quot \cxG^\tau = \cM_{M,\si}^{r,d,\tau} \hookrightarrow \Mod(\R)
$$ 
(the injectivity of this map follows from Corollary \ref{real_embedding}). Moreover, by Proposition \ref{single_real_orbit}, the map sending a $\cG_E^{\, \tau}$-orbit of  minimal Yang-Mills connections to the $\cG_E$-orbit containing it is injective. In sum, we have proved the following result.

\begin{theorem}[\cite{Sch_JSG}]\label{Lag_quotient_as_moduli_space}
There is a homeomorphism 
$$
\cL_\tau = \left(\fibre\right)^\tau / \cG_E^{\, \tau} \longrightarrow \cM_{M,\si}^{r,d,\tau} = \Csst \quot \cxG^\tau \subset \Mod(\R)
$$ 
between the space of gauge equivalence classes of $\tau$-compatible minimal Yang-Mills connections on $(E,\tau)$ and the space of real (resp.\ quaternionic) $S$-equivalence classes of semi-stable $\tau$-compatible holomorphic structures on $(E,\tau)$. 
\end{theorem}

\noindent In that sense, the set $\cL_\tau$, which is naturally in bijection with a set of real points of $\Mod$, is a moduli space for semi-stable real (resp.\ quaternionic) bundles \textit{which have the same topological type as} $(E,\tau)$. 

When $r\wedge d=1$, any semi-stable bundle is in fact stable, so $\ov{\cxG^\tau}=\cxG^\tau/\R^*$ acts freely
on $\Csst = \Cst$ and  
$$
\cL_\tau = \Csst/\ov{\cxG^\tau}.
$$ 
In this case, $\cL_\tau$ is a compact connected manifold of real dimension $r^2(g-1) + 1$ and we will show (Corollary \ref{cohom_of_moduli_space}) that its $\mod{2}$ Poincar\'e polynomial is 
$$
P_t(\cL_\tau;\Z/2\Z) = P_t^{\ov{\cxG^{\tau}}}(\Csst;\Z/2\Z) = (1-t)\, P_t^{\cxG^{\tau}}(\Csst;\Z/2\Z)\, .
$$


\subsection{The Yang-Mills equations over a Klein surface}\label{Yang-Mills_theory}

To conclude this section, we explain a way of thinking about our moduli problem in terms of two-dimensional Yang-Mills theory. 

As shown by Atiyah and Bott (\cite{AB}), the Yang-Mills equations over a compact Riemann surface $M$, i.e. the Euler-Lagrange equations for the Yang-Mills functional $$L_{YM}: A \longmapsto \int_M \|F_A\|^2\, ,$$ are given by $$d_A(*F_A) = 0.$$ Given a fixed Hermitian vector bundle $E$ of rank $r$ and degree $d$, the critical points of the Yang-Mills functional are the unitary connections satisfying
\begin{equation}\label{YM_eqn}
F_A = *i2\pi \left( \frac{d_1}{r_1} \mathrm{I}_{r_1} \oplus \cdots \oplus \frac{d_l}{r_l} \mathrm{I}_{r_l} \right)
\end{equation}
where $$\sum_{i=1}^l r_i = r\quad \mathrm{and} \quad \sum_{i=1}^l d_i =d$$ Here, one may additionally assume that the indices are so labelled that $$\frac{d_1}{r_1} > \cdots > \frac{d_l}{r_l}\cdot$$These critical points are called \textbf{Yang-Mills connections}. Among them, the absolute minima of the Yang-Mills functional are the connections satisfying $$F_A = *i2\pi \frac{d}{r} \Id_E\, .$$ The Morse strata of the Yang-Mills functional coincide with the strata defined by the Harder-Narasimhan types (\cite{AB,Daskalopoulos}).

If now $(M,\sigma)$ is a Klein surface and $(E,\tau)$ is a real (resp.\ quaternionic) Hermitian vector bundle of rank $r$ and degree $d$ on $(M,\sigma)$, then the critical sets of the Yang-Mills functional $L_{YM}$ (in fact, the whole Morse strata, see Subsection \ref{induced_stratification}) are invariant under the involution $A\longmapsto \ov{A}$ induced by $\tau$ on the space of all unitary connections on $E$, since $F_{\ov{A}} = \ov{F_A}$ for any connection and $\ov{F_A} = F_A$ if $A$ satisfies (\ref{YM_eqn}). More precisely, one has the relation $$L_{YM}(\ov{A}) = L_{YM}(A)$$ for any unitary connection $A$, which implies that the Yang-Mills flow (i.e.\,the gradient flow of the Yang-Mills functional) preserves the closed subspace $\Ctau \subset \cC$. Since Daskalopoulos has proved that this flow converges (\cite{Daskalopoulos}), we have that the limiting connection of a $\tau$-compatible connection is itself $\tau$-compatible. In particular, the Yang-Mills flow takes a unitary connection $A\in\Csst$ defining a semi-stable $\tau$-compatible holomorphic structure on $E$ to a unitary connection $A_{\infty}$ satisfying 
\begin{equation}
\left\{ 
\begin{array}{ccc} F_{A_{\infty}} & = & *i2\pi\frac{d}{r} \Id_E\\
\ov{A_{\infty}} & = & A_{\infty}
\end{array}
 \right.
 \end{equation}
More generally, if $A$ defines a $\tau$-compatible holomorphic structure of Harder-Narasi\-mhan type $$\mu= \left(\frac{d_1}{r_1},\cdots,\frac{d_1}{r_1},\cdots,\frac{d_l}{r_l},\cdots,\frac{d_l}{r_l}\right)$$ (see Subsection \ref{Shatz_stratif}), the Yang-Mills flow will take $A$ to $A_{\infty}$ satisfying $\ov{A_{\infty}}=A_{\infty}$ and equation $\eqref{YM_eqn}$. This motivates the following definition.
\begin{definition}[The Yang-Mills equations over Klein surfaces]
Let $(E,\tau)$ be a real or quaternionic Hermitian vector bundle over the Klein surface $(M,\si)$ and let $A$ be a unitary connection on $E$. We say that the Yang-Mills equations over the Klein surface $(M,\si)$ are given by the \textit{system} 
\begin{equation}\label{real_YM_eqn}
\left\{ 
\begin{array}{ccc} d_A(*F_A) & = & 0\\
\ov{A} & = & A
\end{array}
 \right.
 \end{equation}
In other words, \textit{the Yang-Mills connections over the Klein surface $(M,\sigma)$ are the Yang-Mills connections over $M$ which, additionally, are Galois-invariant}.
\end{definition}
\noindent This gives a good notion of Yang-Mills equations over $(M,\si)$ because, by Theorem \ref{Lag_quotient_as_moduli_space}, gauge orbits of \textit{minimal} solutions to (\ref{real_YM_eqn}) are in bijective correspondence with real (resp.\ quaternionic) $S$-equivalence classes of semi-stable $\tau$-compatible holomorphic structures on $(E,\tau)$, thus characterizing poly-stable objects in algebraic geometry in terms of solutions to a certain partial differential equation (in accordance with the approach initated by Donaldson in \cite{Don_NS}).

\section{Based gauge groups and representation varieties}
\label{sec:based_gauge_group}

In this section, we recall the definitions of based gauge groups
and representation varieties introduced in \cite{AB} (complex case) and 
\cite{BHH} (real and quaternionic cases). We also outline our
strategy to compute the mod 2  Poincar\'{e} series
of the classifying space of the real and quaternionic gauge groups. All the maps, bundles, sections, etc. are of class $C^\infty$. 

\subsection{Cell decomposition of Klein surfaces} \label{sec:CW}
Let $(M,\si)$ be a Klein surface of topological
type $(g,n,a)$. In this subsection, we recall
the cell decomposition of $(M,\si)$ introduced in \cite[Section 2]{BHH}.
Let $q: M\to M/\si$ be the projection to the quotient. We will first introduce
a cell decomposition of $M/\si$:
$$
M/\si =\bigcup_{\alpha\in S} e_\alpha \cup \bigcup_{\beta\in T} e_\beta.
$$
For each $\alpha\in S$, the preimage $q^{-1}(e_\alpha)$ is a single cell in $M$, and we still
call  it $e_\alpha$; for each $\beta\in T$, the preimage $q^{-1}(e_\beta)$ is a disjoint
union of two cells $e_\beta^+$ and $e_\beta^- =\si(e_\beta^+)$ in $M$. Then
we have a cell decomposition of $M$:
$$
M =\bigcup_{\alpha\in S} e_\alpha \cup \bigcup_{\beta\in T}(e_\beta^+ \cup e_\beta^-).
$$

\subsubsection{\bf Type 0: $n=0,\ a=1$} 
In this case, $M/\si$ is a non-orientable surface without boundary.
This case was considered in \cite{Ho}. There are two subcases:
\begin{enumerate}
\item[(i)] $g=2\hg$ is even. 
$M/\si$ is homeomorphic to  the connected sum of a Riemann surface of genus $\hg$ and the real projective plane $\RP^2$.
\item[(ii)] $g=2\hg+1$ is odd.
$M/\si$ is homeomorphic to  the connected sum of a Riemann surface of genus $\hg$ and a Klein bottle.
\end{enumerate}
We have a disjoint union of cells:
\begin{equation}\label{eqn:Ms-CW-zero}
M/\si = \begin{cases}
V\cup \bigcup_{i=1}^{\hg}\big(\ralpha_i \cup \rbeta_i \big) \cup \rgamma_0 \cup F, & g=2\hg,\\
V\cup \bigcup_{i=1}^{\hg}\big(\ralpha_i \cup \rbeta_i \big) \cup \rgamma_0 \cup \rdelta_0 \sqcup F, & g=2\hg+1.
\end{cases} 
\end{equation}
In the above cell decomposition:
\begin{enumerate}
\item $V=\{x\}$ is a 0-cell, where $x\in M/\si$. 
\item $\ralpha_i$, $\rbeta_i$, $\rdelta_0$, $\rgamma_0$ are 1-cells. 
Their closures $\alpha_i$, $\beta_i$, $\gamma_0$, $\delta_0$ are
loops in $M/\si$ passing through $x$. 
\item $F$ is a 2-cell such that its oriented boundary is given by
$$
\partial F  = \begin{cases}
\prod_{i=1}^{\hg} [\alpha_i,\beta_i]\, \gamma_0^2, & g=2\hg,\\
\prod_{i=1}^{\hg}[\alpha_i,\beta_i]\, \gamma_0 \delta_0 \gamma_0 \delta_0^{-1}, & g=2\hg+1
\end{cases} 
$$ 
\end{enumerate}

\noindent For each cell $e$ in the cell decomposition \eqref{eqn:Ms-CW-zero}, 
$q^{-1}(e)$ the disjoint union of two cells  $e^+$ and $e^- =\si(e^+)$ in $M$. 
We have a disjoint union of cells: 
\begin{equation}\label{eqn:M-CW-zero}
M =\begin{cases} 
\begin{array}{c} V^+ \cup V^- \cup \bigcup_{i=1}^{\hg} \big( \ralpha^+_i \cup \ralpha^-_i\cup \rbeta_i^+\cup \rbeta_i^-) \\
 \cup \rgamma_0^+ \cup \rgamma_0^- \cup F^+\cup F^-\end{array} & g=2\hg,\\
 & \\
\begin{array}{c} V^+ \cup V^- \cup \bigcup_{i=1}^{\hg} \big( \ralpha^+_i \cup \ralpha^-_i\cup \rbeta_i^+\cup \rbeta_i^-) \\
\cup \rgamma_0^+ \cup  \rgamma_0^-  \cup \rdelta_0^+ \cup \rdelta_0^-\cup F^+\cup F^-\end{array} & g=2\hg+1,
\end{cases}
\end{equation}
In the above cell decomposition,
\begin{enumerate}
\item $V^+=\{x_0\}$ and $V^-=\{\si(x_0)\}$, where $\{x_0, \si(x_0)\}=q^{-1}(x)$.
\item Let $\alpha_i^\pm, \beta_i^\pm, \gamma_0^\pm$, $\delta_0^\pm$  be the closures
of $\ralpha_i^\pm, \rbeta_i^\pm, \rgamma_0^\pm$, $\rdelta_0^\pm$, respectively. Then
$\alpha_i^+, \beta_i^+$ are loops in $M$ passing through $x_0$, and
$\alpha_i^-, \beta_i^-$ are loops in $M$ passing through $\si(x_0)$.
When $g=2\hg$ is even, $\gamma^+_0$ is a path in $M$ from 
$x_0$ to $\si(x_0)$, so $\gamma^-_0$ is a path in $M$ from
$\si(x_0)$ to $x_0$. When $g=2\hg+1$ is odd, 
$\gamma^+_0$ is a loop in $M$ passing through  $x_0$,
and $\delta^+_0$ is a path in $M$ from $x_0$ to $\si(x_0)$; 
so $\gamma^-_0$ is a loop in $M$ passing through $\si(x_0)$
and $\delta^-_0$ is a path in $M$ from $\si(x_0)$ to $x_0$.
\item The oriented boundary of $F^+$ is given by
$$
\partial F^+  = \begin{cases}
\prod_{i=1}^{\hg} [\alpha^+_i,\beta^+_i]\gamma_0^+ \gamma_0^-, & g=2\hg,\\
\prod_{i=1}^{\hg} [\alpha^+_i,\beta^+_i]\gamma_0^+ \delta_0^+ \gamma_0^- (\delta_0^+)^{-1}, & g=2\hg+1.
\end{cases}
$$
Therefore, 
$$
\partial F^- =  \begin{cases}
\prod_{i=1}^{\hg} [\alpha^-_i,\beta^-_i]\gamma_0^- \gamma_0^+, & g=2\hg,\\
\prod_{i=1}^{\hg} [\alpha^-_i,\beta^-_i]\gamma_0^- \delta_0^- \gamma_0^+ (\delta_0^-)^{-1}, & g=2\hg+1.
\end{cases}
$$ 
\end{enumerate}

\subsubsection{\bf Type I: $n>0,\ a=0$}
In this case, $M/\si$ is an orientable surface with boundary.
There is a non-negative integer $\hg$ such that $g=2\hg+n-1$, and 
$M/\si$ is homeomorphic to $\Si_{\hg,n}$, where
$\Si_{\hg,n}$ is obtained by removing $n$ disjoint open disks
from a Riemann surface of genus $\hg$. 
$\Ms = \partial(M/\si)$ is the disjoint union of $n$ circles:
$$
\Ms =\partial(M/\si) = \gamma_1\sqcup \cdots \sqcup \gamma_n.
$$
We have a disjoint union of cells:
\begin{equation}\label{eqn:Ms-CW-I}
M/\si = \bigcup_{i=1}^n V_i \cup \bigcup_{i=1}^{\hg}\big(\ralpha_i \cup \rbeta_i \big) 
\cup \bigcup_{j=1}^n \rgamma_j\cup  \bigcup_{j=2}^n \rdelta_j \cup F. 
\end{equation}
In the above cell decomposition:
\begin{enumerate}
\item Each  $V_i=\{x_i\}$ is a 0-cell, where $x_i\in \gamma_i$.
\item $\ralpha_i,\rbeta_i,\rgamma_i,\rdelta_i$ are 1-cells. $\gamma_j =\rgamma_j \cup \{ x_i\}$
is the closure of $\rgamma_i$.   
\item For $i=1,\ldots,\hg$, the closures $\alpha_i, \beta_i$ of $\ralpha_i, \rbeta_i$  are loops in $M/\si$ passing through $x_1$.
\item For $i=2,\ldots, n$, the closure $\delta_i$ of $\rdelta_i$ is a path from  $x_1$ to $x_i$.  
\item $F$ is a 2-cell such that its oriented boundary is given by
$$
\partial F  = \prod_{i=1}^{\hg} [\alpha_i,\beta_i]  \gamma_1 \prod_{i=2}^n \delta_i \gamma_i \delta_i^{-1}.
$$ 
\end{enumerate}

\noindent In the cell decomposition \eqref{eqn:Ms-CW-I}:
\begin{enumerate}
\item $q^{-1}(V_i)$  a single $0$-cell in $M$, 
and we still call it $V_i$; $q^{-1}(\rgamma_i)$ is a single $1$-cell in $M$, and we still call it $\rgamma_i$.
\item If $e$ is a cell in the decomposition \eqref{eqn:Ms-CW-I}, and $e$ is
neither $V_i$ nor $\rgamma_i$, then $q^{-1}(e)$  is the disjoint union
of two cells $e^+$ and $e^-=\si(e^+)$ in $M$. 
\end{enumerate}
We have a disjoint union of cells:
\begin{equation}\label{eqn:M-CW-I}
M =\bigcup_{i=1}^n V_i \cup \bigcup_{i=1}^{\hg}\big(\ralpha^+_i \cup \ralpha^-_i \cup \rbeta_i^+\cup \rbeta_i^- \big) \\
\cup \bigcup_{i=1}^n \rgamma_j  \cup \bigcup_{i=2}^n (\rdelta^+_j\cup \rdelta^-_j) \cup F^+\cup F^-. 
\end{equation}
In the above decomposition,
\begin{enumerate}
\item Let $\alpha_i^\pm, \beta_i^\pm, \delta_i^\pm$  be the closures
of $\ralpha_i^\pm, \rbeta_i^\pm, \rdelta_i^\pm$, respectively. Then
$\alpha_i^\pm, \beta_i^\pm$ are loops in $M$ passing through $x_1$, 
and $\delta^\pm_i$ is a path in $M$ from $x_1$ to $x_i$.
\item The oriented boundary of $F^+$ is given by
$$
\partial F^+  = 
\prod_{i=1}^{\hg} [\alpha^+_i,\beta^+_i]\gamma_1 \prod_{i=2}^n \big( \delta^+_i \gamma_i (\delta^+_i)^{-1}\bigr) 
$$
Therefore, 
$$
\partial F^-  = 
\prod_{i=1}^{\hg} [\alpha^-_i,\beta^-_i]\gamma_1 \prod_{i=2}^n \big( \delta^-_i \gamma_i (\delta^-_i)^{-1}\bigr) 
$$
\end{enumerate}

\subsubsection{\bf Type II: $n>0,\ a=1$} 
In this case, $M/\si$ is a nonorientable surface
with boundary, and $g-n\geq 0$. $\Ms =\partial(M/\si)$ is the disjoint
union of $n$ circles:
$$
\Ms =\partial(M/\si) = \gamma_1\sqcup \cdots \sqcup \gamma_n. 
$$
There are two subcases:
\begin{enumerate}
\item[(i)] $g-n=2\hg$ is even:  $M/\si$ is homeomorphic
to the connected sum of $\Si_{\hg,n}$ and the real projective
plane $\RP^2$.
\item[(ii)] $g-n=2\hg+1$ is odd: $M/\si$ is homeomophic
to the connected sum of $\Si_{\hg,n}$ and a Klein bottle. 
\end{enumerate}
We have a disjoint union of cells:
\begin{equation}\label{eqn:Ms-CW-II}
M/\si = \begin{cases}
\bigcup_{i=0}^n V_i \cup \bigcup_{i=1}^{\hg}\big(\ralpha_i \cup \rbeta_i \big) 
\cup \bigcup_{i=0}^n \rgamma_i  \cup \bigcup_{i=1}^n \rdelta_i \cup F, & g-n=2\hg,\\
\bigcup_{i=0}^n V_i \cup \bigcup_{i=1}^{\hg}\big(\ralpha_i \cup \rbeta_i \big) 
\cup \bigcup_{i=0}^n \rgamma_i \cup \bigcup_{i=0}^n \rdelta_i \bigcup F, & g-n=2\hg+1.
\end{cases} 
\end{equation}
In the above cell decomposition:
\begin{enumerate}
\item Each $V_i=\{x_i\}$ is a 0-cell, where $x_i\in \gamma_i$ for
$i=1,\ldots,n$, and $x_0$ is in the interior of $M/\si$.  
\item $\ralpha_i, \rbeta_i, \rgamma_i, \rdelta_i$ are 1-cells.
Let $\alpha_i,\beta_i,\delta_i,\gamma_i$ be the closures of
$\ralpha_i,\rbeta_i,\rgamma_i,\rdelta_i$, respectively.
For $i=1,\ldots,n$, $\gamma_i =\rgamma_i\cup\{ x_i\}$.
\item $\alpha_i, \beta_i, \gamma_0, \delta_0$  are loops in $M/\si$ passing through $x_0$. 
\item For $i=1,\ldots, n$, $\delta_i$ is a path from $x_0$ to $x_i$. 
\item $F$ is a 2-cell such that its oriented boundary is given by
$$
\partial F  = \begin{cases}
\prod_{i=1}^{\hg} [\alpha_i,\beta_i] \gamma_0^2\prod_{i=1}^n (\delta_i \gamma_i \delta_i^{-1}) & g-n=2\hg,\\
\prod_{i=1}^{\hg}[\alpha_i,\beta_i]\, \gamma_0 \delta_0 \gamma_0\delta_0^{-1} 
\prod_{i=1}^n (\delta_i\gamma_i\delta_i^{-1}), & g-n =2\hg+1
\end{cases} 
$$ 
\end{enumerate}

\noindent In the cell decomposition \eqref{eqn:Ms-CW-II}:
\begin{enumerate}
\item For $i=1,\ldots,n$, $q^{-1}(V_i)$ is a single $0$-cell in $M$, and we still call it $V_i$;
$q^{-1}(\rgamma_i)$ is a single $1$-cell in $M$, and we still call it $\rgamma_i$.
\item If $e$ is a cell in the decomposition \eqref{eqn:Ms-CW-II} and is 
not in $\{ V_1,\ldots,V_n, \rgamma_1,\ldots, \rgamma_n\}$, then $q^{-1}(e)$  
is the disjoint union of two cells $e^+$ and $e^-=\si(e^+)$ in $M$. 
\end{enumerate}
We have a disjoint union of cells:
\begin{equation}\label{eqn:M-CW-II}
M =\begin{cases} 
\begin{array}{c} V_0^+ \cup V_0^- \cup \bigcup_{i=1}^n V_i 
\cup \bigcup_{i=1}^{\hg} \big( \ralpha^+_i \cup \ralpha^-_i\cup \rbeta_i^+\cup \rbeta_i^-) \\
\cup \rgamma_0^+ \cup \rgamma_0^- \cup \bigcup_{i=1}^n \rgamma_i \cup\bigcup _{i=1}^n (\rdelta^+_i\cup \rdelta^-_i) 
\cup \bigcup F^+\cup F^-\end{array} & g-n=2\hg,\\
 & \\
\begin{array}{c}  V_0^+ \cup V_0^- \cup \bigcup_{i=1}^n V_i 
\cup \bigcup_{i=1}^{\hg} \big( \ralpha^+_i \cup \ralpha^-_i\cup \rbeta_i^+\cup \rbeta_i^-) \\
\cup \rgamma_0^+ \cup \rgamma_0^- \cup \bigcup_{i=1}^n \rgamma_i \cup\bigcup _{i=0}^n (\rdelta^+_i\cup \rdelta^-_i) 
\cup \bigcup F^+\cup F^-\end{array} & g-n=2\hg+1
\end{cases}
\end{equation}
In the above cell decomposition,
\begin{enumerate}
\item $V_0^+=\{x_0\}$ and $V_0^-=\{\si(x_0)\}$, where $\{ x_0, \si(x_0)\}= q^{-1}(x)$.
\item Let $\alpha_i^\pm, \beta_i^\pm, \gamma_0^\pm, \delta_i^\pm$  be the closures
of $\ralpha_i^\pm, \rbeta_i^\pm, \rgamma_0^\pm, \rdelta_i^\pm$, respectively. Then
$\alpha_i^+, \beta_i^+$ are loops in $M$ passing through $x_0$, and
$\alpha_i^-, \beta_i^-$ are loops in $M$ passing through $\si(x_0)$.
$\delta_i^+$ is a path in $M$ from $x_0$ to $x_i$, 
and $\delta_i^-$ is a path in $M$ from $\si(x_0)$ to $x_i$.

\item When $g-n=2\hg$ is even, $\gamma^+_0$ is a path in $M$ from 
$x_0$ to $\si(x_0)$, so $\gamma^-_0$ is a path in $M$ from
$\si(x_0)$ to $x_0$. When $g=2\hg+1$ is odd, 
$\gamma^+_0$ is a loop in $M$ passing through  $x_0$,
and $\delta^+_0$ is a path in $M$ from $x_0$ to $\si(x_0)$; 
so $\gamma^-_0$ is a loop in $M$ passing through $\si(x_0)$
and $\delta^-_0$ is a path in $M$ from $\si(x_0)$ to $x_0$.
\item The oriented boundary of $F^+$ is given by
$$
\partial F^+  = \begin{cases}
\prod_{i=1}^{\hg} [\alpha^+_i,\beta^+_i]\gamma_0^+ \gamma_0^- \prod_{i=1}^n \bigl(\delta_i^+ \gamma_i (\delta_i^+)^{-1}\bigr) , & g-n=2\hg,\\
\prod_{i=1}^{\hg} [\alpha^+_i,\beta^+_i]\gamma_0^+ \delta_0^+ \gamma_0^- (\delta_0^+)^{-1}\
\prod_{i=1}^n \bigl(\delta_i^+ \gamma_i (\delta_i^+)^{-1}\bigr), & g-n=2\hg+1.
\end{cases}
$$
Therefore, 
$$
\partial F^-  = \begin{cases}
\prod_{i=1}^{\hg} [\alpha^-_i,\beta^-_i]\gamma_0^- \gamma_0^+ \prod_{i=1}^n \bigl(\delta_i^- \gamma_i (\delta_i^-)^{-1}\bigr) , & g-n=2\hg,\\
\prod_{i=1}^{\hg} [\alpha^-_i,\beta^-_i]\gamma_0^- \delta_0^- \gamma_0^+ (\delta_0^-)^{-1}\
\prod_{i=1}^n \bigl(\delta_i^- \gamma_i (\delta_i^-)^{-1}\bigr), & g-n=2\hg+1.
\end{cases}
$$
\end{enumerate}

\subsection{The evaluation map and based gauge groups} \label{sec:ev-based}
Let $(E,\tau)\to (M,\si)$ be a real or quaternionic
Hermitian vector bundle of rank $r$, degree $d$ over a Klein surface 
$(M,\si)$ of topological type $(g,n,a)$.

Let $P_E\to M$ be the unitary frame bundle of the Hermitian
vector bundle $E$. Then $P_E$ is a principal $\U(r)$-bundle
over $M$. The structure group $\U(r)$ acts freely on $P_E$ on the
right, and $M=P_E/\U(r)$. Let $\pi:P_E\to M=P_E/\U(r)$ be
the natural projection.  The gauge group $\cG_E$ can be identified
with the space of $\U(r)$-equivariant maps $P_E\to \U(r)$, where
$\U(r)$ acts on itself by conjugation:  
$$
\cG_E=\{ u: P_E \to \U(r)\mid u(p\cdot h) = h^{-1} u(p) h\textup{ for any } 
p\in P_E, h\in \U(r)\}. 
$$
For any $p\in P_E$, there is an evaluation map 
$$
\ev_p:\cG_E \lra \U(r), \quad u\mapsto u(p).
$$
The evaluation map $\ev_p$ is a surjective group homomorphism.
Note that the kernel of $\ev_p$ depends only on $\pi(p)\in M$.
We define the based gauge group $\cG_E(x)$ to 
be the kernel of $\ev_p$, where $p$ is any point in
$\pi^{-1}(x)$.  Then $\cG_E(x)$ is a normal subgroup of $\cG_E$, 
and there is a short exact sequence of groups
\begin{equation} \label{eqn:based-C}
1 \to \cG_E(x) \lra \cG_E \stackrel{\ev_p}{\lra} \U(r)\to 1.
\end{equation}

Given $p=(x,e_1,\ldots,e_r)\in P_E$, where
$x\in M$ and $(e_1,\ldots, e_r)$ is a unitary frame of $E_x$, 
$(\tau(e_1),\ldots, \tau(e_r))$ is
a unitary frame of $E_{\si(x)}$. We define
$$
\tau:P_E\to P_E,\quad (x,e_1,\ldots,e_r)\mapsto (\si(x), \tau(e_1),\ldots, \tau(e_r)).
$$
Then $\tau:P_E\to P_E$ satisfies the following properties:
\begin{enumerate}
\item The diagram
$$
\begin{CD}
P_E @>{\tau}>> P_E\\
@V{\pi}VV  @V{\pi}VV\\
M @>{\si}>> M
\end{CD}
$$
is a commutative diagram,
\item  For all $p\in P_E$ and $h\in \U(r)$,
$$
\tau(p\cdot h) = \tau(p)\cdot \bar{h}.
$$
\item For all $p\in P_E$,  
$$
\tau\circ \tau (p)=\begin{cases}
p, & \tau=\tauR,\\
p\cdot (-I_r), &\tau=\tauH,
\end{cases}
$$
where $I_r\in \U(r)$ is the $r\times r$ identity matrix.
\end{enumerate}

The involution $\tau:\cG_E\to \cG_E$ can be described in terms of $\tau:P_E\to P_E$, as follows. Given 
$$
u\in \cG_E =\{ u:P_E\to \U(r)\mid u(p\cdot h) = h^{-1} u(p) h\textup{ for any } p\in P_E, h\in \U(r)\},
$$
define 
$$
\tau(u):P_E\to \U(r),\quad \tau(u)(p) = \overline{u(\tau(p))}.
$$
It is straighforward to check that $\tau(u)\in \cG_E$ and $\tau(\tau(u))=u$. The involution
$\tau:\cG_E\to \cG_E$ is given by $u\mapsto \tau(u)$. The fixed-point set
$\cG_E^{\, \tau}$ is the real or quaternionic gauge group.

For any $p\in P_E$, there is a commutative diagram
$$
\begin{CD}
\cG_E @>{\tau}>> \cG_E\\
@V{\ev_p}VV  @V{\ev_{\tau(p)}}VV\\
\U(r) @>{\tauR}>> \U(r)
\end{CD}
$$
where $\tauR:\U(r)\to \U(r)$ is defined by $\tauR(A)=\bar{A}$.
Let
$$
J=\left(\begin{array}{cc} 0 & 1\\ -1&0 \end{array} \right) \in SO(2),
$$
and define 
$$
J_m =\mathrm{diag}\Bigl(\underbrace{J,\ldots, J}_{m\textup{ copies }} \Bigr) \in SO(2m).
$$
When $r$ is even, we define an involution
$$
\tauH: \U(r)\to \U(r),\quad A\mapsto -J_{r/2}\bar{A} J_{r/2}.
$$
The fixed points set of $\tau_\bH$ is $\U(r)^{\tau_\bH}=\Sp(\frac{r}{2})$.

Given $x\in M$, there are two cases:
\begin{enumerate}
\item If $\si(x)\neq x$, then for any $p\in \pi^{-1}(x)$, the map
$$
\ev_p \times \ev_{\tau(p)}:\cG_E\to \U(r)\times \U(r), \quad u\mapsto (u(p), u(\tau(p)))
$$
is surjective, and
$$
(\ev_p\times \ev_{\tau(p)})(\cG_E^{\, \tau}) =\{ (A,\bar{A})\mid A\in \U(r) \} \subset \U(r)\times \U(r). 
$$
So we have a surjective map 
$$
\ev_p:\cG_E^{\, \tau}\to \U(r), \quad u\mapsto u(p).
$$

\item If $\si(x)=x$, then $\tau:P_E\to P_E$ restricts to 
$\tau:\pi^{-1}(x)\to \pi^{-1}(x)$. We may choose $p\in \pi^{-1}(x)$ such that
$$
\tau(p)  =\begin{cases}
p, & \textup{if } \tau=\tauR,\\
p\cdot J_{r/2}, & \textup{if }r \textup{ is even and } \tau=\tauH.
\end{cases}
$$
Then 
$$
u(\tauR (p))= u(p),\quad u(\tauH(p)) = J_{r/2}^{-1} u(p) J_{r/2} = -J_{r/2} u(p) J_{r/2}. 
$$
So
\begin{eqnarray*}
(\ev_p\times \ev_{\tau(p)})(\cG_E) &=& \{ (A,\overline{\tau(A)})\mid  A\in \U(r)\}\\
(\ev_p\times \ev_{\tau(p)})(\cG_E^{\, \tau}) &=& \{ (A,\bar{A})\mid A\in \U(r)^\tau\}.
\end{eqnarray*}
So we have a surjective map
$$
\ev_p: \cG_E^{\, \tau}\to \U(r)^\tau,\quad u\mapsto u(p).  
$$
\end{enumerate}

Let $(M,\si)$ be a Klein surface of topological type $(g,n,a)$, and let
$$
\vx=\begin{cases}
(x_1,\ldots,x_n), & \textup{if $a=0$ and $n>0$},\\
(x_0, x_1,\ldots, x_n), & \textup{if $a=1$},
\end{cases}
$$
where $x_i$ is chosen as in Section \ref{sec:CW}. 
If $a=1$, choose $p_0\in \pi^{-1}(x_0)$. 
For $i=1,\ldots,n$, choose $p_i\in \pi^{-1}(x_i)$ such that
\begin{equation}\label{eqn:pi}
\tau(p_i)=\begin{cases}
p_i, & \textup{if } \tau=\tauR,\\
p_i\cdot J_{r/2}, & \textup{if }r \textup{ is even and }\tau=\tauH.
\end{cases}
\end{equation}
Let
$$
\vp=\begin{cases}
(p_1,\ldots,p_n), & \textup{if $a=0$ and $n>0$},\\
(p_0, p_1,\ldots, p_n), & \textup{if $a=1$}.
\end{cases}
$$

\noindent We define
$$
\Gt:= \U(r)^a \times (\U(r)^\tau)^n,
$$
and define an evaluation map $\ev_{\vp}:\cG_E^{\, \tau}\to \Gt$ by
$$
\ev_{\vp}(u)=\begin{cases}
(u(p_1),\ldots, u(p_n)), & \textup{if $a=0$ and $n>0$},\\
(u(p_0), u(p_1),\ldots, u(p_n), & \textup{if $a=1$}.
\end{cases}
$$
where $u(p_0)\in \U(r)$ and $u(p_1),\ldots, u(p_n)\in \U(r)^\tau$. 
Then $\ev_{\vp}:\cG_E^{\, \tau}\to \Gt$ is a surjective group homomorphism.
The kernel of $\ev_{\vp}$ depends only on $\vx \in M^{a+n}$. We define the based gauge
group $\cG_E^{\, \tau}(\vx)$ to be the kernel of $\ev_{\vp}$, where
$\vp=(p_1,\dots,p_n)$ or $(p_0, p_1,\ldots, p_n)$, and 
$p_i$ is chosen to satisfy \eqref{eqn:pi} for $i>0$. 
Then $\cG_E^{\, \tau}(\vx)$ is a normal subgroup of $\cG_E^{\, \tau}$, and there is a short exact sequence 
of groups
\begin{equation}\label{eqn:based-RH}
1\to \cG_E^{\, \tau}(\vx)\lra \cG_E^{\, \tau} \stackrel{\ev_{\vp}}{\lra} \Gt \to 1. 
\end{equation}

\subsection{The complex holonomy map}
A unitary connection $A$ on $E$ can be viewed as a connection
on the principal bundle $\pi: P_E\to M$. Given any path $\gamma:[0,1]\to M$,
let $x_0=\gamma(0), x_1=\gamma(1) \in M$, and let
$$
P_\gamma(A):\pi^{-1}(x_0)\to \pi^{-1}(x_1)
$$
be the parallel transport defined by $A$. Then $P_\gamma(A)$ is
$\U(r)$-equivariant, i.e.,
$$
P_\gamma(A)(p\cdot h) = P_\gamma(A)(p)\cdot h
$$
for any $p\in \pi^{-1}(x_0)$ and $h\in \U(r)$. Given $p_0\in \pi^{-1}(x_0)$ and
$p_1\in \pi^{-1}(x_1)$, let $P_{\gamma,p_0,p_1}(A)\in \U(r)$ be characterized by
$$
P_\gamma(A)(p_0) = p_1 \cdot P_{\gamma,p_0,p_1}(A)^{-1}. 
$$
This gives a map
$$
P_{\gamma,p_0,p_1}: \cC\to \U(r)
$$
for each path $\gamma$ in $M$ and reference points $p_0\in \pi^{-1}(x_0)$, $p_1\in \pi^{-1}(x_1)$.
These maps satisfy the following two properties:
\begin{enumerate} 
\item (dependence on reference points)  For any $A\in \cC$,  $p_0\in \pi^{-1}(x_0)$, $p_1\in \pi^{-1}(x_1)$, $h_0, h_1\in \U(r)$,
\begin{equation} \label{eqn:change-reference}
P_{\gamma,p_0\cdot h_0, p_1\cdot h_1}(A) = h_0^{-1} P_{\gamma,p_0,p_1}(A) h_1 
\end{equation} 

\item (composition of paths) If $\gamma_1$ is a path from $x_0$ to $x_1$, $\gamma_2$ is a path from 
$x_1$ to $x_2$, and $p_i\in \pi^{-1}(x_i)$ for $i=0,1,2$, then
\begin{equation} \label{eqn:compose-path} 
P_{\gamma_1 \cdot \gamma_2, p_0, p_2}(A) = P_{\gamma_1,p_0,p_1}(A) P_{\gamma_2,p_1,p_2}(A)
\end{equation}
for any $A\in \cC$. 
\end{enumerate}

\noindent Suppose that $\gamma(0)=\gamma(1)$, so that $\gamma$ is a loop. 
We define  
$$
P_{\gamma,p_0} :=  P_{\gamma, p_0, p_0}: \cC\to \U(r).
$$
Then the maps $P_{\gamma,p_0}$ satisfy the following properties.
\begin{enumerate} 
\item (dependence on reference points) If $\gamma$ is a based loop in $(M,x)$, then  
\begin{equation}\label{eqn:change-refer}
P_{\gamma,p\cdot h}(A) = h^{-1} P_{\gamma,p}(A) h 
\end{equation} 
for $p\in \pi^{-1}(x)$, $h\in \U(r)$, and $A\in \cC$. 
\item (composition of loops) If $\gamma_1, \gamma_2$ are loops passing through $x$, and $p\in \pi^{-1}(x)$, then
\begin{equation} \label{eqn:compose-loop}
P_{\gamma_1 \cdot \gamma_2, p}(A) = P_{\gamma_1,p}(A) P_{\gamma_2,p}(A)
\end{equation}
for any $A\in \cC$. 
\end{enumerate}
Then
$$
P_{\gamma,p\cdot h}(A) = h^{-1} P_{\gamma,p}(A) h 
$$
for any $A\in \cC$, $p\in \pi^{-1}(\gamma(0))$, $h\in \U(r)$.
We call $P_{\gamma,p_0}(A)$ the {\bf holonomy of the connection $A$ along 
the loop $\gamma$ with respect to the reference point $p_0\in \pi^{-1}(\gamma(0))$}. 

There is a disjoint union of cells:
\begin{equation}\label{eqn:M-CW}
M= V\cup \bigcup_{i=1}^g (\ralpha_i \cup \rbeta_i)\cup F.
\end{equation}
In the above cell decomposition:
\begin{enumerate}
\item $V=\{x\}$ for some $x\in M$.
\item $\ralpha_i$, $\rbeta_i$ are 1-cells. The closures
$\alpha_i, \beta_i$ of $\ralpha_i, \rbeta_i$ are loops in $M$ passing through $x$.
\item $F$ is a 2-cell, and the oriented boundary of $F$ is given by
$$
\partial F =\prod_{i=1}^g [\alpha_i,\beta_i].
$$
\end{enumerate}
We choose $p\in \pi^{-1}(x)$. Then there is a surjective map 
$$
\tHol: \cC \lra \U(r)^{2g},
\quad A\mapsto (P_{\alpha_1,p}(A), P_{\beta_1,p}(A),\ldots, P_{\alpha_g,p}(A), P_{\beta_g,p}(A))
$$
which descends to a surjective map
\begin{equation}\label{eqn:C-hol}
\Hol: \cC/ \cG_E(x) \to \U(r)^{2g}.
\end{equation}
We call $\Hol$ the {\bf holonomy map defined
by the based loops $\alpha_i,\beta_i$ in $(M,x)$ and the reference point $p\in \pi^{-1}(x)$}.
The based gauge group $\cG_E(x)$ acts freely
on the contractible space $\cC$, so the 
quotient $\cC/ \cG_E(x)$ is 
a classifying space of $\cG_E(x)$. 
The action of $\cG_E$ on $\cC$ induces an
action of $\U(r)= \cG_E / \cG_E(x)$
on $B\big(\cG_E(x)\big) = \cC/ \cG_E(x)$. By \eqref{eqn:change-refer},
the holonomy map \eqref{eqn:C-hol} is $\U(r)$-equivariant with respect
to this $\U(r)$-action on $\cC/ \cG_E(x)$
and the following  $\U(r)$-action on $\U(r)^{2g}$:
\begin{equation}\label{eqn:conj-C}
(a_1,b_1,\ldots, a_g, b_g)\cdot u 
= (u^{-1} a_1 u , u^{-1} b_1 u,\ldots, u^{-1}  a_g u , u^{-1} b_g u),
\end{equation}
where $u, a_i, b_i\in \U(r)$.

There is a commutative diagram
\begin{equation}\label{eqn:nine-C}
\begin{CD}
B\big(\Omega^2(\U(r)) \big) @>>> B\big(\cG_E(x)\big) @>{\Hol}>> \U(r)^{2g} \\
@VVV  @VVV  @VVV \\
B\big(\Omega^2(\U(r))\big) @>>> B(\cG_E) @>>> E\U(r) \times_{\U(r)} \U(r)^{2g} \\ 
@VVV  @VVV  @VVV \\
\{ \mathrm{point} \} @>>> B\U(r) @>>> B\U(r)
\end{CD}
\end{equation}
In the above diagram:
\begin{enumerate}
\item All the rows and columns are fibrations.
\item The maps 
$$ 
B\big(\Omega^2(\U(r))\big) \to  B\big(\Omega^2(\U(r))\big),\quad B\U(r)\lra B\U(r)
$$
are the identity maps. 
\item The central column of the diagram \eqref{eqn:nine-C} comes from the short exact sequence \eqref{eqn:based-C}.  
\item $E\U(r)\times_{\U(r)} (\U(r)^{2g})$  is the quotient of $E\U(r)\times (\U(r)^{2g})$ by
the following free action of $\U(r)$:
$$
(y,z )\cdot h = (y\cdot h, z\cdot h)
$$
where $y\in E\U(r)$, $z \in \U(r)^{2g}$, and $z\cdot h$ is given by \eqref{eqn:conj-C}.
Therefore $E\U(r)\times_{\U(r)} (\U(r)^{2g})$ is the homotopy orbit space
of the following {\em left} $\U(r)$-action on $\U(r)^{2g}$:
$$
h\cdot z := z\cdot h^{-1},\quad h\in \U(r),\quad z\in \U(r)^{2g}.
$$

\item $B\big( \Omega^2(\U(r)) \bigr)$ can be identified
with the classifying space of the based gauge group of $S^2$, and is homotopy equivalent to 
$\Omega\U(r)_0$, the connected component of the identity of
the based loop space $\Omega\U(r)$ of  $\U(r)$. 
\end{enumerate}

Atiyah and Bott showed that the fibrations in the central column and in the top row
of \eqref{eqn:nine-C} are homologically trivial over $\bQ$.   So all
the rows and columns of \eqref{eqn:nine-C} are homologically
trivial over $\bQ$. Therefore, 
$$
P_t(B(\cG_E);\bQ) = P_t(\Omega\U(r)_0;\bQ) P_t(\U(r)^{2g};\bQ) P_t(B\U(r);\bQ) 
$$
where $P_t(\U(r)^{2g};\bQ) = P_t(\U(r);\bQ)^{2g}$.

\subsection{Real and quaternionic holomomy maps}

To compute the mod 2 cohomology of the real and quaternionic gauge group,
we consider the following commutative diagram:
\begin{equation}\label{eqn:nine-RH}
\begin{CD}
B\big(\Omega^2(\U(r)) \big) @>>> B\big(\cG_E^{\, \tau}(\vx)\big) @>{\Hol}>> \Wt\\
@VVV  @VVV  @VVV \\
B\big(\Omega^2(\U(r))\big) @>>> B(\cG_E^{\, \tau}) @>>> \EGt \times_{\Gt} \Wt \\ 
@VVV  @VVV  @VVV \\
\{ \mathrm{point} \} @>>> \BGt  @>>> \BGt \big)
\end{CD}
\end{equation}
In the above diagram:
\begin{enumerate}
\item All the rows and columns are fibrations.
\item The maps 
$$
B\big( \Omega^2(\U(r)) \big) \to B\big( \Omega^2(B\U(r)) \big),\quad \BGt\to \BGt
$$
are the identity maps. 
\item The central column of the diagram \eqref{eqn:nine-RH} comes from the short exact sequence \eqref{eqn:based-RH}.  
\item We will define the holonomy space $\Wt$ in Definition \ref{df:W}.
\item $\EGt\times_{\Gt}\Wt$ is the quotient of the $\EGt \times \Wt$ by the 
free $\Gt$-action
$$
(y,z)\cdot g = (y\cdot g, z\cdot g) = (y\cdot g, g^{-1}\cdot z)
$$
where 
$$
y\in \EGt,\ z\in \Wt,\ g\in \Gt,
$$
and $z\cdot g = g^{-1}\cdot z$ will be defined in Definition \ref{df:action-on-W}. 
Therefore, $$\EGt\times_{\Gt}\Wt$$ is the homotopy orbit space
of the left $\Gt$-action on $\Wt$ defined in Definition \ref{df:action-on-W}.

\item We will define the holonomy map 
$$
\Hol:  B\big(\cG_E^{\, \tau}(\vx)\big) \lra \Wt
$$ 
in Definition \ref{df:hol-map}. It is a $\Gt$-equivariant map.
\end{enumerate}

\noindent In Section \ref{topology_of_groups}, 
we will compute  $Q_r^{\tau}(g,n,a)= P_t(B(\cG_E^{\, \tau});\bZ_2)$ by studying the cohomology 
Leray-Serre spectral sequences
associated to the fibrations of the right column and the central row of \eqref{eqn:nine-RH}.
These spectral sequences do not collapse at the $E_2$-term in general.

Let $(E,\tau)$ be a real or quaternionic Hermitian vector bundle
of rank $r$ and  degree $d$ over a Klein surface $(M,\si)$ of topological type $(g,n,a)$.

Define
$$
\O(r)_{\pm 1} = \{ A\in \O(r)\mid \det(A)= \pm1\}.
$$
Then $\O(r)_{+1} =\SO(r)$ and $\O(r)_{-1}$ are the two connected components of $\O(r)$.

\begin{definition}[Holonomy spaces] \label{df:W}
If $n=0$ or $\tau=\tauH$, define
$$ 
\Wt:= \U(r)^{g+a} \times (\U(r)^\tau)^n. 
$$
If $n>0$ and $\tau=\tauR$, define $w^{(1)},\ldots, w^{(n)} \in \bZ/2\bZ$  as in 
Theorem \ref{top_type_of_bundles} (2). Define
$$
\Wt = \U(r)^{g+a} \times \prod_{i=1}^n \O(r)_{(-1)^{w^{(i)}}}
$$
\end{definition} 

\begin{definition}[Group action on holonomy spaces] \label{df:action-on-W} 
~
\begin{enumerate}
\item  Suppose that $a=1$, $n\geq 0$,  and $g-n=2\hg$ is even.
Given
$$
(a_1,b_1,\ldots, a_{\hg}, b_{\hg}, d_1,\ldots, d_n, c_0,c_1,\ldots, c_n) \in \Wt,
$$
where $a_i, b_i, d_j, c_0\in\U(r)$ and $c_1,\ldots, c_n\in \U(r)^\tau$,  and
$$
(h_0,h_1,\ldots, h_n) \in \Gt.
$$
where $h_0\in \U(r)$, $h_1,\ldots, h_r\in \U(r)^\tau$, the
group action is given by
\begin{eqnarray*}
&& (a_i, b_i, d_j, c_0,  c_j) \cdot (h_0, h_1,\ldots, h_n) \\
&=& (h_0^{-1} a_i h_0, h_0^{-1} b_i h_0, h_0^{-1} d_j h_j, h_0^{-1} c_0 \bar{h}_0, h_j^{-1} c_j h_j),
\end{eqnarray*}
where $i=1,\ldots,\hg$, $j=1,\ldots,n$.

\item  Suppose that $a=1$, $n\geq 0$,  and $g-n=2\hg+1$ is odd.
Given
$$
(a_1,b_1,\ldots, a_{\hg}, b_{\hg}, d_0, d_1,\ldots, d_n, c_0,c_1,\ldots, c_n) \in \Wt,
$$
where $a_i, b_i, d_0, d_j, c_0\in\U(r)$ and $c_1,\ldots, c_n\in \U(r)^\tau$,  and
$$
(h_0,h_1,\ldots, h_n) \in \Gt.
$$
where $h_0\in \U(r)$, $h_1,\ldots, h_r\in \U(r)^\tau$, the group action is given by
\begin{eqnarray*}
&& (a_i, b_i, d_0, d_j, c_0,  c_j) \cdot (h_0, h_1,\ldots, h_n) \\
&=& (h_0^{-1} a_i h_0, h_0^{-1} b_i h_0, h_0^{-1} d_0 \bar{h}_0,  h_0^{-1} d_j h_j, h_0^{-1} c_0 h_0, h_j^{-1} c_j h_j),
\end{eqnarray*}
where $i=1,\ldots,\hg$, $j=1,\ldots,n$.

\item Suppose that $a=0$, $n>0$, so that $g=2\hg +n-1$ for some non-negative integer $\hg$.
Given
$$
(a_1,b_1,\ldots, a_{\hg}, b_{\hg}, d_2,\ldots, d_n, c_1,\ldots, c_n) \in \Wt,
$$
where $a_i, b_i, d_j,\in\U(r)$ and $c_1,\ldots, c_n\in \U(r)^\tau$,  and
$$
(h_1,\ldots, h_n) \in \Gt.
$$
where $h_1,\ldots, h_n\in \U(r)^\tau$, the group action is given by
\begin{eqnarray*}
&& (a_i, b_i, d_j, c_1,  c_j) \cdot (h_1,\ldots, h_n) \\
&=& (h_1^{-1} a_i h_1, h_1^{-1} b_i h_1, h_1^{-1} d_j h_j, h_1^{-1} c_1 h_1, h_j^{-1} c_j h_j),
\end{eqnarray*}
where $i=1,\ldots,\hg$, $j=2,\ldots,n$.
\end{enumerate} 
The above are right actions. We define left action of
$\Gt$ on $\Wt$ by 
$$
g\cdot z := z  \cdot g^{-1}, \quad g\in \Gt, z\in \Wt.
$$
\end{definition} 

\begin{definition}[Holonomy maps] \label{df:hol-map} 
~
\begin{enumerate}
\item  Suppose that $a=1$, $n\geq 0$,  and $g-n=2\hg$ is even.
Define 
$$
\tHol:\cC^\tau\to \Wt
$$ 
by 
\begin{eqnarray*}
A &\mapsto&  (P_{\alpha_1,p_0}(A), P_{\beta_1,p_0}(A),\ldots, P_{\alpha_{\hg},p_0}(A), P_{\beta_{\hg}, p_0}(A), \\
&& P_{\delta_1,p_0, p_1}(A),\ldots, P_{\delta_n,p_0,p_n}(A), \\
&& P_{\gamma_0, p_0,\tau(p_0)}(A), P_{\gamma_1,p_1}(A), \ldots, P_{\gamma_n,p_n}(A)) 
\end{eqnarray*}

\item  Suppose that $a=1$, $n\geq 0$,  and $g-n=2\hg+1$ is odd.
Define 
$$
\tHol:\cC^\tau\to \Wt
$$ 
by 
\begin{eqnarray*}
A &\mapsto& (P_{\alpha_1,p_0}(A), P_{\beta_1,p_0}(A),\ldots, P_{\alpha_{\hg},p_0}(A), P_{\beta_{\hg}, p_0}(A), \\
&& P_{\delta_0, p_0,\tau(p_0)}(A), P_{\delta_1,p_0, p_1}(A),\ldots, P_{\delta_n,p_0,p_n}(A), \\
&& P_{\gamma_0, p_0}(A), P_{\gamma_1,p_1}(A), \ldots, P_{\gamma_n,p_n}(A)) 
\end{eqnarray*}

\item Suppose that $a=0$, $n>0$, so that $g=2\hg +n-1$ for some non-negative integer $\hg$.
\begin{eqnarray*}
A &\mapsto& (P_{\alpha_1,p_1}(A), P_{\beta_1,p_1}(A),\ldots, P_{\alpha_{\hg},p_1}(A), P_{\beta_{\hg}, p_1}(A), \\
&& P_{\delta_2,p_1, p_2}(A),\ldots, P_{\delta_n,p_1,p_n}(A),  P_{\gamma_1,p_1}(A), \ldots, P_{\gamma_n,p_n}(A)) 
\end{eqnarray*}
\end{enumerate}
In all the above three cases, the map $\tHol:\cC^\tau\to \Wt$ descends to 
a surjective map
$$
\Hol: \cC^\tau/ \cG_E^{\, \tau}(\vx) \lra \Wt.
$$
\end{definition}

\begin{lemma}[Equivariance of the holonomy map]
The $\cG_E^{\, \tau}$-action on $\cC^\tau$ induces an action
of $\Gt=\cG_E^{\, \tau}/\cG_E^{\, \tau}(\vx)$ on $\cG^\tau/\cG^\tau(\vx)$.
The holonomy map
$$
\Hol:\cC^\tau/\cG_E^{\, \tau}(\vx) \lra \Wt
$$
is $\Gt$-equivariant with respect to this action on $\cC^\tau/\cG_E^{\, \tau}(\vx)$ and
the right action on $\Wt$ defined in Definition \ref{df:action-on-W}.
\end{lemma}
\begin{proof} This follows from the construction and \eqref{eqn:change-reference}.
\end{proof}

\subsection{Representation varieties}\label{representation_varieties}
Let $E$ be a Hermitian vector bundle of rank $r$ and degree $d$ over 
a Riemann surface $X(\C)$ of genus $g\geq 2$.  Define
$$
\fm: \U(r)^{2g}\to \U(r),\quad \fm(a_1,b_1,\ldots, a_g, b_g) =\prod_{i=1}^g[a_i, b_i]. 
$$

\noindent The following is a consequence of the results in \cite{AB,Daskalopoulos}:
\begin{theorem}
The holonomy map
$$
\Hol: \cC/\cG_E(x) \to \U(r)^{2g}
$$
induces a $\U(r)$-equivariant homeomorphism
$$
\Hol:\YMconn/\cG_E(x) \to  V_g(r,d):=\fm^{-1}\big(\exp(-2i \pi \frac{d}{r}I_r)\big) 
$$
\end{theorem} 
We call $V_g(r,d)$ the representation variety of central
Yang-Mills connections on a rank $r$, degree $d$ Hermitian vector bundle
over a Riemann surface of genus $g$.  

\begin{corollary}
We have the following isomorphism of topological stacks
$$
[\YMconn /\cG_E] \simeq [V_g(r,d)/\U(r)]
$$
and the following homotopy equivalences of homotopy orbit spaces
$$
(\Css)_{h\cG_\C} \simeq (\Css)_{h\cG_E} \simeq (\YMconn)_{h\cG_E}.
$$
Therefore, 
$$
P_g(r,d) = P_t^{\U(r)}(V_g(r,d);\bQ). 
$$
\end{corollary}

Let $(E,\tau)$ be a real or quaternionic Hermitian vector bundle over a Klein surface
$(M,\si)$ of topological type $(g,n,a)$. We define
$$
\fm: \Wt \to \U(r)
$$
as follows. Let
$$
\varepsilon_\tau =\begin{cases}
1, & \tau=\tauR,\\
-1,& \tau=\tauH.
\end{cases}
$$

\begin{enumerate}
\item  Suppose that $a=1$, $n\geq 0$,  and $g-n=2\hg$ is even.
Given
$$
\xi= (a_1,b_1,\ldots, a_{\hg}, b_{\hg}, d_1,\ldots, d_n, c_0,c_1,\ldots, c_n) \in \Wt,
$$
where $a_i, b_i, d_j, c_0\in\U(r)$ and $c_1,\ldots, c_n\in \U(r)^\tau$,  define
$$
\fm(\xi) = \prod_{i=1}^{\hg} [a_i,b_i] c_0 (\varepsilon_\tau \bar{c}_0)  \prod_{j=1}^n (d_j c_j d_j^{-1}).
$$

\item  Suppose that $a=1$, $n\geq 0$,  and $g-n=2\hg+1$ is odd.
Given 
$$
\xi =(a_1,b_1,\ldots, a_{\hg}, b_{\hg}, d_0, d_1,\ldots, d_n, c_0,c_1,\ldots, c_n) \in \Wt,
$$
where $a_i, b_i, d_0, d_j, c_0\in\U(r)$ and $c_1,\ldots, c_n\in \U(r)^\tau$, define
$$
\fm(\xi) = \prod_{i=1}^{\hg} [a_i,b_i] c_0 d_0 \bar{c}_0 d_0^{-1} \prod_{j=1}^n (d_j c_j d_j^{-1}).
$$

\item Suppose that $a=0$, $n>0$, so that $g=2\hg +n-1$ for some non-negative integer $\hg$.
Given
$$
\xi= (a_1,b_1,\ldots, a_{\hg}, b_{\hg}, d_2,\ldots, d_n, c_1,\ldots, c_n) \in \Wt,
$$
where $a_i, b_i, d_j,\in\U(r)$ and $c_1,\ldots, c_n\in \U(r)^\tau$,   define
$$
\fm(\xi) = \prod_{i=1}^{\hg} [a_i,b_i] c_1 \prod_{j=2}^n (d_j c_j d_j^{-1}).
$$
\end{enumerate} 

It is straighforward to check that 
$$
\det \circ\, \fm (\Wt)= 
\begin{cases}
(-1)^{w^{(1)}+\cdots + w^{(n)}}, & \textup{if $n>0$ and $\tau=\tauR$},\\
(-1)^{r(g-1)},  & \textup{if $n=0$ and $\tau=\tauH$},\\
1, & \textup{otherwise}. 
\end{cases}
$$
We define the representation variety of central
Yang-Mills connections on a real or quaternionic Hermitian
bundle $(E,\tau)$ of rank $r$ and degree $d$ 
over a Klein surface $(M,\si)$ of topological type $(g,n,a)$ to be
$$
\Vt:= \fm^{-1}\big( \exp(-i\pi \frac{d}{r}I_r)\big) \subset \Wt.
$$
The following is a consequence of the results contained in \cite{BHH} and in Subsection \ref{Yang-Mills_theory} of the present paper:
\begin{theorem}\label{finite-dim_description}
The holonomy map
$$
\Hol: \cC^\tau/\cG_E^{\, \tau}(\vx) \to \Wt
$$
induces a $\Gt$-equivariant homeomorphism
$$
\Hol:\YMconn^{\, \tau}/\cG_E^{\, \tau}(\vx) \to  \Vt.
$$
\end{theorem}

\begin{corollary}
We have the following isomorphism of topological stacks
$$
[\YMconn^{\, \tau} /\cG_E^{\, \tau}] \simeq [\Vt/\Gt]
$$
and the following homotopy equivalences of homotopy orbit spaces
$$
(\Csst)_{h\cG_\C^\tau} \simeq (\Csst)_{h\cG_E^{\, \tau}} \simeq (\YMconn^{\, \tau})_{h\cG_E^{\, \tau}}.
$$
Therefore,
$$
P_{(g,n,a)}^{\ \tau}(r,d)  = P_t^{\Gt}\big(\Vt;\Z_2\big). 
$$
\end{corollary}

\section{The classifying spaces of the real and quaternionic gauge groups}\label{topology_of_groups}
Let $(E,\tau)$ be a real or quaternionic Hermitian vector bundle of rank $r$ and degree $d$
on a Klein surface $(M,\si)$ of topological type $(g,n,a)$. The goal of this section is to compute
$$
Q^{\ \tau}_{(g,n,a)}(r) = P_t(B\cG_E^{\, \tau};\bZ_2).
$$

\subsection{Topology of classical groups and their  classifying spaces }
In this subsection, we summarize some known results on the topology of classical
groups and their classifying spaces. Given a topological group $G$, let $G_0$ be 
the connected component of the identity of $G$.  For example, 
$(\U(r))_0=\U(r)$, $(\O(r))_0=\mathbf{SO}(r)$.

\subsubsection{Unitary groups} \label{sec:U}
$$
H^*(B\U(r);\bZ)=\bZ[c_1,\ldots, c_r],
$$
where $c_j\in H^{2j}(\BU(r);\bZ)$ is the universal $j$-th Chern class.
By the universal coefficient theorem, for any field $K$, 
$$
H^*(\BU(r);K)=K[c_1,\ldots, c_r].
$$
Therefore
$$
P_t(\BU(r);K)= \frac{1}{\prod_{j=1}^r(1-t^{2j})}
$$
for any field $K$.

$H^*(\U(r);K)\simeq \Lambda[x_1,x_3,\ldots, x_{2r-1}]$,
the exterior algebra on generators $x_j$, where 
$\deg\, x_j=j$ (see e.g. Example 5.F on page 150--151 of 
\cite{McCleary}).  Therefore
$$
P_t(\U(r);K)=\prod_{j=1}^r(1+t^{2j-1}).
$$
Let $\Omega \U(r)$ be the loop space of $\U(r)$. Then
$$
\pi_0(\Omega \U(r))= \pi_1(\U(r))=\bZ.
$$
Let $(\Omega \U(r))_0$  be the connected component
of the identity.  Then
$$
(\Omega \U(r))_0 = \Omega \widetilde{\U(r)}
$$
where $\widetilde{\U(r)}= \SU(r)\times \bR$ is the universal cover of $\U(r)$. 
So $(\Omega \U(r))_0$ is homotopy equivalent to $\Omega\SU(r)$.
We now describe the cohomology ring $H^*(\Omega\SU(r);\bZ)$ of the loop space
of $\SU(r)$, following \cite{Bott-loop}. The permutation group $S_m$ acts on
$X^m$, which induces an $S^m$-action on $H^*(X^m;\bZ)$. Let
$\mathcal{S}^m H^*(X;\bZ)$ be the $S_m$-invariant
subring of $H^*(X^m;\bZ)$. We define $\mathcal{S}^\infty H^*(X;\bZ)$
to be the inverse limit of the following inverse system of rings.
$$
H^*(X;\bZ)\longleftarrow \mathcal{S}^2 H^*(X;\bZ) \longleftarrow\mathcal{S}^3 H^*(X;\bZ)
\longleftarrow \cdots
$$
By \cite[Proposition 8.1]{Bott-loop}, there is a ring isomorphism
$$
H^*(\Omega\SU(r);\bZ) \simeq \mathcal{S}^\infty H^*(\mathbb{CP}^{r-1};\bZ).
$$
Introduce countably many variables $\{ x_i: i=1, 2, \ldots \}$. Let
$e_k$ be the $k$-th elementary symmetric functions of $\{x_i: i\in 1, 2,\ldots \}$:
$$
e_1 =\sum_i x_i,\quad e_2 = \sum_{i<j}x_i x_j,\quad
e_3=\sum_{i<j<k} x_i x_j x_k,\quad \cdots\quad .
$$
Let $\p_k=\sum_i x_i^k$. Then $\p_k\in \bZ[e_1, e_2,\ldots]$, and 
$$
H^*(\Omega\SU(r);\bZ) \simeq \bZ[e_1, e_2, \ldots ]/\langle \p_r, \p_{r+1},\ldots\rangle.
$$
where $\deg\, e_k = \deg\, \p_k = 2k$. We have inclusions of subrings:
$$
\bZ[\p_1,\ldots, \p_{r-1}]\subset \bZ[e_1,\ldots, e_{r-1}]\subset H^*(\Omega\SU(r);\bZ),
$$
which become isomorphisms over $\bQ$:
$$
\bQ[\p_1,\ldots, \p_{r-1}] = \bQ[e_1,\ldots, e_{r-1}] \simeq H^*(\Omega\SU(r);\bQ).
$$
Let $b_i = \dim_\bQ H^i(\Omega\SU(r);\bQ)$. Then for any $i$,
$H^i(\Omega(\SU(r));\bZ) \simeq \bZ^{b_i}$ and, or any field $K$, 
$$
P_t(\Omega \SU(r);K) =\sum_{i=0}^{+\infty} b_i t^i = 
\frac{1}{ \prod_{j=1}^{r-1}(1-t^{2j}) }\cdot
$$
Therefore
$$
P_t( (\Omega \U(r))_0;K) = \frac{1}{ \prod_{j=1}^{r-1}(1-t^{2j}) }
$$
for any field $K$.

\subsubsection{Orthogonal groups}
We use the notation $\bZ_2:=\bZ/2\bZ$. We recall that
$$
H^*(\BO(r);\bZ_2)=\bZ_2[w_1,\ldots, w_r],
$$
where $w_j\in H^j(\BO(r);\bZ_2)$ is the universal $j$-th Stiefel-Whitney class.
Therefore,
$$
P_t(\BO(r);\bZ_2)=\frac{1}{\prod_{j=1}^r (1-t^j)}\cdot
$$

$H^*(\SO(r);\bZ_2)$ ($r\geq 2$) has a simple system of
generators $\{ x_1,x_2,\ldots, x_{r-1}\}$, where $\deg\, x_j=j$ (see
e.g. Example 5.H on page 153--155 of \cite{McCleary}). Therefore
$$
P_t((\O(r))_0;\bZ_2)= P_t(\SO(r);\bZ_2)=\prod_{j=1}^{r-1}(1+t^j).
$$
The above formula also holds when $r=1$: in this case
$\SO(1)$ is a point, so $P_t(\SO(1);\bZ_2)=1$.

\subsubsection{Symplectic groups}
$$
H^*(\BSp(m);\bZ)= \bZ[\si_1,\ldots, \si_m]
$$
where $\si_j\in H^{4j}(\BSp(m);\bZ)$.
By the universal coefficient theorem,
$$
H^*(\BSp(m);K)=K[\si_1,\ldots, \si_m]
$$
for any field $K$. Therefore
$$
P_t(\BSp(m);K)=\frac{1}{\prod_{j=1}^m(1-t^{4j})}
$$
for any field $K$.

$H^*(\Sp(m);K) \simeq \Lambda(x_3,x_7,\ldots, x_{4m-1})$, 
the exterior algebra on generators $x_j$ where $\deg\, x_j=j$. Therefore,
$$
P_t(\Sp(m);K)=\prod_{j=1}^m(1+t^{4m-1}).
$$
for any field $K$.

\newcommand{\con}{\mathrm{con}}

\subsection{Equivariant cohomology of the  holonomy space} \label{sec:equ-hol}
In this section, we will compute
\begin{eqnarray*}
f_g(r) &:=& P_t(E\U(r) \times_{\U(r)} \U(r)^{2g};\bQ) =P_t^{\U(r)}(\U(r)^{2g};\bQ)\\ 
f^{\ \tau}_{(g,n,a)}(r) &:=& P_t(E \Gt \times_{\Gt} \Wt;\bZ_2) \\
&=& P_t^{\Gt}(\Wt;\bZ_2).
\end{eqnarray*}

We introduce the following notation. Given a compact Lie group $G$, let
$G^\con$ denote $G$ with the conjugation action by itself, and let
$E^\con_G := EG \times_G G^\con$ be the homotopy orbit space.
Then there is a fibration
$$
G \to E^\con_G\to BG.
$$
The above fibration is homotopic to the fibration
$$
\Omega(BG) \to  L(BG) \to BG,
$$
where $\Omega(BG)$ is homotopic to $G$, and $L(BG)$ is the free
loop space of $BG$ (see e.g. \cite[Lemma A.1]{KSS}).
We have
$$
H^*_G(G^\con;R) = H^*(E^\con_G;R)
$$
for any coefficient ring $R$.

\subsubsection{Computation of $f_g(r)$}
It is known that the cohomology Leray-Serre spectral sequence associated
of the fibration 
$$
\U(r)\to E^\con_{\U(r)}\to B\U(r)
$$
degenerates at the $E_2$ page over $\bZ$ or over any field
coefficient $K$ (see e.g. Example 15.40 on page 325-326 of \cite{Crabb-James}). 
Therefore,
\begin{eqnarray*}
f_g(r) &=& P_t(B\U(r);\bQ)\, P_t(\U(r)^{2g};\bQ)\\
& = & P_t(B\U(r);\bQ)\, P_t(\U(r);\bQ)^{2g}\\
&=& \frac{\prod_{j=1}^r(1+t^{2j-1})^{2g}}{\prod_{j=1}^r(1-t^{2j})}\cdot
\end{eqnarray*}

\subsubsection{Some preliminary results} \label{sec:UU}
Let $E$ be the homotopy orbit space of the following left $\U(r)\times \U(r)$-action on 
$\U(r)$
$$
(b_1, b_2) \cdot c  = b_1 c b_2^{-1}.
$$
There is a fibration
$$
\U(r) \to E \to B\U(r) \times B\U(r).
$$
The $E_2$-term of the cohomological Leray-Serre spectral sequence associated to the above fibration is
$$
E_2^{p,q} \simeq H^p(\U(r);\bZ_2)\otimes H^q(B\U(r)\times B\U(r);\bZ_2),
$$
where
\begin{eqnarray*}
H^*(\U(r);\bZ_2)&=&\Lambda[x_1, x_3,\cdots, x_{2r-1}], \\
H^*(\BU(r)\times \BU(r);\bZ_2)&=& \bZ_2[y_2,y_4,\cdots, y_{2r}, z_2, z_4,\cdots, z_{2r}].
\end{eqnarray*}
\begin{lemma} The nonzero differentials $d_k$, $k\geq 2$, are given by 
$$
d_{2\ell}: E^{p,q}_{2\ell} \to E^{p+2\ell, q-2\ell+1}_{2\ell},
\quad \alpha x_{2\ell-1}\mapsto  \alpha (y_{2\ell}+ z_{2\ell}).
$$
\end{lemma}
\begin{proof} We make the following observations:
\begin{enumerate} 
\item The group homomorphism 
$$
\U(r)\to \U(r)\times \U(r),\quad  h \mapsto (h,I_r)
$$
induces a continuous map $i: B\U(r)\to B\U(r)\times B\U(r)$. 
$$
i^*: H^*(B\U(r)\times B\U(r);\Z_2) =\Z_2[y_{2i}, z_{2i}] \lra 
H^*(B\U(r);\Z_2) =\Z_2[y_{2i}]
$$
is given by
$$
y_{2i}\mapsto y_{2i}, \quad z_{2i}\mapsto 0,\quad i=1,\ldots,r.
$$
The pullback fibration $i^* E\to B\U(r)$ is isomorphic to the
universal $\U(r)$-bundle $E\U(r)\to B\U(r)$. Therefore,
$$
i^* d_{2\ell}(x_{2\ell-1}) = y_{2\ell}
$$
and $i^*d_k=0$ otherwise.

\item The diagonal map 
$$
\U(r)\to \U(r)\times \U(r),\quad  h \mapsto (h,h)
$$
induces a continuous map $j: B\U(r)\to B\U(r)\times B\U(r)$. 
$$
j^*: H^*(B\U(r)\times B\U(r);\Z_2) =\Z_2[y_{2i}, z_{2i}] \lra 
H^*(B\U(r);\Z_2) =\Z_2[y_{2i}]
$$
is given by
$$
y_{2i}\mapsto y_{2i}, \quad z_{2i}\mapsto y_{2i},\quad i=1,\ldots,r.
$$
The pullback fibration $j^* E\to B\U(r)$ is isomorphic to
the fibration $E^\con_{\U(r)}\to B\U(r)$.
Therefore, $j^*d_k=0$ for all $k\geq 2$.
\end{enumerate}
The lemma follows from the above two observations, and induction
on $k\geq 2$.
\end{proof}

Let $K$ by any field.
The group isomorphism $\U(r)\to \U(r)$, $h\mapsto \bar{h}$ induces
a homeomorphism $\phi: B\U(r)\to B\U(r)$. 
$$
\phi^*:H^*(B\U(r);K)= K[ u_2, u_4, \ldots,u_{2r}]
\lra H^*(B\U(r);K) =K[u_2, u_4,\ldots, u_{2r}]
$$
is given by 
$$
u_{2i} \mapsto (-1)^i u_{2i}, \quad i=1,\ldots, r.
$$
In particular, when $K=\bZ_2$, $\phi^*$ is the identity map.

The inclusion $\O(r)\hookrightarrow \U(r)$ induces a continuous map
$\phi_{\tauR}: B\O(r)\to B\U(r)$.
$$
\phi_{\tauR}^*: H^*(B\U(r);\bZ_2)=\bZ_2[u_2,u_4, \ldots, u_{2r}] 
\to H^*(B\O(r);\bZ_2) =\bZ_2[w_1,w_2, \ldots, w_r]
$$
is given by 
$$
u_{2i} \mapsto w_i^2,\quad i=1,\ldots, r.
$$

Suppose that $r$ is even
The inclusion $\Sp(\frac{r}{2})\hookrightarrow \U(r)$ induces a continuous map
$\phi_{\tauH}: B\Sp(\frac{r}{2}) \to B\U(r)$.
$$
\phi_{\tauH}^*: H^*(B\U(r);\bZ_2)=\bZ_2[u_2,u_4,\ldots, u_{2r}] 
\to H^*(B\Sp(\frac{r}{2});\bZ_2) =\bZ_2[y_4, y_8, \ldots, y_{2r}]
$$
is given by 
$$
u_{4i-2}\mapsto 0,\quad u_{4i}\mapsto  y_{4i},\quad i=1,\ldots, \frac{r}{2}.
$$

\subsubsection{Computation of $f^{\ \tau}_{(g,n,a)}(r)$ when $a=1$, $n\geq 0$,  and $g-n=2\hg$ is even}
Let $\Wt$ be defined as in Definition \ref{df:W}. Given
$$
(a_1,b_1,\ldots, a_{\hg}, b_{\hg}, d_1,\ldots, d_n, c_0,c_1,\ldots, c_n) \in \Wt,
$$
where $a_i, b_i, d_j, c_0\in\U(r)$ and $c_1,\ldots, c_n\in \U(r)^\tau$,  and
$$
(h_0,h_1,\ldots, h_n) \in \Gt.
$$
where $h_0\in \U(r)$, $h_1,\ldots, h_n\in \U(r)^\tau$, the
group action is given by
\begin{eqnarray*}
&& (h_0, h_1,\ldots, h_n)\cdot (a_i, b_i, d_j, c_0,  c_j) \\
&=& (h_0 a_i h_0^{-1}, h_0 b_i h_0^{-1}, h_0 d_j h_j^{-1}, h_0 c_0 \bar{h}_0^{-1}, h_j c_j h_j^{-1}),
\end{eqnarray*}
where $i=1,\ldots,\hg$, $j=1,\ldots,n$.
If $n=0$, then $G^{\ \tau}_{(0,1)}(r)= \U(r)$,  
\begin{eqnarray*}
P_t^{G^{\ \tau}_{(0,1)}(r)}(W^\tau_{(2\hg,0,1)};\Z_2) & = &
P_t(\U(r);\bZ_2)^{2\hg+1} P_t(B\U(r);\bZ_2)\\
& = & \frac{\prod_{j=1}^r (1+t^{2j-1})^{g+1}}{\prod_{j=1}^r(1-t^{2j})}\cdot
\end{eqnarray*}

\noindent
From now on, we assume that $n>0$.  We have
\begin{eqnarray*}
&& P_t^{\GtR}(W^{\tauR}_{(2\hg+n,n,1)};\bZ_2) \\
&=& P_t(\U(r);\bZ_2)^{2\hg+1} \big(\prod_{j=1}^n P_t(\O(r)_{(-1)^{w^{(j)} } };\bZ_2) \big) P_t^{\GtR}(\U(r)^n;\bZ_2) \\
&& P_t^{\GtH}(W^{\tauH}_{(2\hg+n,n,1)};\bZ_2) \\
&=& P_t(\U(r);\bZ_2)^{2\hg+1} P_t(\Sp(\frac{r}{2});\bZ_2)^n P_t^{\GtH}(\U(r)^n;\bZ_2) 
\end{eqnarray*}
where $\Gt$ acts on $\U(r)^n$ by
$$
(h_0, h_1,\ldots, h_n)\cdot (d_1,\ldots, d_n) = (h_0 d_1 h_1^{-1},\ldots, h_0 d_n h_n^{-1}).
$$
Consider the fibration
$$
\U(r)^n \to E \Gt\times_{\Gt} (\U(r)^n) \to B \Gt.
$$
The $E_2$-term of the cohomological Leray-Serre spectral sequence associated to
the above fibration is
$$
E_2^{p,q} = H^p( B\Gt;\bZ_2)\otimes H^q(\U(r)^n;\bZ_2),
$$
where
\begin{eqnarray*}
H^*(B\GtR;\Z_2) &=& \Z_2[u_2, u_4,\ldots, u_{2r}] \otimes \bigotimes_{j=1}^n \bZ_2[y_{j,1}, y_{j,2}, \ldots, y_{j,r}] ,\\
H^*(B\GtH;\Z_2) &=& \Z_2[u_2, u_4,\ldots, u_{2r}]\otimes \bigotimes_{j=1}^n \bZ_2[y_{j,4}, y_{j,8}, \ldots, y_{j,2r}].
\end{eqnarray*}
$$
H^*(\U(r)^n;\bZ_2) = \bigotimes_{j=1}^n \Lambda[x_{j,1}, x_{j,3},\ldots, x_{j, 2r-1}]. 
$$
$$
d_{2\ell}: E^{p,q}_{2\ell}\lra E^{p+2\ell, q-2\ell+1}_{2\ell}
$$
is given by 
$$
\alpha x_{j,2\ell-1} \mapsto 
\begin{cases}
\alpha (u_{2\ell} + y_{j,\ell}^2), & \tau=\tauR,\\
\alpha u_{2\ell}, &\tau=\tauH \textup{ and $\ell$ is odd},\\
\alpha (u_{2\ell} + y_{j,2\ell}), &\tau =\tauH\textup{ and $\ell$ is even}.
\end{cases}
$$
Therefore,
\begin{eqnarray*}
& & H^*_{\GtR}(\U(r)^n;\Z_2) \\
& \simeq &  \Z_2[u_2, u_4,\ldots, u_{2r}] \otimes \bigotimes_{j=1}^n \Lambda[y_{j,1}, y_{j,2}, \ldots, y_{j,r}] \\
\mathrm{and} & &\\
& & H^*_{\GtH}(\U(r)^n;\Z_2) \\
&\simeq& \bZ_2[y_{1,4}, y_{1,8},\ldots, y_{1,2r}] \\
& & \otimes
\bigotimes_{j=2}^n \Lambda[x_{j,1}-x_{1,1} , x_{j,5}-x_{1,5}, \ldots, x_{j, 2r-3}-x_{1,2r-3} ]
\end{eqnarray*}
\begin{eqnarray*}
P_t^{\GtR}(\U(r)^n;\bZ_2)&=& \frac{\prod_{j=1}^r (1+t^j)^n}{\prod_{j=1}^r (1-t^{2j})} \\
P_t^{\GtH}(\U(r)^n;\bZ_2)&=&  \frac{\prod_{j=1}^{r/2} (1+t^{4j-3})^{n-1} }{\prod_{j=1}^{r/2}(1-t^{4j})}
\end{eqnarray*}
\begin{eqnarray*}
f^{\ \tauR}_{(g,n,a)}(r) &=& \frac{\prod_{j=1}^r (1+t^{2j-1})^{g-n+1} \prod_{j=1}^{r-1} (1+t^j)^n \prod_{j=1}^r(1+t^j)^n}{\prod_{j=1}^r (1-t^{2j})}\\
f^{\ \tauH}_{(g,n,a)}(r) &=&\frac{\prod_{j=1}^r(1+t^{2j-1})^{2\hg+1} \prod_{j=1}^{r/2} (1+t^{4j-1})^n \prod_{j=1}^{r/2} (1+t^{4j-3})^{n-1} }
{\prod_{j=1}^{r/2} (1-t^{4j}) }\\
&=& \frac{\prod_{j=1}^r(1+t^{2j-1})^g \prod_{j=1}^{r/2} (1+t^{4j-1})}{\prod_{j=1}^{r/2} (1-t^{4j}) }
\end{eqnarray*}

\subsubsection{Computation of $f^{\ \tau}_{(g,n,a)}(r)$ when $a=1$, $n\geq 0$,  and $g-n=2\hg+1$ is odd}
Let $\Wt$ be defined as in Definition \ref{df:W}.
Given
$$
(a_1,b_1,\ldots, a_{\hg}, b_{\hg}, d_0, d_1,\ldots, d_n, c_0,c_1,\ldots, c_n) \in \Wt,
$$
where $a_i, b_i, d_0, d_j, c_0\in\U(r)$ and $c_1,\ldots, c_n\in \U(r)^\tau$,  and
$$
(h_0,h_1,\ldots, h_n) \in \Gt.
$$
where $h_0\in \U(r)$, $h_1,\ldots, h_r\in \U(r)^\tau$, the group action is given by
\begin{eqnarray*}
&& (h_0,h_1,\ldots, h_n)\cdot  (a_i, b_i, d_0, d_j, c_0,  c_j)  \\
&=& (h_0 a_i h_0^{-1}, h_0 b_i h_0^{-1}, h_0 d_0 \bar{h}_0^{-1},  h_0 d_j h_j^{-1}, h_0 c_0 h_0^{-1}, h_j c_j h_j^{-1}),
\end{eqnarray*}
where $i=1,\ldots,\hg$, $j=1,\ldots,n$.
If $n=0$, then $G^{\ \tau}_{(0,1)}(r)= \U(r)$,  
\begin{eqnarray*}
P_t^{G^{\ \tau}_{(0,1)}(r)}(W^\tau_{(2\hg+1,0,1)};\Z_2) & = &
P_t(\U(r);\bZ_2)^{2\hg+2} P_t(B\U(r);\bZ_2)\\
& = & \frac{\prod_{j=1}^r (1+t^{2j-1})^{g+1}}{\prod_{j=1}^r(1-t^{2j})}\cdot
\end{eqnarray*}

\noindent
From now on, we assume that $n>0$.  We have
\begin{eqnarray*}
&& P_t^{\GtR}(W^{\tauR}_{(2\hg+n+1,n,1)};\bZ_2) \\
&=& P_t(\U(r);\bZ_2)^{2\hg+2} \big(\prod_{j=1}^n P_t(\O(r)_{(-1)^{w^{(j)} } };\bZ_2) \big) P_t^{\GtR}(\U(r)^n;\bZ_2) \\
&& P_t^{\GtH}(W^{\tauH}_{(2\hg+n+1,n,1)};\bZ_2) \\
&=& P_t(\U(r);\bZ_2)^{2\hg+2} P_t(\Sp(\frac{r}{2});\bZ_2)^n P_t^{\GtH}(\U(r)^n;\bZ_2) 
\end{eqnarray*}
where 
\begin{eqnarray*}
P_t^{\GtR}(\U(r)^n;\bZ_2)&=& \frac{\prod_{j=1}^r (1+t^j)^n}{\prod_{j=1}^r (1-t^{2j})} \\
P_t^{\GtH}(\U(r)^n;\bZ_2)&=&  \frac{\prod_{j=1}^{r/2} (1+t^{4j-3})^{n-1} }{\prod_{j=1}^{r/2}(1-t^{4j})}
\end{eqnarray*}
Therefore, 
\begin{eqnarray*}
f^{\ \tauR}_{(g,n,a)}(r) &=& \frac{\prod_{j=1}^r (1+t^{2j-1})^{g-n+1} \prod_{j=1}^{r-1} (1+t^j)^n \prod_{j=1}^r(1+t^j)^n}{\prod_{j=1}^r (1-t^{2j})}\\
f^{\ \tauH}_{(g,n,a)}(r) &=& \frac{\prod_{j=1}^r(1+t^{2j-1})^g \prod_{j=1}^{r/2} (1+t^{4j-1})}{\prod_{j=1}^{r/2} (1-t^{4j}) }
\end{eqnarray*}

\subsubsection{Computation of $f^{\ \tau}_{(g,n,a)}(r)$ when $a=0$, $n>0$, so that $g=2\hg +n-1$ for some non-negative integer $\hg$}
Let $\Wt$ be defined as in Definition \ref{df:W}. Given
$$
(a_1,b_1,\ldots, a_{\hg}, b_{\hg}, d_2,\ldots, d_n, c_1,\ldots, c_n) \in \Wt,
$$
where $a_i, b_i, d_j,\in\U(r)$ and $c_1,\ldots, c_n\in \U(r)^\tau$,  and
$$
(h_1,\ldots, h_n) \in \Gt.
$$
where $h_1,\ldots, h_n\in \U(r)^\tau$, the group action is given by
\begin{eqnarray*}
&& (h_1,\ldots h_n)\cdot (a_i, b_i, d_j, c_1,  c_j) \\
&=& (h_1 a_i h_1^{-1}, h_1 b_i h_1^{-1}, h_1 d_j h_j^{-1}, h_1 c_1 h_1^{-1}, h_j c_j h_j^{-1}),
\end{eqnarray*}
where $i=1,\ldots,\hg$, $j=2,\ldots,n$.
\begin{eqnarray*}
&& P_t^{\GtR}(W^{\tauR}_{(2\hg+n-1,n,0)};\bZ_2) \\
&=& P_t(\U(r);\bZ_2)^{2\hg} \big(\prod_{j=1}^n P_t(\O(r)_{(-1)^{w^{(j)} } };\bZ_2) \big) P_t^{\GtR}(\U(r)^{n-1};\bZ_2) \\
&& P_t^{\GtH}(W^{\tauH}_{(2\hg+n-1,n,0)};\bZ_2) \\
&=& P_t(\U(r);\bZ_2)^{2\hg} P_t(\Sp(\frac{r}{2});\bZ_2)^n P_t^{\GtH}(\U(r)^{n-1};\bZ_2) 
\end{eqnarray*}
where $\Gt=(\U(r)^\tau)^n$ acts on $\U(r)^{n-1}$ by
$$
(h_1,\ldots, h_n)\cdot (d_2,\ldots, d_n) = (h_1 d_2 h_2^{-1},\ldots, h_1 d_n h_n^{-1}).
$$
Consider the fibration
$$
\U(r)^{n-1} \to E \Gt\times_{\Gt} (\U(r)^n) \to B \Gt.
$$
There is a spectral sequence  with
$$
E_2^{p,q} = H^p( B\Gt;\bZ_2)\otimes H^q(\U(r)^{n-1};\bZ_2)
$$
where
\begin{eqnarray*}
H^*(B\GtR;\Z_2) &=& \bigotimes_{j=1}^n \bZ_2[y_{j,1}, y_{j,2}, \ldots, y_{j,r}] ,\\
H^*(B\GtH;\Z_2) &=& \bigotimes_{j=1}^n \bZ_2[y_{j,4}, y_{j,8}, \ldots, y_{j,2r}].
\end{eqnarray*}
$$
H^*(\U(r)^{n-1};\bZ_2) = \bigotimes_{j=2}^n \Lambda[x_{j,1}, x_{j,3},\ldots, x_{j, 2r-1}]. 
$$
$$
d_{2\ell}: E^{p,q}_{2\ell}\lra E^{p+2\ell, q-2\ell+1}_{2\ell}
$$
is given by 
$$
\alpha  x_{j,2\ell-1} \mapsto 
\begin{cases}
\alpha (y_{1,\ell}^2 + y_{j,\ell}^2), & \tau=\tauR,\\
0 , &\tau=\tauH \textup{ and $\ell$ is odd},\\
\alpha (y_{1,2\ell} + y_{j,2\ell}), &\tau =\tauH\textup{ and $\ell$ is even}.
\end{cases}
$$
Therefore,
\begin{eqnarray*}
H^*_{\GtR}(\U(r)^{n-1};\Z_2) &\simeq&   \Z_2[ y_{1,1}, y_{1,2},\ldots, y_{1,r}] \otimes \bigotimes_{j=2}^n \Lambda[y_{j,1}, y_{j,2}, \ldots, y_{j,r}] \\
H^*_{\GtH}(\U(r)^{n-1};\Z_2) &\simeq&   \bZ_2[y_{1,4}, y_{1,8},\ldots, y_{1,2r}] \otimes
\bigotimes_{j=2}^n \Lambda[x_{j,1} , x_{j,5}, \ldots, x_{j, 2r-3} ]
\end{eqnarray*}
\begin{eqnarray*}
P_t^{\GtR}(\U(r)^{n-1};\bZ_2)&=& \frac{\prod_{j=1}^r (1+t^j)^{n-1}}{\prod_{j=1}^r (1-t^j)} 
=\frac{\prod_{j=1}^r (1+t^j)^n}{\prod_{j=1}^r (1-t^{2j})} \\
P_t^{\GtH}(\U(r)^{n-1};\bZ_2)&=&  \frac{\prod_{j=1}^{r/2} (1+t^{4j-3})^{n-1} }{\prod_{j=1}^{r/2}(1-t^{4j})}
\end{eqnarray*}
\begin{eqnarray*}
f^{\ \tauR}_{(g,n,a)}(r) &=& \frac{\prod_{j=1}^r (1+t^{2j-1})^{g-n+1} \prod_{j=1}^{r-1} (1+t^j)^n \prod_{j=1}^r(1+t^j)^n}{\prod_{j=1}^r (1-t^{2j})}\\
f^{\ \tauH}_{(g,n,a)}(r) &=&\frac{\prod_{j=1}^r(1+t^{2j-1})^{2\hg} \prod_{j=1}^{r/2} (1+t^{4j-1})^n \prod_{j=1}^{r/2} (1+t^{4j-3})^{n-1} }
{\prod_{j=1}^{r/2} (1-t^{4j}) }\\
&=& \frac{\prod_{j=1}^r(1+t^{2j-1})^g \prod_{j=1}^{r/2} (1+t^{4j-1})}{\prod_{j=1}^{r/2} (1-t^{4j}) }\cdot
\end{eqnarray*}

\subsubsection{Summary of computation of $f^{\ \tau}_{(g,n,a)}(r)$}

\begin{theorem} \label{thm:f}
~
\begin{enumerate}
\item Suppose that $n=0$ (in which case $a=1$). Then
$$
f^{\tauR}_{(g,0,1)}(r) = f^{\tauH}_{(g,0,1)}(r) =\frac{\prod_{j=1}^r (1+t^{2j-1})^{g+1}}{\prod_{j=1}^r(1-t^{2j})}\cdot
$$
\item Suppose that $n>0$. Then
\begin{eqnarray*}
f^{\ \tauR}_{(g,n,a)}(r) &=& \frac{\prod_{j=1}^r (1+t^{2j-1})^{g-n+1} \prod_{j=1}^{r-1} (1+t^j)^n \prod_{j=1}^r(1+t^j)^n}{\prod_{j=1}^r (1-t^{2j})}\\
f^{\ \tauH}_{(g,n,a)}(r) &=& \frac{\prod_{j=1}^r(1+t^{2j-1})^g \prod_{j=1}^{r/2} (1+t^{4j-1})}{\prod_{j=1}^{r/2} (1-t^{4j}) }\cdot
\end{eqnarray*}
\end{enumerate}
\end{theorem}

\subsection{Cohomology of the classifying space}
By \eqref{eqn:nine-C}, there is a fibration
\begin{equation}\label{eqn:degenerate-C}
B\big(\Omega^2(\U(r))\big) \lra  B(\cG_E) \lra E\U(r) \times_{\U(r)} \U(r)^{2g}.
\end{equation}
By the results in \cite{AB},
$$
P_t(B(\cG_E);\bQ) = P_t\bigl( B\big(\Omega^2(\U(r))\big) ;\bQ\bigr) P_t\bigl(E\U(r) \times_{\U(r)} \U(r)^{2g};\bQ \bigr).
$$
\noindent So the rational Poincar\'e series of $B(\cG_E)$ is
$$
Q_g(r) = \frac{ f_g(r)}{\prod_{j=1}^{r-1}(1-t^{2j})} =\frac{\prod_{j=1}^r(1+t^{2j-1})^{2g}}{\prod_{j=1}^{r-1}(1-t^{2j})\prod_{j=1}^r(1-t^{2j})}\cdot
$$

\noindent By \eqref{eqn:nine-RH}, there is a fibration
\begin{equation}\label{eqn:degenerate}
B\big(\Omega^2(\U(r))\big)\lra B(\cG_E^{\, \tau}) \lra \EGt \times_{\Gt} \Wt. 
\end{equation}
We will compute the mod 2 Poincar\'{e} series of $B(\cG_E^{\, \tau})$
using the mod 2 cohomological Leray-Serre spectral squence associated
to the fibration \eqref{eqn:degenerate}. We will see that
the $E_2$ page collapses when $\tau=\tauR$ or $n=0$, but
does not collapse when $\tau=\tauH$ and $n>0$.

\subsubsection{Pushforward of the universal bundle}
Recall from \cite[Section 2]{AB} that $B(\cG_E)$ is a connected component of $\mathrm{Map}(M;B\U(r))$ and that there is therefore an evaluation map 
$$
B(\cG_E) \times M \lra B\U(r).
$$
Let $V\in K^0(B(\cG_E)\times M)$ be the universal bundle (i.e.\ the pullback, under the evaluation map, of the vector bundle $W$ associated to the principal $\U(r)$-bundle $E\U(r)\lra B\U(r)$) and let
$U=f_! V \in K^0(B(\cG_E))$ be the pushforward to the first factor
(see \cite[Section 2]{AB}). Then $U$ can be viewed as a virtual
complex vector bundle over $B(\cG_E)$, or a virtual $\cG_E$-equivariant
complex vector bundle over $\cC$.  More explicitly, the fibre of $U$
over $\cE\in \cC$ is the virtual complex vector space
\begin{equation}\label{eqn:Wfiber}
H^0(M,\cE)-H^1(M,\cE). 
\end{equation}
$U$ is a virtual complex vector bundle over $B(\cG_E)$ of (complex) rank $d+r(1-g)$.
Let us consider the Chern classes of $U$: these define an infinite sequence of integral cohomology classes
$$
\delta_k :=c_k(U) \in H^{2k}(B(\cG_E);\bZ), \quad k\geq 1.
$$
We introduce some notation as follows.
\begin{enumerate}
\item Let $i: \Omega\U(r)_0 = B\big(\Omega^2(\U(r))\big)\lra B(\cG_E)$ be the inclusion of
a fibre of the fibration \eqref{eqn:degenerate-C}.
\item Let $i_\tau: \Omega\U(r)_0 = B\big(\Omega^2(\U(r))\big)\lra B(\cG_E^{\, \tau})$ be the inclusion
of a fibre of the fibration \eqref{eqn:degenerate}.
\item Let $j:B(\cG_E^{\, \tau})\lra B(\cG_E)$ be induced by the inclusion map $\cG_E^{\, \tau} \hookrightarrow \cG_E$. 
\end{enumerate}
By the results in \cite[Section 2]{AB},  
$\{\delta'_k:= i^* \delta_k:k =1,2,\ldots \}$ are multiplicative generators of $H^*(\Omega\U(r)_0 ;\bZ)$. 
(In the notation of page 543 of \cite{AB}, Atiyah-Bott showed that $\{ i^*e_k: k=1,2, \ldots \}$
are multiplicative generators of $H^*(\Omega\U(r)_0;\bZ)$, but the difference $e_k-\delta_k$ is a polynomial
in $a_1,\ldots, a_r$, which are in the kernel of $i^*: B(\cG_E)\lra \Omega\U(r)_0$,
so $i^* e_k = i^* \delta_k$.) For each degree $p$, $H^p(\Omega\U(r)_0;\bZ)$ is a finitely generated
$\bZ$-module.

\subsubsection{The real case: $\tau=\tauR$}\label{sec:tauRgenerate}
For $\cE\in \cC^\tau$, $\tau$ induces $\bC$-anti-linear involutions on $H^0(M,\cE)$ and $H^1(M,\cE)$.
The $\tau$-invariant part of \eqref{eqn:Wfiber} forms a $\cG_E^{\, \tau}$-equivariant
virtual real vector bundle $U^\tau$ over $\cC^\tau$ of (real) rank
$d+r(1-g)$, and  $j^*U = U^\tau\otimes_\bR \bC$. For any positive integer $k$, we have
$$
j^* \delta_{2k-1} = c_{2k-1}(U^\tau\otimes_\bR \bC)=0,\quad
j^* \delta_{2k} = c_{2k}(U^\tau\otimes_\bR\bC) = (-1)^k p_k(U^\tau),
$$
where $p_k(U^\tau)\in H^{4k}(B(\cG_E^{\, \tau});\bZ)$ is the $k$-th Pontryagin class
of $U^\tau$. Let us consider the Stiefel-Whitney classes of $U^\tau$: these define an infinite sequence of mod 2 cohomology classes
$$
v_k := w_k(U^\tau)\in H^k(B(\cG_E^{\, \tau});\bZ_2),\quad k\geq 1.
$$
We have 
$$
i_\tau^* U^\tau = i^* U + m \bR
$$ 
as elements in  $KO(\Omega\U(r)_0)$, where
$m\in \bZ$ and $\bR$ is the trivial real line bundle. So for any 
positive integer $k$, $i_\tau^*v_{2k-1}=0$, and $i_\tau^*v_{2k}$ is the mod 2 reduction of $\delta_k'=i^*\delta_k$. 
Therefore, $\{ i_\tau^*v_{2k}: k=1,2,\ldots\}$ are multiplicative generators of  $H^*(\Omega\U(r)_0 ;\bZ_2)$.
For each degree $p$, $H^p(\Omega\U(r)_0;\bZ_2)$ is a finite-dimensional vector space 
over $\bZ_2$. We may apply the Leray-Hirsch theorem and
conclude that the mod 2 cohomological Leray-Serre 
spectral sequence of the fibration \eqref{eqn:degenerate}
degenerates at the $E_2$ page. Therefore, 
$$
P_t(B(\cG_E^{\tau});\bZ_2)  =  P_t(B\big(\Omega^2(\U(r))\big);\bZ_2) P_t(\EGt \times_{\Gt} \Wt;\bZ_2).
$$
So the $\mod{2}$ Poincar\'e series of $B(\cG_E^{\tau})$ is 
\begin{equation} \label{eqn:Qf}
Q_{(g,n,a)}^{\ \tau}(r) = \frac{f^{\ \tau}_{(g,n,a)}(r)}{\prod_{i=1}^{r-1}(1-t^{2i})}\cdot
\end{equation}
Part (1) of Theorem \ref{classifying_space} then follows from \eqref{eqn:Qf} and Theorem \ref{thm:f}.

\subsubsection{The path-loop fibration}
We have a fibration
\begin{equation}\label{eqn:disk-gauge}
\Omega^2(B\U(r))\lra \Map(D^2, B\U(r))\lra \Omega B\U(r)
\end{equation}
which can be identified with the following path-loop fibration
$$
\Omega(\Omega B\U(r)) \lra P(\Omega B\U(r)) \lra \Omega B\U(r).
$$
Up to homotopy equivalence, the above fibration can be identified with
\begin{equation} \label{eqn:path-loop}
\Omega \SU(r) \lra P\SU(r)\lra \SU(r).
\end{equation}
The total space $P\SU(r)$ is contractible. We now consider the 
Leray-Serre spectral sequence associated to the fibration \eqref{eqn:path-loop}.
We first consider rational cohomology. Then
$$
E_2^{p,q}(\bQ)= H^p(\SU(r);\bQ)\otimes H^q(\Omega\SU(r);\bQ),
$$
where
\begin{eqnarray*}
H^*(\SU(r);\bQ) &=& \Lambda[x_3,x_5,\ldots, x_{2r-1}] \\
H^*(\Omega\SU(r);\bQ)&=&\bQ[\p_1,\ldots, \p_{r-1}].
\end{eqnarray*}
Here $\p_k\in H^{2k}(\Omega\SU(r);\bQ)$ are defined as in Section \ref{sec:U}.
$$
d_{2\ell+1}: E^{p,q}_{2\ell+1}(\bQ)\lra E^{p+2\ell+1, q-2\ell}_{2\ell+1}(\bQ)
$$
is determined by $d_{2\ell+1}(\p_\ell)= x_{2\ell+1}$.
We have $d_j=0$ when $j$ is even or $j\geq 2r$, and 
$E^{p,q}_k(\bQ)=0$ for $k\geq 2r$. 

We now consider integral cohomology. Then for all $p,q, k \in \bZ$, where $p,q\geq 0$ and $k\geq 2$: 
\begin{enumerate}
\item $E^{p,q}_k(\bZ)$ is a finitely generated
free abelian group, and $$E^{p,q}_k(\bQ) = E^{p,q}(\bZ)\otimes_\bZ\bQ\, .$$
\item $d_k: E^{p,q}_k(\bZ) \lra E^{p+k, q-k+1}(\bZ)$ is the restriction
of $$d_k: E^{p,q}_k(\bQ)\lra E^{p+k,q-k+1}(\bQ)\, .$$
\item If we consider the mod 2 cohomological Leray-Serre spectral
sequence associated to the fibration \eqref{eqn:path-loop}. Then
$$d_k:E^{p,q}_k(\bZ_2)\lra E^{p+k,q-k+1}(\bZ_2)$$
is the mod 2 reduction of $d_k:E^{p,q}(\bZ)\lra E^{p+k,q-k+1}(\bZ)$.
\end{enumerate}

When $r$ is even, let $\cG^{\ \tauH}_{D^2,\partial D^2,1}(r)$ denote the space of
continuous maps $f:D^2 \lra \SU(r)$ such that $f(\partial D^2)\subset \Sp(\frac{r}{2})$ and
$f(1)=I_r$, where $1\in \partial D^2$, and $I_r\in \Sp(\frac{r}{2})$ is the $r\times r$
identity matrix.
Then there is a fibration
\begin{equation}\label{eqn:disk-H}
B(\Omega^2(\U(r)))\lra  B(\cG^{\ \tauH}_{D^2,\partial D^2,1}(r))\lra  B(\Omega\Sp(\frac{r}{2})).
\end{equation} 
Note that \eqref{eqn:disk-H} 
is the pullback of the fibration \eqref{eqn:disk-gauge} under the map $$B(\Omega\Sp(\frac{r}{2}))\lra B(\Omega\U(r)_0)$$
induced by the inclusion $\Sp(\frac{r}{2})\lra \SU(r)$. We consider the rational cohomological
spectral sequence associated to the fibration \eqref{eqn:disk-H}. Then
$$
E_2^{p,q}(\bQ)=H^p(\Sp(\frac{r}{2});\bQ)\otimes H^q(\Omega\SU(r);\bQ)
$$
where
\begin{eqnarray*}
H^*(\Sp(\frac{r}{2});\bQ)&=& \Lambda[x_3, x_7,\ldots, x_{2r-1}],\\
H^*(\Omega\SU(r);\bQ)&=&\bQ[\p_1,\ldots, \p_{r-1}],
\end{eqnarray*}
and 
$$
d_{4\ell-1}: E^{p,q}_{4\ell-1}(\bQ) \lra E^{p+4\ell-1, q-4\ell+2}_{4\ell-1}(\bQ)
$$ 
is determined by $d_{4\ell-2}(\p_{2\ell-1})= x_{4\ell-1}$. We have $d_j=0$ if
$j\notin (-1+ 4\bZ)$ or $j\geq 2r$, and $E^{p,q}_k=0$ for $k\geq 2r$.
We have
$$
H^*(\cG^{\ \tauH}_{D^2,\partial D^2,1}(r);\bQ)\simeq \bQ[\p_2, \p_4, ..., \p_{r-2}]
$$
$$
P_t(\cG^\tau_{D^2,\partial D^2, 1}(r);\bQ)=\frac{1}{\prod_{i=1}^{\frac{r}{2}-1}(1-t^{4j})}\cdot
$$
We also have statements similar to (1), (2), (3), above, and
$$
P_t(\cG^\tau_{D^2,\partial D^2, 1}(r);\bZ)=\frac{1}{\prod_{i=1}^{\frac{r}{2}-1}(1-t^{4j})}\cdot
$$

\subsubsection{The quaternionic case: $\tau=\tauH$} \label{sec:tauHgenerate}
We first consider the case $n=0$. Then the exists a quaternionic line bundle
$\cL^\bH$ of degree $(g-1)$. Let $(L^{\bH},\tau_L) $ denote the underlying topological line
bundle. Then $(E',\tauR):= (E\otimes L^{\bH}, \tauH \otimes \tau_L)$ is a real vector bundle
of rank $r$ and degree $(d+r(g-1))$, which is indeed even if $\tau=\tauH$.
There is a group isomorphism and homeomorphism
$$
\cG^{\tauH}_E \stackrel{\simeq}{\longrightarrow} \cG^{\tauR}_{E'},
\quad u\mapsto u\otimes \mathrm{Id}_{L^{\bH}}.
$$
Therefore, 
$$
Q^{\ \tauH}_{(g,0,1)}(r)= Q^{\ \tauR}_{(g,0,1)}(r).
$$
This proves Part (2) of Theorem \ref{classifying_space} when $n=0$.
This also implies that the mod 2 cohomological Leray-Serre spectral sequence 
associated to the fibration \eqref{eqn:degenerate} collapses at
the $E_2$ page in this case.

From now on, we assume that $n>0$, so that both $r$ and $d$ are even.
We have 
$$
i_\tau^*j^*U = i^*U \oplus \overline{i^*U}  + m\bC
$$
as elements in $K(\Omega\U(r)_0)$, where $m\in \bZ$ and
$\bC$ is the trivial complex line bundle. So for any positive integer $k$,
$$
i_\tau^* j^* \delta_{2k-1}=0,\quad i_\tau^* j^*\delta_{2k}= (-1)^k i^*p_k(U)
$$
where $p_k(U)\in H^{4k}(B(\cG_E);\bZ)$ is the $k$-th Pontryagin class of $U$.

We consider the mod 2 cohomological Leray-Serre spectral sequence associated
to the fibration \eqref{eqn:degenerate}. We have
$$
E_2^{p,q} = H^p_{\Gt}(\Wt;\bZ_2)\otimes H^q(\Omega\U(r)_0;\bZ_2)
$$
where (see Section \ref{sec:equ-hol}) 
\begin{eqnarray*}
&& H^*_{\Gt}(\Wt;\bZ_2)\\
&=& \bigotimes_{i=1}^g \Lambda[x_{i,1},x_{i,3},\ldots, x_{i,2r-1}] \otimes\Lambda[y_3, y_7,\ldots, y_{2r-1}]\otimes
\bZ_2[z_4,z_8,\ldots, z_{2r}]
\end{eqnarray*}
and
$$
H^*(\Omega\U(r)_0;\bZ_2) = \bZ_2[e_1, e_2,\ldots ]/\langle \p_r, \p_{r+1},\ldots\rangle
\,.$$
Let $R \subset H^*(\Omega\U(r)_0;\bZ)$ be the subring generated by $\{ i^*p_k(U):k=1,2, ...\}$.
If $\alpha \in R$ then its mod 2 reduction is in the kernel of all higher differentials of the
spectral sequence.

Let $\mathcal{I}:\Sp(\frac{r}{2})\lra E(\Gt)\times_{\Gt}\Wt$ 
be the composition of the inclusion 
$$
\Sp(\frac{r}{2})\lra \Wt = \U(r)^{g+a}\times \Sp(\frac{r}{2})^n,\quad 
c\mapsto (\underbrace{I_r,\ldots, I_r}_{g+a}, c, \underbrace{I_r,\ldots, I_r}_{n-1})
$$
and the inclusion of a fibre
$$
\Wt\lra E(\Gt)\times_{\Gt}\Wt.
$$
The pull-back of the fibration \eqref{eqn:degenerate} under $\mathcal{I}$ can be identified
with the fibration \eqref{eqn:disk-H}. So  $d_{4\ell-1}: E^{p,q}_{4\ell-1}\lra E^{p+4\ell-1, q-4\ell+2}_{4\ell-1}$ satisfies 
$$
d_{4\ell-1}(\p_{2\ell-1})= y_{4\ell-1} \quad (\mod{\mathrm{Ker}\, \mathcal{I}^*})
$$
where 
$$
\mathcal{I}^*:H^*_{\Gt}(\Wt;\bZ_2)\lra H^*(\Sp(\frac{r}{2});\bZ_2)
$$
is induced by the inclusion $\mathcal{I}$. 

Based on the above discussion, if $n>0$,
\begin{eqnarray*}
P_t(B(\cG_E^{\tauH});\bZ_2) &=& P_t\Big(\bigotimes_{i=1}^g \Lambda[x_{i,1},x_{i,3},\ldots, x_{i,2r-1}]\otimes\bZ_2[z_4, z_8,\ldots, z_{2r}]\Big)\\
&& \times \ P_t(B\cG^{\ \tauH}_{D^2,\partial D^2,1}(r);\bZ_2)\\
&=& \frac{\prod_{j=1}^r(1+t^{2j-1})^g }{\prod_{j=1}^{r/2-1}(1-t^{4j})\prod_{j=1}^{r/2}(1-t^{4j})}\cdot
\end{eqnarray*}
This proves Part (2) of Theorem \ref{classifying_space} when $n>0$.

\section{Stratifications of spaces of real and quaternionic structures}\label{stratification_section}

Let $(E,\tau)$ be a fixed real or quaternionic Hermitian vector bundle of rank $r$ and degree $d$ on a Klein surface $(M,\si)$, with complex gauge group $\cG_\C$, and let $\cC$ be the space of holomorphic structures / unitary connections on $E$. We saw in Section \ref{real_and_quaternionic_bundles} that there was a $\Gal(\C/\R)$-action defined on $\cC$ by $$A \longmapsto \ov{A} := \phi\, \os{A}\, \phi^{-1},$$ where $\phi:\os{E} \overset{\simeq}{\longrightarrow} E$ is the bundle isomorphism determined by $\tau$, as well as a compatible involution $$g \longmapsto  \ov{g} := \phi\, \os{g}\, \phi^{-1}$$ of $\cG_\C$. The goal of the present section is to show that this Galois action preserves the strata of the Shatz (Harder-Narasimhan) stratification of $\cC$ and that the Galois-invariant parts of the strata form a stratification of $\cC^{\tau}$. Moreover, this induced stratification is $\cG_\C^{\, \tau}$-equivariantly perfect over the field $\Z/2\Z$. We start by recalling the basics about the Shatz stratification and the strategy of Atiyah and Bott to show that it is $\cG_\C$-equivariantly perfect over any of the fields $K=\bQ$ or $\Z/p\Z$ for $p$ prime. We then proceed with our case, emphasizing the analogy with the Atiyah-Bott picture.

\subsection{The Shatz stratification}\label{Shatz_stratif}

Harder and Narasimhan showed in \cite{HN} that a holomorphic vector bundle $\cE$ of rank $r$ and degree $d$, say, had a unique filtration $$\{0\} = \cE_0 \subset \cE_1 \subset \cdots \subset \cE_l = \cE$$ by holomorphic sub-bundles such that $\cE_i / \cE_ {i-1}$ is semi-stable for all $i$ and $$\mu(\cE_1 / \cE_0) > \mu(\cE_2 / \cE_1) > \cdots > \mu(\cE_l / \cE_{l-1}).$$ This Harder-Narasimhan filtration is uniquely defined and we can sketch a proof of its existence as follows. Given a holomorphic bundle $\cE$, the slope function is bounded on the set of sub-bundles of $\cE$, so one may consider the set of sub-bundles of $\cE$ the slope of which is maximal. Among those, choose a sub-bundle $\cE_1$ whose rank is maximal. This $\cE_1$ is necessarily semi-stable and we have $\cE_1 = \cE$ if and only if $\cE$ is semi-stable, in which case the Harder-Narasimhan filtration is of length $l=1$ and the Harder-Narasimhan type of $\cE$ is $(\frac{d}{r}, \cdots , \frac{d}{r})$. If $\cE_1
\neq \cE$, one applies the same construction to $\cE / \cE_1$. This gives a sub-bundle $\cE_2 \supset \cE_1$ of $\cE$ and $\mu(\cE_1) > \mu(\cE_2)$ by the choice of $\cE_1$, which implies $\mu(\cE_1) > \mu(\cE_2 / \cE_1)$. The uniqueness of the filtration shows in particular that, among sub-bundles of $\cE$ the slope of which is maximal, there is a unique sub-bundle $\cE_1$ which has maximal rank (called the \textit{destabilizing bundle} of $\cE$). We denote 
$$
r_i = \rk(\cE_i / \cE_{i-1}), \quad d_i = \deg(\cE_i / \cE_{i-1}),\quad \mu_i = \frac{d_i}{r_i},
$$ 
$$
\mu = (\underbrace{\mu_1, \cdots, \mu_1}_{r_1}, \cdots, \underbrace{\mu_l, \cdots, \mu_l}_{r_l}),
$$ 
and 
$$
\cP_{\mu} = \big\{ (0,0), (r_1,d_1), (r_1+r_2,d_1+d_2), \cdots, (r_1+\cdots+r_l, d_1+\cdots+d_l)\big\}.
$$ 
One has 
$$
r_1+\cdots+r_i = \rk(\cE_i)
$$ 
and 
$$
d_1+\cdots+d_i= \deg(\cE_i)
$$ 
for all $i$, and 
$$
\mu_1 > \cdots > \mu_l.
$$ The $r$-tuple $\mu$ is called the Harder-Narasimhan type of $\cE$ and $\cP_{\mu}$ is its associated Shatz polygon. We denote $\Ird$ the set of all possible Harder-Narasimhan types of holomorphic structures on a Hermitian vector bundle $E$ of rank $r$ and degree $d$ and, for all $\mu\in\Ird$, we denote $\Cmu$ the set of holomorphic structures of type $\mu$ on $E$. In particular, $\Css = \cC_{\mu_{ss}}$ where 
$$
\mu_{ss} = \Big(\frac{d}{r},\cdots,\frac{d}{r}\Big)
$$ 
and one has $\cP_{\mu_{ss}} = \{(0,0),(r,d)\}$. Then, 
$$
\cC = \Css \cup \{\Cmu : \mu\in\Irdp\},
$$ and this is a stratification of $\cC$ called the Shatz stratification (\cite{Shatz}). Remarkably, it coincides with the Morse stratification of the Yang-Mills functional (\cite{Daskalopoulos}). As two isomorphic holomorphic bundles have the same Harder-Narasimhan type, any $\Cmu$ is a union of $\cxG$-orbits. When $\cE$ is a holomorphic bundle of type $\mu$, one denotes $\End'\,\cE$ the holomorphic bundle of endomorphisms of $\cE$ that preserve the Harder-Narasimhan filtration and define $\End''\,\cE$ by the exact sequence $$0 \longrightarrow \End'\,\cE \longrightarrow \End\, \cE \longrightarrow \End''\,\cE \longrightarrow 0.$$ This was used by Atiyah and Bott to identify the normal bundle to $\Cmu$ in $\cC$.

\begin{proposition}[\cite{AB}]\label{normal_bundle_AB}
The complex dimension of the sheaf cohomology group $H^1(M; \End''\, \cE)$ only depends on the Harder-Narasimhan type $\mu$ of $\cE$. It is denoted $\dmu$. Additionally, $\Cmu$ is a locally closed submanifold of $\cC$, of codimension $\dmu$, and the fibre of the normal bundle to $\Cmu$ at a point $\cE$ is isomorphic to $H^1(M;\End''\,\cE)$. Finally, for $\mu=(\mu_1, \cdots, \mu_l)$, one has 
$$
\dmu = \sum_{1\leq i < j \leq l} r_ir_j\big( \mu_i -\mu_j + (g-1)\big).
$$
\end{proposition}

\noindent A partial ordering on $\Ird$ with the property that the closure of $\Cmu$ satisfies 
\begin{equation*}
\ov{\Cmu} \subset \bigcup_{\mu' \geq \mu} \cC_{\mu'}
\end{equation*} 
is defined as follows~:
$$
\mu \leq \mu'\quad \mathrm{if\ and\ only\ if} \quad \mathrm{conv}(\cP_{\mu}) \subset \mathrm{conv}(\cP_{\mu'}),
$$ 
where $\mathrm{conv}(\cP_{\mu})$ is the convex polygon in the plane $(r,d)$ determined by the $r$-axis, the $d$-axis, the line $r=\rk(E)$ and the points $(\rk(\cE_i), \deg(\cE_i))$. The condition $\mu \leq \mu'$ is equivalent to 
$$
\mu_1 + \cdots + \mu_i \leq \mu'_1 + \cdots + \mu'_i
$$ 
for all $i\in\{1,\cdots,r\}$. This partial ordering 
can be refined to a total ordering (see e.g. \cite[Proposition 2.4]{Ram}).
From now on, we fix such a total ordering on $\Ird$ (when such a total ordering is fixed, the surjectivity of the Kirwan map immediately follows from the perfection of the stratification, see \cite{Kirwan}; this is the only place where we actually use the total ordering). Setting 
$$
U_{\mu} = \bigcup_{\mu' \leq \mu} \Cmup,
$$ 
Atiyah and Bott used the equivariant Thom isomorphism 
$$
H^j_{\cG_\C} (\Umu, \Umu \setminus \Cmu; K) \simeq H^{j-2\dmu}_{\cG_\C} (\Cmu;K),
$$ 
where $2\dmu = \mathrm{codim}_{\R}\, \Cmu$ in $\cC$, to write the equivariant Gysin exact sequences (one for each $\mu$) of the stratification 
$$
\cC = \bigsqcup_{\mu \in \Ird} \Cmu
$$ 
in the following way~:
$$
\cdots \longrightarrow H^{j-2\dmu}_{\cG_\C} (\Cmu;K) \longrightarrow H^{j}_{\cG_\C}(\Umu;K) 
\overset{\mathrm{restr.}}{\longrightarrow} H^j_{\cG_\C}(\Umu \setminus \Cmu; K) \longrightarrow \cdots\ .
$$ 
By definition, a stratification is called equivariantly perfect over the field $K$ if its equivariant Gysin exact sequences with coefficients in $K$ break up into short exact sequences 
$$
0 \longrightarrow H^{j-2\dmu}_{\cG_\C} (\Cmu;K) \longrightarrow H^{j}_{\cG_\C}(\Umu;K) \overset{\mathrm{restr.}}{\longrightarrow} H^j_{\cG_\C}(\Umu \setminus \Cmu; K) \longrightarrow  0\ .
$$ 
This implies that 
$$
P_t^{\cG_\C}(\cC;K) = \sum_{\mu\in\Ird} t^{2d_\mu} P_t^{\cG_\C}(\Cmu;K).
$$ 
As $\cC$ is contractible, $$H^*_{\cG_\C}(\cC;K) \simeq H^*(B\cG_\C;K)$$ and equivariant perfection  also implies that the map $$H^*(B\cG_\C;K)\simeq H^*_{\cG_\C}(\cC;K) \overset{\mathrm{restr.}}{\longrightarrow} H^*_{\cG_\C}(\Css;K),$$ known as the Kirwan map, is surjective. 
Atiyah and Bott showed that the Shatz stratification was perfect over any the fields $K=\bQ$ or $\Z/p\Z$ with $p$ prime, by showing that the composed map

$$\xymatrix{ 
H_{\cG_\C}^{j-2\dmu}(\Cmu;K) \ar[r] \ar@{-->}[dr]_{\cdot \cup e_{\cG_\C}(\Nmu)}
& H_{\cG_\C}^j(\Umu;K) \ar[d]^{\mathrm{restr.}} \\
& H_{\cG_\C}^j(\Cmu;K)
}
$$

\noindent was multiplication by the equivariant Euler class of the normal bundle $\Nmu$ of $\Cmu$ in $\cC$ and that the latter was not a zero divisor in the equivariant cohomology ring $H_{\cG_\C}^*(\Cmu;K)$, forcing the horizontal arrow to be injective. We note that
the equivariant Euler class of $\Nmu$ is, for any field $K$, a well-defined element in $H_{\cG_\bC}^{2d_\mu}(\Cmu;K)$ because the equivariant normal bundle to $\Cmu$, being a complex vector bundle (this follows from Proposition \ref{normal_bundle_AB}), is orientable.
Indeed, the equivariant Euler class of $\Nmu$ is equal to its top equivariant Chern class 
(with coefficients in $K$):
$$
e_{\cG_\bC}(\Nmu) = (c_{d_\mu})_{\cG_\bC}(\Nmu) \in H_{\cG_\bC}^{2\dmu}(\Cmu;K).
$$
The proof that $\euler$ is not a zero divisor goes as follows (\cite{AB}, pp.568-569 and 605-606). 
Denote $\Fmu$ the set of smooth filtrations of type 
$$
\mu=\left(\frac{d_1}{r_1},\cdots,\frac{d_1}{r_1},\cdots, \frac{d_l}{r_l},\cdots,\frac{d_l}{r_l}\right)
$$ 
on the Hermitian vector bundle $E$. By uniqueness of the Harder-Narasimhan filtration, there is a \textit{continuous} map 
$$
\Cmu \longrightarrow \Fmu
$$ 
sending a holomorphic structure of type $\mu$ to the smooth filtration underlying its Harder-Narasimhan filtration (\cite{AB}, sections 14 and 15). Fix now a base point $F_0$ in $\Fmu$, i.e. a smooth filtration of type $\mu$ on $E$, and denote $\Bmu$ the fibre of the map $\Cmu \to \Fmu$ above $F_0$ (=the set of holomorphic structures of type $\mu$ yielding the given smooth filtration $F_0$) and $\cGmu$ the subgroup of $\cG_\C$ preserving $F_0$. Choose, moreover, a splitting of the smooth filtration $F_0$, so that 
$$
E \simeq D_1 \oplus \cdots \oplus D_l
$$ 
with $\rk\, D_i = r_i$ and $\deg\, D_i = d_i$ for all $i$. Denote $\Bmu^0 \subset \Bmu$ the set of holomorphic structures of type $\mu$ which are compatible with the direct sum decomposition above and denote $\cGmu^0 \subset \cGmu$ the subgroup of $\cGmu$ preserving the decomposition of $E$ into a direct sum. Then, as shown by Atiyah and Bott (\cite[Section 7]{AB}), $\Fmu$ is the homogeneous space $\cG_\C / \cGmu$ and 
$$
\Cmu = \cG_\C \times_{\cGmu} \Bmu\, .
$$ 
As a consequence, 
$$
E\cG_{\C} \times_{\cG_\C} \Cmu = E\cG_\C \times_{\cG_\C} (\cG_\C \times_{\cGmu}  \Bmu) = E\cGmu \times_{\cGmu} \Bmu\, ,
$$ and therefore
$$
H^*_{\cG_\C}(\Cmu;K) = H^*_{\cGmu}(\Bmu;K).
$$ 
Moreover, the isomorphism 
$$
E \simeq  D_1 \oplus \cdots \oplus D_l
$$
provides homotopy equivalences $$\cGmu\leadsto \cGmu^0,\ \mathrm{and}\ \Bmu \leadsto\Bmu^0\, ,$$
so
$$
H^*_{\cGmu}(\Bmu;K) = H^*_{\cGmu^0} (\Bmu^0;K).
$$
Finally, as 
$$
\Bmu^0 = \prod_{i=1}^l \Css(D_i)
$$ 
(a holomorphic structure of type $\mu$ on $D_1 \oplus \cdots \oplus D_l$ necessarily is a direct sum of semi-stable structures on each $D_i$) and since, by definition, 
$$
\cGmu^0 \simeq \prod_{i=1}^l \cG_{D_i},
$$  we have~:

\begin{equation*}
P_t^{\cG_\C}(\Cmu;K) = \prod_{i=1}^l P_t^{\cG_{D_i}} \big(\Css(D_i); K\big) = \prod_{i=1}^l P_g(r_i,d_i).
\end{equation*}

\noindent Going back to $H^*_{\cGmu^0}(\Bmu^0;K)$ and fixing a base point $x\in M$, we may consider the group $\cGmu^0(x)$ of gauge transformations of $E=D_1\oplus \cdots \oplus D_l$ which are the identity on the fibre of $E$ at $x$ (the based gauge group). Then $\cGmu^0(x)$ is a normal subgroup of $\cGmu^0$ and 
$$
\cGmu^0 / \cGmu^0(x) \simeq \Kmu := \U(r_1) \times \cdots \times \U(r_l),
$$ 
the structure group of $D_1 \oplus \cdots \oplus D_l$. As $\cGmu^0(x)$ acts freely on $\Bmu^0$, we obtain 
$$
H^*_{\cG_\C}(\Cmu;K) = H^*_{\cGmu^0} (\Bmu^0;K) = H^*_{\Kmu} (\Bmu^0 / \cGmu^0(x) ; K).
$$ 
To simplify the notation, we denote 
$$
\widetilde{\Cmu} := \Bmu^0 / \cGmu^0(x)
$$ 
and $c$ the element of $H^*_{\Kmu}(\widetilde{\Cmu};K)$ corresponding to the equivariant Euler class $e_{\cG_\C}(\Nmu) \in H^*_{\cG_\C}(\Cmu;K)$ under the ring isomorphism above. The following lemma is well-known (\cite[p.605]{AB} and \cite[p.121]{Le_Potier}).

\newcommand{\K}{\mathbf{K}}

\begin{lemma}[Reduction to a maximal torus]\label{reduction_to_T}
Let $K$ be any field. Let $\K$ be a compact connected Lie group, $T$ a maximal torus of $K$ and $i:T \hookrightarrow \K$ be the inclusion map. Then there is an isomorphism of $K$-vector spaces $$H^*_{T}(Y;K) \simeq H^*_{\K}(Y;K) \otimes_{K} H^*(\K/T; K).$$ In particular, the inclusion $i:T \hookrightarrow K$ induces an injective ring homomorphism $$i^* : H^*_{\K}(Y;K) \hookrightarrow H^*_{T}(Y;K)\, .$$ 
\end{lemma}

\noindent Let now $T$ be a maximal torus of $\Kmu=\U(r_1)\times \cdots \times \U(r_l)$ and let us denote $$c' = i^*c \in H^*_T(\widetilde{\Cmu};K)$$ the image of the equivariant Euler class $c$. Since $i^*$ is an injective ring homomorphism, to prove that $c$ is not a zero divisor in $H^*_{\Kmu}(\widetilde{\Cmu};K)$, it suffices to prove that $c'$ is not a zero divisor in $H^*_T(\widetilde{\Cmu};K)$.

\begin{lemma}[\cite{AB} p.605]\label{torus_acting_trivially}
Let $K$ be any field. Let $T=T_0 \times T_1$ be the product of two sub-tori, in such a way that $T_0$ acts trivially on the $T$-space $Y$. Recall that there is an isomorphism of $K$-vector spaces $$H^*_{T}(Y;K) \simeq H^*(BT_0; K) \otimes_{K} H^*_{T_1}(Y;K)\, ,$$ and that $H^*(BT_0;K) \simeq K[x_1, \cdots, x_l]$ is a polynomial ring over a field ($l=\dim T_0$, $\deg\, x_i=2$). Any $\alpha \in H^*_T(Y;K)$ is of the form $$\alpha = \alpha_0 \otimes 1 + \mathrm{terms\ of\ positive\ degree\ in}\ H^*_{T_1}(Y;K)$$ and, if $\alpha_0\neq 0$ in $H^*(BT_0;K)$, then $\alpha_0$ is not a zero divisor in that integral domain and consequently $\alpha$ is not a zero divisor in $H_T^*(Y;K)$. 
\end{lemma}

\noindent We note that, if we fix a point $y_0 \in Y$ and consider the well-defined map 
$$\mu_{y_0}: \begin{array}{ccc}
BT_0 = ET / T_0 & \longrightarrow & Y_{hT} = ET \times_T Y \\
\left[ e\right]& \longmapsto &  \left[e,y_0\right]
\end{array}$$

\noindent then $\alpha_0$ is the pull-back of $\alpha$ under  $\mu_{y_0}$.

\begin{corollary}\label{non_zero_element_cx_case}
If $\alpha \in H^*_T(Y;K)$ satisfies the condition that $$\mu_{y_0}^*\alpha \neq 0\ \mathrm{in}\ H^*(BT_0;K)\, ,$$ then $\alpha$ is not a zero divisor in $H^*_T(Y;K)$.
\end{corollary}

\noindent In order to apply this corollary to the element $c'\in H^*_T(\widetilde{\Cmu};K)$\ (the image of the equivariant Euler class under the natural inclusion of $H^*_{\Kmu}(\widetilde{\Cmu};K)$ in $H^*_T(\widetilde{\Cmu};K)$ given by lemma \ref{reduction_to_T}), we choose a holomorphic structure $A_0\in\Bmu^0$ on 
$$
E=D_1\oplus \cdots \oplus D_l,
$$ 
and denote 
$$
\cE_0 = \cD_1 \oplus \cdots \oplus \cD_l
$$ 
the associated holomorphic vector bundle of type $\mu$. The fibre of $\Nmu$ (the normal bundle to $\Cmu$) at $A_0$ is isomorphic to $$
H^1\big(M;\End''\,\cE_0\big) = \bigoplus_{1\leq i < j \leq l} H^1\big(M;\Hom(\cD_i,\cD_j)\big).
$$ 
Take then $y_0$ to be the image of $A_0$ in $\widetilde{\Cmu}=\Bmu^0/\cGmu^0(x)$ and 
$$
T_0 = \underbrace{\U(1) \times \cdots \times \U(1)}_{l\ \mathrm{times}} = \mathcal{Z}(\Kmu) \subset \Kmu.
$$ 
This torus acts trivially on $\widetilde{\Cmu}$ and we now compute (\cite[pp.568-569]{AB}) $$c'' := \mu_{y_0}^*c'\, .$$ One has $$H^*(BT_0;K) = K[x_1, \cdots, x_l],$$ where $x_i$ is the equivariant Euler class of the $\big(\U(1)\big)^l$-bundle over a point associated to the representation
$$
\begin{array}{ccc}
\U(1) \times \cdots \times \U(1) & \longrightarrow & \mathrm{Aut}(\C) \\
(t_1, \cdots, t_l) & \longmapsto & \mathrm{multiplication\ by}\ t_i.
\end{array}
$$

\noindent The group $T_0$ acts on $\Hom(\cD_i,\cD_j) = \cD_i^* \otimes \cD_j$ by multiplication by $t_i^{-1}t_j$, so it acts on the vector space 
$$
H^1\big(M; \Hom(\cD_i,\cD_j)\big)
$$ 
via the same character. By functoriality of the equivariant Euler class, one has 
$$
\mu_{y_0}^*c' = \prod_{1\leq i < j \leq l} (x_j - x_i)^{\lambda_{ij}},
$$ 
where 
$$
\lambda_{ij} := \dim_{\C} H^1 \big(M; \Hom(\cD_i,\cD_j)\big) = r_i r_j \big(\mu_i - \mu_j + (g-1)\big).
$$ 
In particular, $\mu_{y_0}^*c'  \neq 0$ in $H^*(BT_0;K) = K[x_1, \cdots, x_l]$ so, by Corollary \ref{non_zero_element_cx_case}, the equivariant Euler class of $\Nmu$ is not a zero divisor in $H^*_{\Kmu} (\widetilde{\Cmu};K) = H^*_{\cG_\C}(\Cmu;K).$

\subsection{The induced stratification}\label{induced_stratification}

As the functor $\cE \mapsto \os{\cE}$ preserves the rank and degree of a holomorphic vector bundle, it takes the Harder-Narasimhan filtration of $\cE$ to the Harder-Narasimhan filtration of $\os{\cE}$. In particular, it preserves the Harder-Narasimhan type of a holomorphic bundle, which implies that the Shatz/Morse strata of $\cC$ are invariant under the involution $\alpha_{\tau}:A \mapsto \ov{A}$ of Section \ref{real_and_quaternionic_bundles}. The Galois-invariant part $$\Cmut = \Cmu \cap \Ctau$$ of $\Cmu$ is the set of $\tau$-compatible holomorphic structures of type $\mu$ on $(E,\tau)$. Before stating the next result, we recall that if $\tau: \cE \to \cE$ is a real or quaternionic structure on the holomorphic bundle $\cE$, a sub-bundle $\cF$ of $\cE$ is called real (resp.\ quaternionic) if $\tau(\cF)=\cF$, which means that $\tau$ induces by restriction a real (resp.\ quaternionic) structure on $\cF$. Equivalently, if $\phi: \os{\cE} \overset{\simeq}{\longrightarrow} \cE$ is the isomorphism determined by $\tau$, the condition $\tau(\cF)=\cF$ is equivalent to $\phi(\os{\cF}) = \cF$.

\begin{lemma}\label{max_sub-bundle}
Let $\cE$ be a real (resp.\ quaternionic) holomorphic bundle and let $\cF$ be the unique maximal rank sub-bundle of $\cE$ among sub-bundles of $\cE$ whose slope is maximal. Then $\cF$ is itself real (resp.\ quaternionic).
\end{lemma}

\begin{proposition}\label{real_filtration}
The Harder-Narasimhan filtration of a real (resp.\ quaternionic) holomorphic bundle consists of real (resp.\ quaternionic) sub-bundles.
\end{proposition}

\begin{proof}[Proof of Lemma \ref{max_sub-bundle}]
Let $\phi: \os{\cE} \to \cE$ be the isomorphism determined by the real or quaternionic structure of $\cE$ and let $\cF$ be the unique sub-bundle of $\cE$ satisfying the assumptions of the lemma.  Then $\phi$ induces an isomorphism between $\os{\cF}$ and a sub-bundle of $\cE$ having the same slope and the same rank as $\cF$. By uniqueness of such a sub-bundle, we see that $\phi(\os{\cF}) = \cF$, which proves that $\cF$ is either real or quaternionic, according to the type of $\phi$.
\end{proof}

\begin{proof}[Proof of Proposition \ref{real_filtration}]
This is immediate in view of Lemma \ref{max_sub-bundle} and the fact that if $\cE$ and $\cF$ are both real (resp.\ quaternionic) then so is $\cE / \cF$.
\end{proof}

\noindent We note that, for real bundles (=algebraic bundles defined over $\R$), Proposition \ref{real_filtration} is in fact a special case of a result of Harder and Narasimhan (\cite{HN}). In the complex case considered by Atiyah and Bott, specifying the Harder-Narasimhan type $\mu$ is equivalent to specifying the topological invariants $(r_i,d_i)_{1\leq i \leq l}$ of the successive quotients of the filtration. When $M^{\si}=\emptyset$ or $\tau=\tauH$, there are no further topological invariants, so the definition will be similar, but some care should be taken when it comes to defining the real Harder-Narasimhan type of a real bundle on a curve with real points.

\begin{definition}[Real and quaternionic Harder-Narasimhan types]\label{real_HN_types}
Let $(M,\si)$ be a real algebraic curve and let $(\cE,\tau)$ be a real (resp.\ quaternionic) bundle on $(M,\si)$, with Harder-Narasimhan filtration $$\{0\} = \cE_0 \subset \cE_1 \subset \cdots \subset \cE_l = \cE.$$ The \textbf{real (resp.\ quaternionic) Harder-Narasimhan type} of $(\cE,\tau)$ is the $l$-tuple formed by the topological invariants of the real (resp.\ quaternionic) bundles $\cE_i/\cE_{i-1}$ (the successive quotients of the filtration), namely:
\begin{enumerate}
\item $(r_i,d_i,\vec{w}_i)_{1\leq i \leq l}$ if $\tau$ is real (if $\Msi=\emptyset$, we interpret each $\vw_i$ as $0$, in particular $d_i\ \mod{2}=0$).
\item $(r_i,d_i)_{1\leq i\leq l}$ if $\tau$ is quaternionic,
\end{enumerate}
We denote $\Irdtau$ the set of real (resp.\ quaternionic) Harder-Narasimhan types of $\tau$-compatible holomorphic structures on $E$. In particular, if $\mu\in\Irdtau$,  there is, associated to it, a uniquely defined holomorphic Harder-Narasimhan type, which we also denote $\mu$, and which satisfies $\Cmut := \cC^{\tau} \cap \Cmu\neq\emptyset$.
\end{definition}
\noindent The important thing to realize is that different real Harder-Narasimhan types $(r_i,d_i,\vec{w}_i)_{1\leq i\leq l}$ might occur for a same $(r_i,d_i)_{1\leq i \leq l}$ (the map $\Irdtau \longrightarrow \Ird$ is not injective). This will be useful in practical computations when considering sums over the set $\Irdtau$ of all real Harder-Narasimhan types~: such a sum will be equal to the sum over the set $\Ird$ of all holomorphic Harder-Narasimhan types, multiplied by a factor of $2^{(n-1)(l-1)}$ (where $l$ is the length of the Harder-Narasimhan filtration), corresponding to the choice of topological invariants $(\vec{w}_i)_{1\leq i\leq l}$ of the successive quotients of the filtration (see Subsection \ref{solving_the_recursion} for concrete examples of this).

As a first step towards showing that the stratification 
$$
\cC^{\tau} = \bigsqcup_{\mu\in\Irdtau} \Cmut
$$ 
is equivariantly perfect for the action of the group $\cG_\C^{\, \tau}$, we identify the normal bundle to $\Cmut$ in $\cC^{\tau}$. Let $\cE \in \Cmut$ be a real or quaternionic holomorphic bundle of Harder-Narasimhan type $\mu$. As earlier, we denote $\End\, \cE$ the holomorphic bundle of endomorphisms of $\cE$. The group $\Gal(\C/\R)$ acts on $\End\, \cE = \cE^* \otimes \cE$ by 
$$
\xi \otimes v \longmapsto (\ov{\xi \circ \tau^{-1}}) \otimes \tau(v)
$$ 
(note that we do have an (anti-linear) \textit{involution} of $\End\, \cE$, 
regardless of whether $\tau$ is real or quaternionic, i.e. squares to $+\Id_E$ or $-\Id_E$, meaning that $\End\, \cE$ is always a \textit{real} bundle). This way, the Galois-invariant elements of $\End\, \cE$ are the endomorphisms of $\cE$ commuting to $\tau$. As a consequence of Proposition \ref{real_filtration}, the sub-bundle $\End'\, \cE$, consisting of endomorphisms that preserve the Harder-Narasimhan filtration of $\cE$, is pointwise Galois-invariant for the above action, so the Galois action on $\End\, \cE$ induces an action on the bundle $\End''\, \cE$ defined by the exact sequence 
$$
0 \longrightarrow \End'\,\cE \longrightarrow \End\, \cE\longrightarrow \End''\,\cE \longrightarrow 0\, ,
$$ 
which in turn shows that we have a Galois action on the complex vector space $H^1(M; \End''\,\cE)$. We simply denote $\tau$ the various $\Gal(\C/\R)$-actions that we have defined and, combining the above with Proposition \ref{normal_bundle_AB}, we obtain the following result.

\begin{proposition}\label{normal_bundle_real_case}
$\Cmut$ is a $\cG_\C^{\, \tau}$-invariant, locally closed submanifold of $\Ctau$, of real codimension $\dmu$, and the fibre of the normal bundle to $\Cmut$ at a point $\cE$ is isomorphic to the real vector space $\big(H^1(M;\End''\,\cE)\big)^{\tau}$.
\end{proposition}

\begin{proof}
In view of Proposition \ref{normal_bundle_AB}, it only remains to prove that $\Cmut$ is $\cG_\C^{\, \tau}$-invariant. This follows from the fact that $\Cmu$ is $\cG_\C$-invariant and from the compatibility relation $\ov{g(A)} = \ov{g}(\ov{A})$ for all $A\in\cC$ and all $g\in\cG_\C$.
\end{proof}

\noindent So $\Cmut$ is a $\cG_\C^\tau$-invariant submanifold of finite codimension of $\Ctau$ and 
$$
\Ctau = \bigsqcup_{\mu\in\Irdtau} \Cmut.
$$ 
As $\End\, \cE=\cE^*\otimes \cE$ is always a real bundle, so are the conormal and normal bundles to $\Cmu$ in $\cC$. In particular, the normal bundle to $\Cmut$ in $\Ctau$ is a real vector bundle in the ordinary sense (it is simply the bundle $\Nmut$, whose fibre at $\cE$ is isomorphic to $\big(H^1(M;\End''\,\cE)\big)^{\tau}$, by Proposition \ref{normal_bundle_real_case}). Such a bundle is not orientable in general (see Subsection \ref{orientability}) and this forces us to restrict to cohomology with $\mod{2}$ coefficients. In particular, $\Nmut$ has a well-defined $\cxG^{\tau}$-equivariant, $\mod{2}$ Euler class, which is equal to its top $\cxG^{\tau}$-equivariant Stiefel-Whitney class:
$$
e_{\cxG^{\tau}}(\Nmut) = (w_{\dmu})_{\cxG^{\tau}}(\Nmut) \in H^{\dmu}_{\cxG^{\tau}}(\Cmut; \Z/2\Z).
$$ 
The relation $$\ov{\Cmut} \subset \bigcup_{\mu' \geq \mu} \cC_{\mu'}^{\tau}$$ remains true and we denote $$\Umut:=\bigcup_{\mu' \leq \mu} \cC_{\mu'}^{\tau}.$$ For cohomology with $\mod{2}$ coefficients, the equivariant Thom map is always an isomorphism, so we have 
$$
H^j_{\cxG^{\tau}}(\Umut,\Umut\setminus\Cmut; \Z/2\Z) \simeq H^{j-\dmu}_{\cxG^{\tau}}(\Cmut;\Z/2\Z),
$$ as $\dmu=\mathrm{codim}_{\R}\, \Cmut$ in $\Ctau$. So the associated equivariant Gysin exact sequence is 

$$\xymatrix{ 
\cdots \ar[r] & H_{\cxG^{\tau}}^{j-\dmu}(\Cmut) \ar[r] \ar@{-->}[dr]_{\cdot \cup e_{\cxG^{\tau}}(\Nmut)}
& H_{\cxG^{\tau}}^j(\Umut) \ar[d]^{\mathrm{restr.}} \ar[r]^{\mathrm{restr.}} & H^j_{\cxG^{\tau}}(\Umut \setminus \Cmut) \ar[r] & \cdots \\
& & H_{\cxG^{\tau}}^j(\Cmut) & & 
}
$$

\noindent where $\mod{2}$ coefficients are now understood. The proof that $e_{\cxG^{\tau}}(\Nmut)$ is not a zero divisor in $H^*_{\cxG^{\tau}}(\Cmut)$ then runs parallel to the Atiyah-Bott proof. Denote $\Fmut \subset \Fmu$ the set of $\tau$-invariant smooth filtrations of type $$\mu=\left(\frac{d_1}{r_1},\cdots,\frac{d_1}{r_1},\cdots, \frac{d_l}{r_l},\cdots,\frac{d_l}{r_l}\right)$$ on $(E,\tau)$. Since the Harder-Narasimhan filtration of a $\tau$-compatible holomorphic structure on $E$ consists, by Proposition \ref{real_filtration}, of $\tau$-invariant sub-bundles, there is a continuous map 
$$
\Cmut \longrightarrow \Fmut
$$ 
sending a $\tau$-compatible holomorphic structure to its underlying smooth filtration (note that this map is continuous because the Atiyah-Bott map $\Cmu \longrightarrow \Fmu$ is continuous and sends a $\tau$-compatible holomorphic structure to a smooth filtration by $\tau$-invariant sub-bundles). Let $\Bmut$ denote the fibre of this map above some fixed $\tau$-invariant smooth filtration of $(E,\tau)$. $\Bmut$ is the set of $\tau$-compatible holomorphic structures on $(E,\tau)$ that yield the chosen smooth filtration of $(E,\tau)$. The group of automorphisms of $E$ preserving that filtration is 
$$
\Gmut = \cGmu \cap \cxG^{\tau}.
$$ 
Choose now a splitting of the given smooth filtration of $E$~: 
$$
E= D_1 \oplus \cdots \oplus D_l
$$ 
with $\rk\, D_i=r_i$, $\deg\, D_i=d_i$ and each $D_i$ $\tau$-invariant. The set of $\tau$-compatible holomorphic structures of type $\mu$ that are, in addition, compatible with this direct sum decomposition is 
$$
(\Bmu^0)^{\tau} = \Bmu^0 \cap \Bmut,
$$ 
and the subgroup of $\Gmut$ consisting of automorphisms of $E$ preserving the direct sum decomposition is 
$$
(\cGmu^0)^{\tau} = \cGmu^0 \cap \Gmut.
$$ 
Then $\Fmut$ is the homogeneous space $\cG_\C^\tau / \Gmut$ and 
$$
\Cmut = \cG_\C^\tau \times_{\Gmut} \Bmut\, .
$$ 
As a consequence, 
$$
E\cG_{\C}^\tau \times_{\cG_\C^\tau} \Cmut = E\cG_\C^\tau \times_{\cG_\C^\tau} (\cG_\C^\tau \times_{\Gmut} \Bmut) = E\Gmut \times_{\Gmut} \Bmut\, ,
$$ 
and therefore
$$
H^*_{\cG_\C^\tau}(\Cmut;\Z/2\Z) = H^*_{\Gmut}(\Bmut;\Z/2\Z).
$$ 
Moreover, the splitting 
$$
E \simeq  D_1 \oplus \cdots \oplus D_l
$$
being compatible with $\tau$, there are homotopy equivalences $$\Gmut\leadsto (\Gmut)^0,\ \mathrm{and}\ \Bmut \leadsto(\Bmut)^0\, ,$$
so
$$
H^*_{\Gmut}(\Bmut;\Z/2\Z) = H^*_{(\Gmut)^0} ((\Bmut)^0;\Z/2\Z).
$$
Finally, as 
$$
(\Bmu^0)^{\tau} = \prod_{i=1}^l \Css^{\, \tau}(D_i)
$$ 
(a $\tau$-compatible holomorphic structure of type $\mu$ on $D_1 \oplus \cdots \oplus D_l$ necessarily is a direct sum of $\tau$-compatible semi-stable structures on each $D_i$) and, since 
$$
(\cGmu^0)^{\tau} \simeq \prod_{i=1}^l \cG_{D_i}^{\, \tau},
$$  
we have~:

\begin{equation*}
P_t^{\cxG^{\tau}}(\Cmut;\Z/2\Z) = \prod_{i=1}^l P_t^{\cG_{D_i}^{\, \tau}} \big(\Css^{\, \tau}(D_i);\Z/2\Z\big) = \prod_{i=1}^l P_{(g,n,a)}^{\ \tau}(r_i,d_i).
\end{equation*}

\noindent Going back to $H^*_{(\cGmu^0)^{\tau}}\big((\Bmu^0)^{\tau};\Z/2\Z\big)$ and fixing a base point $x\in M$, we may consider the group $(\cGmu^0)^{\tau}(x)$ of gauge transformations of $E=D_1\oplus \cdots \oplus D_l$ which are the identity on the fibre of $E$ at $x$ (and therefore also at $\si(x)$). Take $x$ such that $\si(x) \neq x$,  and fix $p\in \pi^{-1}(x)$, where
$\pi: P_E\to M$ is the unitary frame bundle of $E\to M$. By the discussion in Section \ref{sec:ev-based}, 
we have a surjective group homomorphism
$$
\ev_p: \cG_\C^\tau \lra  \U(r)
$$
which restricts to a surjective group homomorphism
\begin{equation}\label{eqn:ev-mu}
(\cGmu^0)^\tau \lra \Kmu =\U(r_1)\times\cdots \times \U(r_l). 
\end{equation}
Then~:
\begin{enumerate}
\item $(\cGmu^0)^\tau(x)$ is the kernel of \eqref{eqn:ev-mu}.
\item $(\cGmu^0)^\tau(x)$ is a normal subgroup of $(\cGmu^0)^\tau$ and $(\cGmu^0)^\tau/(\cGmu^0)^\tau(x)\simeq \Kmu$, 
\item $(\cGmu^0)^\tau(x)$ acts freely on $(\Bmu^0)^{\tau}$. 
\end{enumerate}
We obtain  
$$ 
H^*_{\cxG^{\tau}}(\Cmut;\Z/2\Z) = H^*_{(\cGmu^0)^{\tau}} \big((\Bmu^0)^{\tau};\Z/2\Z\big) = H^*_{\Kmu} \big((\Bmu^0)^{\tau} / (\cGmu^0)^{\tau}(x);\Z/2\Z\big).
$$
To simplify the notation, we denote 
$$
\widetilde{\Cmut} := (\Bmu^0)^{\tau} / (\cGmu^0)^{\tau}(x)
$$ 
and $w$ the element of  $H^*_{\Kmu}(\widetilde{\Cmut})$ corresponding to the mod 2 equivariant Euler class 
$$
e_{\cxG^{\tau}}(\Nmut) \in H^*_{\cxG^{\tau}}(\Cmut;\Z/2\Z)
$$ 
under the ring isomorphism above. To prove that $w$ is not a zero divisor in $H^*_{\Kmu}(\widetilde{\Cmut};\Z/2\Z)$, it is enough, by lemma \ref{reduction_to_T}, to prove that \begin{equation}\label{first_injection}w':= i^*w \in H^*_T(\widetilde{\Cmut}; \Z/2\Z)\end{equation} is not a zero divisor in $H^*_T(\widetilde{\Cmut};\Z/2\Z)$, where $T$ is any maximal torus of $\Kmu$. We take that torus to be invariant under the complex conjugation $\tau$ of $\Kmu = \U(r_1) \times \cdots \times \U(r_l)$ (explicitly, we take $T=(S^1)^r$).

\begin{lemma}[Reduction to the real part of the maximal torus]\label{reduction_to_real_part_of_T}
Let $T\simeq (S^1)^r$ be a torus, with complex conjugation $\tau$. Let $j:T^{\tau} \hookrightarrow T$ be the inclusion map and let $Y$ be a $T$-space. Then there is an isomorphism of $(\Z/2\Z)$-vector spaces $$
H^*_{T^\tau}(Y;\Z/2\Z) \simeq H^*_{T}(Y;\Z/2\Z) \otimes_{\Z/2\Z} H^*(T/T^\tau;\Z/2\Z)\, .
$$ 
In particular, the inclusion $j: T^{\tau} \hookrightarrow T$ induces an injective ring isomorphism 
$$
j^*: H^*_T(Y;\Z/2\Z) \hookrightarrow H^*_{T^{\tau}}(Y;\Z/2\Z)\, .
$$ 
\end{lemma}

\noindent In particular, to prove that $w'$ is not a zero divisor in $H^*_T(\widetilde{\Cmut};\Z/2\Z)$, it suffices to prove that \begin{equation}\label{second_injection}w'' := j^*w' \in H^*_{T^{\tau}}(\widetilde{\Cmut};\Z/2\Z)\end{equation} is not a zero divisor in $H^*_{T^{\tau}}(\widetilde{\Cmut}; \Z/2\Z)$.

\begin{proof}[Proof of Lemma \ref{reduction_to_real_part_of_T}]
The proof is similar to that of Lemma \ref{reduction_to_T}, but we give it in detail. Consider the following commutative diagram

$$
\begin{CD}
T/T^{\tau} @>i_Y>> Y_{hT^{\tau}} = ET \times_{T^{\tau}} Y @>p_Y>> Y_{hT} = ET \times_{T} Y \\
@|  @VVprV  @VVprV \\
T/T^{\tau} @>i>> BT^{\tau} = ET/T^{\tau} @>>> BT =ET/T
\end{CD}
$$
\vskip5pt
\noindent Recall that $T\simeq S^1\times \cdots \times S^1$ ($r$ times). So $T^{\tau}\simeq \{\pm 1\} \times \cdots \times \{\pm 1\}$ ($r$ times) and 
$$
H^*(BT^{\tau};\Z/2\Z) \simeq (\Z/2\Z) [y_1, \cdots, y_r]
$$
(with $\deg\, y_i = 1$). Moreover $T/T^{\tau} \simeq (S^1)^r$, so $$H^*(T/T^{\tau};\Z/2\Z) \simeq (\Z/2\Z)[w_1,\cdots, w_r] / <w_i^2 = 0>$$ (with $\deg\, w_i = 1$). The map $$i^* : H^*(BT^{\tau};\Z/2\Z) \longrightarrow H^*(T/T^{\tau};\Z/2\Z)$$ takes $y_i$ to $w_i$, so it is surjective. By the commutativity of the diagram above, $i_Y^*$ is therefore also surjective. Moreover, the $\mod{2}$ cohomology of $T/T^{\tau}\simeq (S^1)^r$ is finite-dimensional in all degrees, so we may apply the Leray-Hirsch Theorem to the locally trivial fibration $$T/T^{\tau} \longrightarrow Y_{hT^{\tau}} \longrightarrow Y_{hT}\, ,$$ and obtain an isomorphism of $(\Z/2\Z)$-vector spaces
$$
H^*_{T^{\tau}}(Y;\Z/2\Z) \simeq  H^*_T(Y;\Z/2\Z) \otimes_{\Z/2\Z} H^*(T/T^{\tau};\Z/2\Z). 
$$
\end{proof}

\noindent It is perhaps worth emphasizing that Lemma \ref{reduction_to_real_part_of_T} uses that 
$$
H^*(BT^{\tau};\Z/2\Z) \simeq (\Z/2\Z)\, [y_1,\cdots, y_r]\, ,
$$ 
which only holds for $\mod{2}$ coefficients. Then we have the following exact analogue of Lemma \ref{torus_acting_trivially} (when $T=\U(1)^r$ is a torus, we call $T^{\tau}=\O(1)^r$ 
the real part of $T$, or, slightly abusively, a real sub-torus).

\begin{lemma}
Let $T^{\tau} = T_0^{\tau} \times T_1^{\tau}$ be the product of two real sub-tori, in such a way that $T_0^{\tau}$ acts trivially on the $T^{\tau}$-space $Y$. Recall that there is an isomorphism of $(\Z/2\Z)$-vector spaces $$H^*_{T^{\tau}}(Y;\Z/2\Z) \simeq H^*(BT_0^{\tau}; \Z/2\Z) \otimes_{\Z/2\Z} H^*_{T_1^{\tau}}(Y;\Z/2\Z)\, ,$$ and that $H^*(BT_0^{\tau};\Z/2\Z) \simeq (\Z/2\Z)\, [y_1, \cdots, y_l]$ is a polynomial ring over a field ($l=\dim T_0$, $\deg\, y_i=1$). Any $\alpha \in H^*_{T^{\tau}}(Y;\Z/2\Z)$ is of the form $$\alpha = \alpha_0 \otimes 1 + \mathrm{terms\ of\ positive\ degree\ in}\ H^*_{T_1^{\tau}}(Y;\Z/2\Z)\, ,$$ and, if $\alpha_0\neq 0$ in $H^*(BT_0^{\tau};\Z/2\Z)$, then $\alpha_0$ is not a zero divisor in this integral domain and consequently $\alpha$ is not a zero divisor in $H_{T^{\tau}}^*(Y;\Z/2\Z)$.
\end{lemma}

\noindent As in the Atiyah-Bott case, we note that, if we fix a point $y_0\in Y$ and consider the map 

$$\mu_{y_0}: \begin{array}{ccc}
BT_0^{\tau} = ET^{\tau} / T_0^{\tau} & \longrightarrow & Y_{hT^{\tau}} = ET^\tau \times_{T^{\tau}} Y \\
\left[ e\right]& \longmapsto &  \left[e,y_0\right]
\end{array}$$

\noindent then $\alpha_0$ is the pull-back of $\alpha$ under  $\mu_{y_0}$.

\begin{corollary}\label{non_zero_element_real_case}
If $\alpha \in H^*_{T^{\tau}}(Y;\Z/2\Z)$ satisfies the condition that $$\mu_{y_0}^*\alpha \neq 0\ \mathrm{in}\ H^*(BT_0^{\tau};\Z/2\Z)\, ,$$ then $\alpha$ is not a zero divisor in $H^*_{T^{\tau}}(Y;\Z/2\Z)$.
\end{corollary}
 
\noindent In order to apply this corollary to the element $w''\in H^*_{T^{\tau}}(\widetilde{\Cmut};\Z/2\Z)$ (the image of the $\mod{2}$ equivariant Euler class $w\in H^*_{\Kmu}(\widetilde{\Cmut};\Z/2\Z)$ under the successive inclusions (\ref{first_injection}) and (\ref{second_injection})), we choose a $\tau$-compatible holomorphic structure $A_0 \in (\Bmu^0)^{\tau}$ on 
$$
E=D_1\oplus \cdots \oplus D_l,
$$ 
and denote 
$$
\cE_0 = \cD_1 \oplus \cdots \oplus \cD_l
$$ 
the associated holomorphic vector bundle of type $\mu$. 
By Proposition \ref{normal_bundle_real_case}, the fibre of $\Nmut$ at $A_0$ is isomorphic to 
$$
\big(H^1\big(M;\End''\,\cE_0\big)\big)^{\tau} = \bigoplus_{1\leq i < j \leq l} \big(H^1\big(M;\Hom(\cD_i,\cD_j)\big)\big)^{\tau}.
$$ 
Take then $y_0$ to be the image of $A_0$ in $\widetilde{\Cmut}=(\Bmu^0)^{\tau}/(\cGmu^0)^{\tau}(x)$,  and let 
$$
T_0^\tau := \{ \pm I_{r_1}\} \times \cdots \times \{ \pm I_{r_l} \} \subset \Kmu,
$$ 
where $I_{r_i}$ denote the identity $r_i\times r_i$ matrix. Then $T_0^\tau \simeq (\Z/2\Z)^l$
acts trivially on $\widetilde{\Cmut}$ and we now compute $$w''':= \mu_{y_0}^*w'' \in H^*(BT_0^{\tau};\Z/2\Z)\, ,$$ where $\mu_{y_0}$ is defined as above, with $Y=\widetilde{\Cmut}$. Moreover,

$$
H^*(BT_0^{\tau};\Z/2\Z) = (\Z/2\Z)[y_1, \cdots, y_l],
$$ 
where $y_i$ is the mod 2 equivariant Euler class of the $\big(\Z/2\Z\big)^l$-bundle over a point associated to the representation
$$
\begin{array}{ccc}
\Z/2\Z \times \cdots \times \Z/2\Z & \longrightarrow & \mathrm{Aut}(\R) \\
(t_1, \cdots, t_l) & \longmapsto & \mathrm{multiplication\ by}\ t_i.
\end{array}
$$

\noindent The group $T_0^{\tau}$ acts on $\big(\Hom(\cD_i,\cD_j)\big)^{\tau} = \big(\cD_i^* \otimes \cD_j\big)^{\tau}$ by multiplication by $t_i^{-1}t_j$, so it acts on the vector space $$\big(H^1\big(M; \Hom(\cD_i,\cD_j)\big)\big)^{\tau}$$ via the same character. By functoriality of the equivariant Euler class, one has $$\mu_{y_0}^*w'' = \prod_{1\leq i < j \leq l} (y_j - y_i)^{\lambda_{ij}},$$ where 

\begin{eqnarray*}
\lambda_{ij} & = & \dim_{\R} \big(H^1 \big(M; \Hom(\cD_i,\cD_j)\big)\big)^{\tau}\\ 
& = & \dim_{\C} H^1 \big(M; \Hom(\cD_i,\cD_j)\big)\\ 
& = & r_i r_j \big(\mu_i - \mu_j + (g-1)\big).
\end{eqnarray*}

\noindent In particular, $\mu_{y_0}^*w'' \neq 0$ in $H^*(BT_0^{\tau};\Z/2\Z) = (\Z/2\Z)[w_1, \cdots, w_l]$ so, by Corollary \ref{non_zero_element_real_case}, the mod 2 equivariant Euler class of $\Nmut$ is not a zero divisor in 
$$
H^*_{\Kmu} (\widetilde{\Cmut};\Z/2\Z) = H^*_{\cxG^{\tau}}(\Cmut;\Z/2\Z).
$$ 
Therefore, we have the following result.

\begin{theorem}
The stratification $$
\Ctau = \bigsqcup_{\mu\in\Irdtau} \Cmut
$$ 
is $\cG_\C^{\, \tau}$-equivariantly perfect over the field $\Z/2\Z$. In particular, 
$$
P_t^{\cG_\C^{\, \tau}} (\Ctau; \Z/2\Z) = \sum_{\mu\in\Irdtau}t^{d_\mu}P_t^{\cG_\C^{\, \tau}}(\Cmut; \Z/2\Z),
$$ 
and the real Kirwan map 
$$
H^*(B\cG_\C^{\, \tau};\Z/2\Z) \longrightarrow H^*_{\cG_\C^{\, \tau}}(\Css^{\, \tau} ; \Z/2\Z)
$$ 
is surjective.
\end{theorem}

\noindent As a consequence, the proof of Theorem \ref{stratification_thm} is now complete. 

\subsection{Orientability of the equivariant normal bundle}\label{orientability}

We use the notation of the previous subsection. 
The $\cxG^{\tau}$-equivariant real vector bundle $\Nmut \to (\Bmu^0)^\tau$ descends to 
a $\Kmu$-equivariant real vector bundle  $V_\mu\to \widetilde{\Cmut}$. Let
$$
w_\mu:= (w_1)_{\Kmu}(\widetilde{\Cmut}) \in H^1_{\Kmu}(\widetilde{\Cmut};\bZ/2\bZ)
$$
be the $\Kmu$-equivariant first Stiefel-Whitney class
of $V_\mu \to \widetilde{\Cmut}$. Then  $(\Nmut)_{h\cxG^{\tau}}$ is an orientable real
vector bundle over the homotopy orbit space $(\Cmut)_{h\cxG^{\tau}}$ if and only if
$w_\mu=0$. The inclusion $T_0^\tau = (\Z/2\Z)^l \hookrightarrow \Kmu$ induces a group homomorphism
\begin{equation}\label{eqn:restrict-to-center}
H^1_{\Kmu}(\widetilde{\Cmut};\bZ/2\bZ) \to H^1_{T_0^\tau}(\widetilde{\Cmut};\bZ/2\bZ).
\end{equation}
$T_0^\tau$ acts trivially on $\widetilde{\Cmut}$ so, for any point $y_0\in \widetilde{\Cmut}$, the inclusion 
$\{y_0\}\subset \widetilde{\Cmut}$ is $T_0^\tau$-equivariant and induces a group
homomorphism
\begin{equation} \label{eqn:restrict-to-point}
H^1_{T_0^\tau}(\widetilde{\Cmut};\bZ/2\bZ) \to H^1_{T_0^\tau}(\{y_0\};\bZ/2\bZ) \simeq \bigoplus_{i=1}^l (\bZ/2\bZ)w_i. 
\end{equation}
Let 
$$
\eta: H^1_{\Kmu}(\widetilde{\Cmut};\bZ/2\bZ)\to \bigoplus_{i=1}^l (\bZ/2\bZ)w_i
$$ 
be the composition of \eqref{eqn:restrict-to-center} and \eqref{eqn:restrict-to-point}. Then
\begin{eqnarray*}
\eta(w_\mu) &=& \sum_{1\leq i<j\leq l} \lambda_{ij}(w_j-w_i) \\
&=& \sum_{1\leq i<j\leq l} (d_i r_j - r_i d_j + r_ir_j(g-1))(w_j-w_i)\\
&=& \sum_{i=1}^l(rd_i + (d+(r-1)(g-1))r_i) w_i  \\
&=& \begin{cases}
(d+g-1) \sum_{i=1}^l r_i w_i & \textup{if $r$ is even},\\
\sum_{i=1}^l(d_i + dr_i)w_i & \textup{if $r$ is odd}.
\end{cases}
\end{eqnarray*}
Suppose that $\tau=\tau_\bR$. If either (i) $r$ and $g+d$ are even, or (ii) $r$ is odd and $r>1$, then
there exists some $\mu$ such that  $(\Nmut)_{h\cxG^{\tau}}$ is a non-orientable real
vector bundle over $(\Cmut)_{h\cxG^{\tau}}$.

We note that the argument above does not give any non-orientability statement when $n>0$ and $\tau=\tau_\bH$.

\section{Betti numbers of moduli spaces of real and quaternionic bundles}

\subsection{A recursive formula for the equivariant Poincar\'e series}

We quickly summarize our results, using the same notation as in Section \ref{intro}. $(M,\si)$ is a Klein surface of topological type $(g,n,a)$ and $(E, \tau)$ is a real (resp.\ quaternionic) Hermitian bundle of rank $r$ and degree $d$ on $(M,\si)$. The  group of $\tau$-compatible endomorphisms  of $E$ is denoted $\cG_\C^\tau$. We consider the set $$\cM_{(g,n,a)}^{\ \tau}(r,d) = \Csst \quot \cG_\C^\tau = (\fibre)^{\tau} / \cG_E^{\tau}$$ of real (resp.\ quaternionic) $S$-equivalence classes of $\tau$-compatible, semi-stable holomorphic structures on $E$. The $\cG_\C^\tau$-equivariant $\mod{2}$ Poincar\'e series of $\Csst$ is computed recursively via the following formula 
$$
P_{(g,n,a)}^{\ \tau}(r,d) \quad =\quad  Q_{(g,n,a)}^{\ \tau}(r) \quad - \sum_{\mu\in\Irdtaup} t^{\dmu} \prod_{i=1}^l P_{(g,n,a)}^{\ \tau}(r_i,d_i)\ ,
$$ 
where the sum ranges over all possible real (resp.\ quaternionic) Harder-Narasimhan types $\mu=(r_i,d_i,\vec{w}_i)_{1\leq i\leq l}$ (resp.\ $(r_i,d_i)_{1\leq i\leq l}$, see Definition \ref{real_HN_types})
of $\tau$-compatible, \textit{non semi-stable} holomorphic structures on $E$ and 
$$
\dmu = \sum_{1\leq i < j \leq l} r_i r_j \left(\frac{d_i}{r_i} -\frac{d_j}{r_j} + (g-1)\right).
$$ 
The expression for $Q_{(g,n,a)}^{\ \tau}(r)=P(B\cG_\C^\tau)$ is given in Theorem \ref{classifying_space}. In particular, 
$$
P_{(g,n,a)}^{\ \tau}(1,d) = \frac{(1+t)^{g+1}}{1-t^2} = \frac{(1+t)^g}{1-t}\, ,
$$ 
which is consistent with the fact that $\cM_{(g,n,a)}^{\ \tau}(1,d)$ is, for all $\tau$, a real torus $\R^g / \Z^g \subset \mathrm{Pic}_{M,\si}^{\ d}(\C) \simeq \C^g/\Z^{2g}$.

\subsection{Solving the recursion}\label{solving_the_recursion}

As one might expect by analogy with the Harder-Narasimhan-Desale-Ramanan/Atiyah-Bott recursive formula (\cite{HN,DR} and \cite{AB}), Zagier's method to solve the recursion (\cite{Zagier}) carries 
over to the real and quaternionic cases. We have
$$
Q_{(g,n,a)}^{\ \tau}(r) =\sum_{\mu\in \Irdtau} t^{d_\mu} \prod_{i=1}^l P_{g,n,a}^\tau(r_i,d_i),
$$
where 
$$
\mu=\Bigl(\frac{d_1}{r_1},\ldots, \frac{d_l}{r_l}\bigr),
$$
and 
$$
d_\mu =\sum_{1\leq i<j\leq l} (d_i r_j - r_i d_j + r_i r_j (g-1)) = \sum_{1\leq i<j\leq l}(d_i r_j - r_i d_j) + \frac{g-1}{2}(r^2 - \sum_{i=1}^l r_i^2).
$$
Define
\begin{eqnarray*}
\bad_\mu &=&  \sum_{1\leq i<j\leq l} (d_i r_j -r_i d_j), \\
\baQ_{(g,n,a)}^{\ \tau}(r)&=& t^{-r^2(g-1)/2} Q^{\ \tau}_{(g,n,a)}(r),\\
\baP_{(g,n,a)}^{\ \tau}(r,d) &=& t^{-r^2(g-1)/2} P_{(g,n,a)}^{\ \tau}(r,d) .
\end{eqnarray*}
Then 
$$
\baQ_{(g,n,a)}^{\ \tau}(r) =\sum_{\mu\in \Irdtau} t^{\bad_\mu} \prod_{i=1}^l \baP_{(g,n,a)}^{\ \tau}(r_i,d_i).
$$
In order to apply Zagier's theorem, we need sums over $\Ird$ (see Definition \ref{real_HN_types} and the discussion thereafter for how to relate sums over $\Irdtau$ to sums over $\Ird$). Taking into account the topological invariants of the real (resp.\ quaternionic) bundles given by the successive quotients of the Harder-Narasimhan filtration, we obtain
\begin{eqnarray*}
\baQ_{(g,0,1)}^{\ \tauR}(r) &=& \sum_{\mu\in I_{r,d}} t^{2\bad_\mu} \prod_{i=1}^l \baP_{(g,0,1)}^{\tauR}(r_i, 2d_i)\, , \\
\baQ_{(2g'-1,0,1)}^{\ \tauH}(r) &=& \sum_{\mu\in I_{r,d}} t^{2\bad_\mu}\prod_{i=1}^l \baP_{(2g'-1,0,1)}^{\tauH}(r_i, 2d_i),\\
\baQ_{(2g',0,1)}^{\ \tauH}(r) &=& \sum_{\mu\in I_{r,d}} t^{2\bad_\mu}\prod_{i=1}^l \baP_{(2g',0,1)}^{\tauH}(r_i, 2d_i+r_i)
\end{eqnarray*}
For $n>0$, we have
\begin{eqnarray*}
2^{n-1} \baQ_{(g,n,a)}^{\ \tauR}(r) &=& 2^{n-1} \sum_{\mu\in I_{r,d}} t^{\bad_\mu} 2^{(n-1)(l-1)}\prod_{i=1}^l\baP_{(g,n,a)}^{\tauR}(r_i, d_i) \\ 
& = &
\sum_{\mu\in I_{r,d}} t^{\bad_\mu} \prod_{i=1}^l(2^{n-1}\baP_{(g,n,a)}^{\tauR}(r_i, d_i))\, , \\
\baQ_{(g,n,a)}^{\ \tauH}(2r) &=& \sum_{\mu\in I_{r,d}} t^{4\bad_\mu}\prod_{i=1}^l \baP_{(g,n,a)}^{\tauH}(2r_i, 2d_i). 
\end{eqnarray*}

\begin{theorem}[{Zagier, \cite[Theorem 2]{Zagier}}] \label{thm:zag}
Let $Q_r$ and $P_{r,d}$ ($r\in \bZ$, $d\in \bZ/r\bZ$) be elements of a not necessarily commutative algebra 
over the field of formal power series $\bQ((x))$ which are related by 
$$
Q_r =\sum_{\mu\in I_{r,d}} x^{\bad_\mu} P_{r_1,d_1} \cdots P_{r_l,d_l}.
$$
Then for any $r$ and $d$, we have
$$
P_{r,d} =\sum_{l=1}^r \sum_{\substack{r_1,\ldots, r_l>0 \\ r_1+\cdots + r_l =r}}
\frac{(-1)^{l-1} x^{M(r_1,\ldots,r_l;\frac{d}{r})} }{\prod_{i=1}^{l-1}(1-x^{r_i+r_{i+1}})} Q_{r_1}\cdots Q_{r_l},
$$
where 
$$
M(r_1,\ldots,r_l;\lambda) =\sum_{i=1}^{l-1} (r_i+r_{i+1})\langle(r_1+\cdots+ r_i)\lambda\rangle.
$$
Here $\langle x\rangle = 1+ [x]-x$ for a real number $x$ denotes the unique $t\in (0,1]$
with $x+t\in \bZ$. 
\end{theorem}

\begin{theorem}\label{closed-formulae}
One has~:
\begin{enumerate}
\item \begin{eqnarray*}
& & P^{\ \tauR}_{(g,0,1)}(r,2d)\\
&=& \sum_{l=1}^r\sum_{\substack{r_1,\ldots, r_l\in\bZ_{>0}\\ \sum r_i=r}}
(-1)^{l-1}  \frac{t^{2\sum_{i=1}^{l-1} (r_i+r_{i+1}) \langle (r_1+\cdots + r_i)(\frac{d}{r})\rangle} }
{\prod_{i=1}^{l-1}(1-t^{2(r_i+r_{i+1})} ) } t^{(g-1)\sum_{i<j} r_i r_j}
\\
& & \qquad \qquad \qquad \quad \prod_{i=1}^l \frac{\prod_{j=1}^{r_i} (1+t^{2j-1})^{g+1} }
{\prod_{j=1}^{r_i-1}(1-t^{2j}) \prod_{j=1}^{r_i}(1-t^{2j})}\cdot
\end{eqnarray*}

\item \begin{eqnarray*}
& & P^{\ \tauH}_{(2g'-1,0,1)}(r,2d) \\
& = & \sum_{l=1}^r\sum_{\substack{r_1,\ldots, r_l\in\bZ_{>0}\\ \sum r_i=r}}
(-1)^{l-1} \frac{t^{2\sum_{i=1}^{l-1} (r_i+r_{i+1}) \langle (r_1+\cdots + r_i)(\frac{d}{r})\rangle} }
{\prod_{i=1}^{l-1}(1-t^{2(r_i+r_{i+1})} ) } t^{ (2g'-2)\sum_{i<j} r_i r_j}\\
& & \qquad \qquad \qquad \quad  \prod_{i=1}^l \frac{\prod_{j=1}^{r_i} (1+t^{2j-1})^{2g'} }
{ \prod_{j=1}^{r_i-1}(1-t^{2j}) \prod_{j=1}^{r_i}(1-t^{2j}) }\cdot
\end{eqnarray*}

\item \begin{eqnarray*}
& & P^{\ \tauH}_{(2g',0,1)}(r,2d+r) \\
& = & \sum_{l=1}^r\sum_{\substack{r_1,\ldots, r_l\in\bZ_{>0}\\ \sum r_i=r}}
(-1)^{l-1} \frac{t^{2\sum_{i=1}^{l-1} (r_i+r_{i+1}) \langle (r_1+\cdots + r_i)(\frac{d}{r})\rangle} }
{\prod_{i=1}^{l-1}(1-t^{2(r_i+r_{i+1})} ) } t^{(2g'-1)\sum_{i<j} r_i r_j}\\
& & \qquad \qquad \qquad \quad  \prod_{i=1}^l \frac{\prod_{j=1}^{r_i} (1+t^{2j-1})^{2g'+1} }
{ \prod_{j=1}^{r_i-1}(1-t^{2j}) \prod_{j=1}^{r_i}(1-t^{2j}) }\cdot
\end{eqnarray*}

\item Suppose that $n>0$. Then
\begin{eqnarray*}
&& P^{\ \tauR}_{(g,n,a)}(r,d)\\
&=& \sum_{l=1}^r\sum_{\substack{r_1,\ldots, r_l\in\bZ_{>0}\\ \sum r_i=r}}
(-1)^{l-1}\frac{t^{\sum_{i=1}^{l-1} (r_i+r_{i+1}) \langle (r_1+\cdots + r_i)(\frac{d}{r})\rangle}}
{\prod_{i=1}^{l-1}(1-t^{r_i+r_{i+1}} ) } t^{(g-1)\sum_{i<j} r_i r_j}
\\
& & \qquad \quad   2^{(n-1)(l-1)} \prod_{i=1}^l \frac{\prod_{j=1}^{r_i} (1+t^{2j-1})^{g-n+1} \prod_{j=1}^{r_i-1}(1+t^j)^{n} \prod_{j=1}^{r_i}(1+t^j)^n }
{ \prod_{j=1}^{r_i-1}(1-t^{2j}) \prod_{j=1}^{r_i}(1-t^{2j}) }\cdot
\end{eqnarray*} 

\item Suppose that $n>0$. Then
\begin{eqnarray*}
&& P^{\ \tau_\bH}_{(g,n,a)}(2r,2d)\\
&=& \sum_{l=1}^{r}\sum_{\substack{r_1,\ldots, r_l\in\bZ_{>0}\\ \sum r_i=r}}
(-1)^{l-1} \frac{t^{4\sum_{i=1}^{l-1} (r_i+r_{i+1}) \langle (r_1+\cdots + r_i)(\frac{d}{r})\rangle} }
{\prod_{i=1}^{l-1}(1-t^{4(r_i+r_{i+1})} ) } t^{4(g-1)\sum_{i<j} r_i r_j}\\
& & \qquad \qquad \qquad \quad  \prod_{i=1}^l \frac{\prod_{j=1}^{2r_i} (1+t^{2j-1})^g }
{ \prod_{j=1}^{r_i-1}(1-t^{4j})\prod_{j=1}^{r_i}(1-t^{4j}) }\cdot
\end{eqnarray*}
\end{enumerate}
\end{theorem}

\begin{proof}
We apply Theorem \ref{thm:zag} with
\begin{enumerate}
\item $Q_r=\baQ_{(g,0,1)}^{\ \tauR}(r)$, $P_{r,d}=\baP^{\ \tauR}_{(g,0,1)}(r,2d)$, $x=t^2$.
\item $Q_r=\baQ_{(2g'-1,0,1)}^{\ \tauH}(r)$, $P_{r,d}=\baP^{\ \tauH}_{(2g'-1,0,1)}(r,2d)$, $x=t^2$.
\item $Q_r=\baQ_{(2g',0,1)}^{\ \tauH}(r)$, $P_{r,d}=\baP^{\ \tauH}_{(2g',0,1)}(r,2d+r)$, $x=t^2$.
\item $Q_r=2^{n-1}\baQ_{(g,n,a)}^{\ \tauR}(r)$, $P_{r,d}=2^{n-1}\baP^{\ \tauR}_{(g,n,a)}(r,d)$, $x=t$.
\item $Q_r=\baQ_{(g,n,a)}^{\ \tauH}(2r)$, $P_{r,d}=\baP^{\ \tauH}_{(g,n,a)}(2r,2d)$, $x=t^4$.
\end{enumerate}
\end{proof}

\noindent
In the Appendix, we use Theorem \ref{closed-formulae} to compute
$P_{(g,n,a)}^{\ \tau}(r,d)$ explicitly for $r\leq 4$. 

\bigskip

\noindent 
It is interesting to note the following equalities.

\begin{corollary}\label{equal}
One has~:
\begin{enumerate}
\item[(a)] $P^{\ \tauR}_{(2g'-1,0,1)}(r,2d) = P^{\ \tauH}_{(2g'-1, 0,1)}(r,2d) = P_{g'}(r,d)$.
\item[(b)] $P^{\ \tauR}_{(2g',0,1)}(r,2d) = P^{\ \tauH}_{(2g', 0,1)}(r,2d+r)$. 
\end{enumerate}
\end{corollary}

\noindent We now give geometric proofs of the first equality of (a) and the equality in (b). 
\begin{proof}
On a real curve of type $(2g'-1,0,1)$, there exists a quaternionic line bundle
$\cL^\H$ of degree 0. Let $L^\H$ denote the underlying topological
line bundle.  Then there is a group isomorphism
$$
\phi: \cG^{\tauR}_\C\stackrel{\simeq}{\lra} \cG^{\tauH}_\C, \quad u\mapsto u\otimes \mathrm{Id}_{L^\H }.
$$
There is a homeomorphism
$$
i : \Css(r,2d)^{\tauR} \stackrel{\simeq}{\lra} \Css(r,2d)^{\tauH},\quad \cE\mapsto  \cE\otimes \cL^{\H}
$$
which is equivariant with respect to the $\cG_\bC^{\tauR}$-action on
$\Css(r,2d)^{\tauR}$ and the $\cG_\bC^{\tauH}$-action on $\Css(r,2d)^{\tauH}$:
$$
i(u\cdot A) = \phi(u)\cdot i(A).
$$
This implies the first equality in (a).
There is also a homeomorphism
$$
\cM^{\ \tauR}_{(2g-1,0,1)}(r,2d) =\cA_{\min}(r,2d)^{\tauR}/\cG_E^{\tauR} \simeq
\cM^{\ \tauH}_{(2g-1,0,1)}(r,2d)=\cA_{\min}(r,2d)^{\tauH}/\cG_E^{\tauH}. 
$$

On a real curve of type $(2g',0,1)$, there exists a quaternionic line bundle
$\cL^\H$ of degree 1. Let $L^\H$ denote the underlying topological
line bundle.  Then there is a group isomorphism
$$
\phi: \cG^{\tauR}_\C\stackrel{\simeq}{\lra} \cG^{\tauH}_\C, \quad u\mapsto u\otimes \mathrm{Id}_{L^\H }.
$$
There is a homeomorphism
$$
i : \Css(r,2d)^{\tauR} \stackrel{\simeq}{\lra} \Css(r,2d+r)^{\tauH},\quad \cE\mapsto  \cE\otimes \cL^{\H}
$$
which is equivariant with respect to the $\cG_\bC^{\tauR}$-action on
$\Css(r,2d)^{\tauR}$ and the $\cG_\bC^{\tauH}$-action on $\Css(r,2d+r)^{\tauH}$:
$$
i(u\cdot A) = \phi(u)\cdot i(A).
$$
This implies (b).
There is also a homeomorphism
\begin{eqnarray*}
& & \cM^{\ \tauR}_{(2g',0,1)}(r,2d) = \cA_{\min}(r,2d)^{\ \tauR}/\cG_E^{\tauR}\\
& \simeq & \cM^{\ \tauH}_{(2g',0,1)}(r,2d+r) = \cA_{\min}(r,2d+r)^{\ \tauH}/\cG_E^{\tauH}\, . 
\end{eqnarray*}
\end{proof}

\subsection{Cohomology of the moduli space in the coprime case}\label{smooth_case}

When $r$ and $d$ are coprime, the topological space $\cM_{(g,n,a)}^{\ \tau}(r,d)$, being a connected component of the fixed locus of an involutive isometry in a smooth compact manifold of real dimension $2(r^2(g-1)+1)$, is a smooth compact connected manifold of real dimension $r^2(g-1) + 1$. We will show in this subsection that its $\mod{2}$ Poincar\'e polynomial is related to $P_{(g,n,a)}^{\ \tau}(r,d)$ in the following way~: 
\begin{equation}\label{actual_Poincare_polynomial}
P_t\big(\cM_{(g,n,a)}^{\ \tau}(r,d);\bZ_2\big) = (1-t)\,P_{(g,n,a)}^{\ \tau}(r,d).
\end{equation}
In particular, when $r \wedge d = 1$, the formal series $(1-t)P_{(g,n,a)}^{\ \tau}(r,d)$ will be a \textit{polynomial} of degree $r^2(g-1) + 1$ which satisfies ($\mod{2}$) Poincar\'e duality~: 
\begin{equation}\label{mod_2_Poincare_duality}
t^{r^2(g-1) + 1} P_{\frac{1}{t}} = P_t.
\end{equation}
We note that relation \eqref{mod_2_Poincare_duality} can indeed be tested directly on the right-hand-side of \eqref{actual_Poincare_polynomial} in the coprime case. Note that \eqref{mod_2_Poincare_duality} does \textit{not} imply, in general, that $P_t$ is a polynomial. To show that \eqref{actual_Poincare_polynomial} holds in the coprime case, we shall in fact prove the following stronger result.

\begin{theorem}\label{comparison_of_Poincare_series_in_good_cases}
Let $(E,\tau)$ be a real (resp.\ quaternionic) vector bundle of rank $r$ and degree $d$
over a Klein surface $(M,\si)$of topological type $(g,n,a)$. We denote, as usual, $\cG_\C^\tau$  the group of complex linear transformations of $E$ that commute to $\tau$ and we let $\ov{\cG_\C^\tau}=\cG_\C^\tau/\R^*$ be the quotient of $\cG_\C^\tau$ by its centre. Then, if either
\begin{enumerate}
\item $n>0$ and $r$ and $d$ are not both even, or
\item $n=0$ and $r$ is odd,
\end{enumerate}
then, for any $\overline{\cG^\tau_\C}$-space $X$, one has
$$
P_t^{\overline{\cG^\tau_\C}}(X;\bZ_2)=(1-t)P_t^{\cG^\tau_\C}(X;\bZ_2).
$$
\end{theorem}

We observe that, in case (1), we are necessarily dealing with a real bundle $(E,\tau)$, for if $n>0$ and $(E,\tau)$ is quaternionic, then $r$ is even, which in turn forces $d$ to be even by Theorem \ref{top_type_of_bundles}.  Case (1) of Theorem \ref{comparison_of_Poincare_series_in_good_cases} has also been treated in \cite[Lemma 7.1]{Baird} using a different method.

\begin{corollary}\label{cohom_of_moduli_space}
If $r\wedge d=1$, then $\ov{\cG_\C^\tau}$ acts freely on $\Csst$, with quotient $\cM_{(g,n,a)}^{\ \tau}(r,d)$, and the $\mod{2}$ Poincar\'e polynomial of $\cM_{(g,n,a)}^{\ \tau}(r,d)$ is 
\begin{equation}\label{equalities_in_the_coprime_case}
P_t\big(\cM_{(g,n,a)}^{\ \tau}(r,d);\bZ_2\big) = P_t^{\overline{\cG^\tau_\C}}(\Csst;\bZ_2) = (1-t)P_t^{\cG_\C^\tau}(\Csst;\bZ_2).
\end{equation}
\end{corollary}

\begin{proof}[Proof of Corollary \ref{cohom_of_moduli_space}]
When $r\wedge d=1$, every semi-stable real (resp.\ quaternionic) bundle is in fact stable so, on the one hand, the space $\cM_{(g,n,a)}^{\ \tau}(r,d)$, which by definition is the space of real (resp.\ quaternionic) $S$-equivalence classes of $\tau$-compatible semi-stable holomorphic structures on a smooth real (resp.\ quaternionic) bundle $(E,\tau)$, is in fact equal to the orbit space $\Csst / \ov{\cG_\C^\tau}$ and, on the other hand, the action of $\ov{\cG_\C^\tau}$ on $\Csst$ is free. This gives the first equality in \eqref{equalities_in_the_coprime_case}. To prove the second equality, assume first that $n>0$. Since $r$ and $d$ are coprime, they cannot be both even and case (1) of Theorem \ref{comparison_of_Poincare_series_in_good_cases} therefore gives the second equality. Assume now that $n=0$. If $(E,\tau)$ is real, then $d$ is even by Theorem \ref{top_type_of_bundles}, so $r$ is odd by the coprimality assumption and case (2) of Theorem \ref{comparison_of_Poincare_series_in_good_cases} applies. Finally, if $(E,\tau)$ is quaternionic and its rank is even, then by Theorem \ref{top_type_of_bundles} its degree is also even, thus contradicting the coprimality assumption. So the rank is odd in this case and we can again apply case (2) of Theorem \ref{comparison_of_Poincare_series_in_good_cases}.
\end{proof}

\noindent The remainder of this subsection is devoted to the proof of Theorem \ref{comparison_of_Poincare_series_in_good_cases}. We will work throughout with the groups $\cGt$ and $\O(1)$, as opposed to $\cG_\C^{\tau}$ and $\R^*$. This is alright because, as we have noted in Section \ref{notation}, these give the same equivariant cohomology for any $\cG_\C^{\tau}$-space $X$. The key result to prove Theorem \ref{comparison_of_Poincare_series_in_good_cases} is the following proposition.

\begin{proposition}\label{cohom_of_a_rigidified_stack_in_good_cases}
Let $(E,\tau)$ be a real or quaternionic vector bundle  over a Klein surface and 
let $$\iota: \begin{array}{ccc}\O(1) &\lra & \cGt\\ -1& \longmapsto & -I_E \end{array}$$ be the inclusion of the centre. If the induced group homomorphism 
$$\iota_*: \pi_0(\O(1))\to \pi_0(\cGt)$$ is injective,
then for any $\overline{\cGt}$-space $X$, 
$$
P_t^{\overline{\cGt}}(X;\bZ_2) =  (1-t) P_t^{\cGt}(X;\bZ_2).$$
\end{proposition}

\noindent The proof of Proposition \ref{cohom_of_a_rigidified_stack_in_good_cases}
requires the following intermediate results, which are extracted from \cite[Section 6]{BHH}.

\begin{theorem}[{\cite{BHH}}]\label{pi-one-BcG}
Let $(E,\tau)$ be a real or quaternionic vector bundle  over a Klein surface of $(g,n,a)$. Then
$$
\pi_0(\cG_E^{\, \tau})=\pi_1(B(\cG_E^{\, \tau}))=
\begin{cases}
\bZ^g \oplus \bZ_2, & n=0,\\
\bZ^g\oplus \bZ_2^{n+1}, & \tau=\tau_\bR, n>0, r>2,\\
\bZ^{g-n+1} \oplus(\bZ^n\rtimes (\bZ^{n-1}\oplus \bZ_2)), & \tau=\tau_\bR, n>0, r=2,\\
\bZ^g \oplus \bZ_2, & \tau=\tau_\bR, n>0, r=1,\\
\bZ^g, & \tau=\tau_\bH, n>0.
\end{cases}
$$
where  the semi-direct product $\bZ^n\rtimes (\bZ^{n-1}\oplus \bZ_2)$ is defined by 
\begin{eqnarray*}
(\bZ^{n-1}\oplus \bZ_2)\times \bZ^n &\lra& \bZ^n, \\
((a_2,\ldots, a_r, b), (c_1,\ldots, c_r)) &\longmapsto & ((-1)^b c_1, (-1)^{b+a_2}c_2,\ldots, (-1)^{b+a_r}c_r).
\end{eqnarray*}
\end{theorem}

Taking the abelianization of $\pi_1(B(\cG_E^{\, \tau}))$, we get the following immediate
corollary of Theorem \ref{pi-one-BcG}.
\begin{corollary}\label{Hone-BcG}
$$
H_1(B(\cG_E^{\, \tau});\bZ)=
\begin{cases}
\bZ^g \oplus \bZ_2, & n=0,\\
\bZ^g\oplus \bZ_2^{n+1}, & \tau=\tau_\bR, n>0, r>1,\\
\bZ^g \oplus \bZ_2, & \tau=\tau_\bR, n>0, r=1,\\
\bZ^g, & \tau=\tau_\bH, n>0.
\end{cases}
$$
\end{corollary}

\begin{proof}[Proof of Proposition \ref{cohom_of_a_rigidified_stack_in_good_cases}] The map $\pi_0(\O(1))\to \pi_0(\cGt)$ can be identified with
the group  homomorphism $\pi_1(B\O(1))\to \pi_1(B\cGt)$. 
By Theorem \ref{pi-one-BcG} and Corollary \ref{Hone-BcG}, if $\pi_1(B\O(1))=\bZ_2\to \pi_1(B\cGt)$ is injective,
then $H_1(B\O(1);\bZ)=\bZ_2\to H_1(B\cG^\tau;\bZ)$ is a direct summand. Therefore, 
\begin{equation} \label{eqn:Hom-Hone}
\Hom(H_1(B\cGt;\bZ);\bZ_2)\lra \Hom(H_1(B\O(1);\bZ);\bZ_2)
\end{equation}
is surjective. By the universal coefficient theorem, for $G=\O(1), \cGt$,
there is a short exact sequence of groups
$$
0\lra \Ext(H_0(BG;\bZ),\bZ_2 )\lra H^1(BG;\bZ_2)\lra \Hom(H_1(BG;\bZ),\bZ_2)\lra 0,
$$
where $H_0(BG;\bZ)=\bZ$, so $\Ext(H_0(BG;\bZ);\bZ_2)=0$. So \eqref{eqn:Hom-Hone}
can be identified with the map
$$
j^*: H^1(B\cGt;\bZ_2)\lra H^1(B\O(1);\bZ_2)=\bZ_2 x 
$$
as $\bZ_2$-linear maps between vector spaces over $\bZ_2$.  In particular,
$j^*$ is surjective,  so there exists $y\in H^1(B\cGt;\bZ_2)$ such that $j^*(y)=x$.

For any $\overline{\cGt}$-space $X$, consider the following commutative diagram:
$$
\begin{CD}
B\O(1) @>{j_X}>> E(\cGt) \times_{\cGt} X  @>>> 
E(\overline{\cGt})\times_{\overline{\cGt} } X \\
@V{i}VV @V{\pi}VV @VVV \\
B\O(1) @>{j}>> B(\cGt) @>>> B(\overline{\cGt})
\end{CD}
$$
where the rows are fibrations, $i:B\O(1)\lra B\O(1)$ is the identity map,
and $j$, $j_X$ are inclusions of a fibre. 
Let $y'=\pi^* (y)\in H^1(E(\cGt)\times_{\cGt}X;\bZ_2)$. Then
$j_X^* (y')=x$, which generates $H^*(B\O(1);\bZ_2)=\bZ_2[x]$. Therefore, 
the mod 2 Leray-Serre spectral sequence associated
to the fibration in the top row collapses at the $E_2$ page, which implies
$$
P_t^{\cGt}(X;\bZ_2) = P_t(B\O(1);\bZ_2)P_t^{\overline{\cGt}}(X;\bZ_2)
=\frac{1}{1-t} P_t^{\overline{\cGt}}(X;\bZ_2).
$$
\end{proof}

\noindent Because of Proposition \ref{cohom_of_a_rigidified_stack_in_good_cases}, in order to prove Theorem \ref{comparison_of_Poincare_series_in_good_cases}, it suffices to proves that if either assumption (1) or (2) in Theorem \ref{comparison_of_Poincare_series_in_good_cases} holds, then the map $$\iota_*: \pi_0(\O(1))\lra \pi_0(\cGt)$$ of Proposition \ref{cohom_of_a_rigidified_stack_in_good_cases} is injective.

We start with the case $n>0$. Then, as we have already noted, the assumption that $r$ and $d$ are not both even implies that we are dealing with a real bundle in this case. We then have the following lemma, describing the situation over a single circle for a real vector bundle in the usual sense (fibre $\R^r$).

\begin{lemma} \label{loop}
Let $\gamma\simeq S^1$. Let $E\to\gamma$ be a smooth real vector bundle (with fibre $\bR^r$) equipped with an inner product $h$, and let $\cG_E =\mathrm{Aut}(E,h)$ be the group of gauge transformations compatible with the metric.
Let $I_E:E\to E$ be the identity map.  
Then $I_E$ and $-I_E$  are in the same connected component of $\cG_E$ if and only 
if $r$ is even and $E$ is trivial (or equivalently, $w_1(E)=0$).
\end{lemma}
\begin{proof}
Suppose that $E$ is isomorphic to the product bundle $\gamma \times \bR^r$, where $r=2r'$ is even.
Let $e_1,\ldots, e_{2r'}$ be the standard
basis of $\bR^{2r}$. Define $u_t: \gamma \times \bR^r\to \gamma \times \bR^r$ by 
$$
e_{2i-1}\longmapsto \cos(\pi t) e_{2i-1} + \sin(\pi t) e_{2i},\quad
e_{2i}\longmapsto -\sin(\pi t) e_{2i-1} + \cos(\pi t) e_i.
$$
Then $t\mapsto u_t$ defines a path $[0,1]\to \cG_E$ such that
$u_0=I_E$ and $u_1=-I_E$. 

Conversely, suppose that $I_E$ and $-I_E$ are in the same connected component of $\cG_E$.
Then there exists a path $[0,1]\to \cG_E$, $t\mapsto u_t$, such that $u_0=I_E$ and $u_1=-I_E$.
We define a vector bundle $\pi:\tilde{E}\to \gamma \times S^1 \times [0,1]$ as follows.
\begin{enumerate}
\item The total space $\tilde{E}$ is the quotient
of 
$$
E\times [0,1]\times [0,1]=\{(x,v,s,t): x\in \gamma, v\in E_x, s, t\in [0,1]\}
$$
by the equivalence relation $(x,v,0,t)\sim (x, u_t(x)v, 1, t)$.
\item The base $\gamma \times S^1\times [0,1]$ is the quotient of 
$$
\gamma\times[0,1] \times [0,1]= \{ (x, s,t): x\in\gamma, s, t,\in [0,1]\}
$$
by the equivalence relation $(x,0,t)\sim (x,1,t)$.

\item The projection $\pi$ is given by $(x,v, s,t)\mapsto (x,s,t)$.
\end{enumerate}

For $t\in [0,1]$, let $E_t\to \gamma \times S^1$ be the restriction of $\tilde{E}$ to $\gamma\times S^1 \times \{t\}$. Then $E_0$ is isomorphic to $E_1$. 
Let $H$ and $L$ be the tautological line bundles over
$\gamma\simeq \RP^1$ and $S^1\simeq \RP^1$, respectively. 
Let $x'=w_1(H)\in H^1(\gamma;\bZ_2)$, $y'=w_1(L)\in H^1(S^1;\bZ_2)$. Let 
$p_1:\gamma\times S^1\to \gamma$ and $p_2:\gamma\times S^1\to S^1$ be projections
to the first and second factors, respectively. Then
$$
H^*(\gamma\times S^1;\bZ_2)=\bZ_2[x,y]/\langle x^2, y^2 \rangle,
$$
where $x=p_1^*x'$ and $y=p_2^*y'$.

Recall that if $V$ is a real
vector bundle of rank $r$ and $\ell$ is a real vector bundle of rank 1, then
$w_1(V\otimes \ell) = w_1(V)+rw_1(\ell)$ and 
$w_2(V\otimes \ell) = w_2(V) +(r-1)w_1(V)w_1(\ell)+(r(r-1)/2) w_1(\ell)^2$. 
We have $E_0\simeq p_1^* E$, $E_1\simeq p_1^* E \otimes p_2^* L$, so the total
Stiefel-Whitney classes are given by
$$
w(E_0)= 1+ p_1^*(w_1(E)), \quad w(E_1) = 1 + p_1^*(w_1(E)) + ry + (r-1)p_1^*(w_1(E))y.
$$
$$
w(E_1)-w(E_0)= ry + (r-1)p_1^*(w_1(E))y =0 \in H^*(\gamma\times S^1;\bZ_2),
$$
where $p_1^*(w_1(E)) \in \{0, x\}$. So we must have $r$ is even and $w_1(E)=0$. 
\end{proof}

\noindent With the help of Lemma \ref{loop}, we can prove that assumption (1) in Theorem \ref{comparison_of_Poincare_series_in_good_cases} yields the desired conclusion.

\begin{proposition}
Let $(E,\tau)$ be a real vector bundle of rank $r$, degree
$d$ over a Klein surface $(M,\si)$ of topological type $(g,n,a)$, where $n>0$.
Let $\iota: \O(1)\to \cGt$ be the group homomorphism defined by $-1\mapsto -I_E$. 
Suppose that either 
\begin{enumerate}
\item $r$ is even and
$w_1(E^\tau)=(w^{(1)}, \ldots, w^{(n)})\neq (0,\ldots,0)$, or 
\item $r$ is odd,
\end{enumerate} 
\noindent then the map
$$\iota_*: \pi_0(\O(1))=\bZ_2 \lra \pi_0(\cGt)$$
is injective. In particular, if $r$ and $d$ are not both even, then $\iota_*$ is injective.
\end{proposition}
\begin{proof}
Let $\gamma_1,\dots, \gamma_n$ be the connected components of $M^\si$.

Suppose that $r$ is even and $(w^{(1)},\ldots, w^{(n)})\neq (0,\ldots, 0)$.
Without loss of generality, assume that $w^{(1)}\neq 0$. Let $E'= E^\tau|_{\gamma_1}
\to \gamma_1$. Then there is a continuous map $\rho: \cGt\to \cG_{E'}$ given
by $u\mapsto u|_{E'}$. By Lemma \ref{loop}, the composition
$\rho\circ \iota: \O(1)\to \cG_{E'}$ induces an injective map
$(\rho\circ \iota)_*= \rho_* \circ \iota_*: \pi_0(\O(1))\to \pi_0(\cG_{E'})$. So
$\iota_*: \pi_1(\O(1))\to \pi_0(\cGt)$ must be injective.

Suppose that $r$ is odd, and let $\rho:\cG^\tau\to \cG_{E'}$ be defined as above.
By Lemma \ref{loop} again, the composition $\rho\circ \iota: \O(1)\to \cG_{E'}$ induces an injective map
$(\rho\circ \iota)_*= \rho_* \circ \iota_*: \pi_0(\O(1))\to \pi_0(\cG_{E'})$. So
$\iota_*: \pi_1(\O(1))\to \pi_0(\cGt)$ must be injective.

Finally, recall from Theorem \ref{top_type_of_bundles} that $$d\equiv w^{(1)}+\cdots+w^{(n)}\ (\mod{2})$$ so, if $r$ and $d$ are not both even, we have that either $r$ is odd or $w_1(E^\tau)\neq 0$.
\end{proof}

\noindent This proves Theorem \ref{comparison_of_Poincare_series_in_good_cases} in the case $n>0$. 
We now treat the case $n=0$. 

\begin{proposition}
Let $(E,\tau)$ be a real or quaternionic vector bundle of rank $r$ over a Klein surface $(M,\si)$ of topological type $(g,0,1)$. Let $\iota: \O(1)\to \cGt$ be the group homomorphism defined by $-1\mapsto -I_E$. If $r$ is odd then the map $$\iota_*: \pi_0(\O(1))=\bZ_2 \lra \pi_0(\cGt)$$ is injective. 
\end{proposition}

\begin{proof}  
Given $u\in \cGt$, define $\phi_u:M\to \U(1)$ by 
$x\mapsto \det(u(x))$. Then $\phi_u(\si(x))= \overline{\phi_u(x)}$
for any $x\in M$. Choose $x_0\in M$. Then $\si(x_0)\neq x_0$ since
$M^\si$ is empty. Let $\delta:[0,1]\to M$ be a path such that
$\delta(0)=x_0$, $\delta(1)=\si(x_0)$. 

Suppose that $\iota_*$ is not an injection. Then there is a path 
$[0,1]\to \cGt$, $t\mapsto u_t$, such that $u_0=I_E$ and $u_1=-I_E$.
Define 
$$
\phi: \begin{array}{ccc} [0,1]\times [0,1]& \lra & \U(1)\\ (s,t) & \longmapsto & \phi_{u_t}(\delta(s))\end{array}.
$$
Then:
\begin{enumerate}
\item For all $s\in [0,1]$, $\phi(s,0)=1$ and $\phi(s,1)=-1$ (the latter because $r$ is odd). 
\item For all $t\in [0,1]$, $\phi(1,t)=\overline{\phi(0,t)}$.
\end{enumerate}

Let $\gamma$ be the following parametrization of the boundary
of $[0,1]\times [0,1]$:
$$
\gamma(t)= \begin{cases}
(4t,0), & 0\leq t\leq \frac{1}{4},\\
(1, 4t-1), & \frac{1}{4}\leq t \leq \frac{1}{2},\\
(3-4t, 1), & \frac{1}{2}\leq t \leq \frac{3}{4},\\
(0, 4-4t), & \frac{3}{4}\leq t\leq 1.
\end{cases}
$$
Define $f:=\phi\circ \gamma: [0,1]\to \U(1)$. Then
$$
f(t)=
\begin{cases}
1, & 0\leq t\leq \frac{1}{4},\\
\phi(1,4t-1), & \frac{1}{4} \leq t \leq \frac{1}{2},\\
-1, & \frac{1}{2}\leq t \leq \frac{3}{4}\\
\overline{\phi(1,4-4t)}=\overline{f(\frac{5}{4}-t)}, & \frac{3}{4} \leq t\leq 1 
\end{cases}
$$
In particular, $f(0)=f(1)=1$, so $f$ defines
a based loop in $(\U(1),1)$. Moreover, this loop is contractible because it
is the restriction of a map from $[0,1]\times [0,1]$ to its boundary. 

Let $\pi:\bR\to \U(1)$, $t\mapsto e^{2\pi i t}$, 
be the universal covering. Then there exists a unique lifting
$\tilde{f}:[0,1]\to \bR$ such that $\pi\circ \tilde{f}(0)=0$. We have
$\tilde{f}(t)=0$ for $0\leq t\leq \frac{1}{4}$, and
$\tilde{f}(t)= k+\frac{1}{2}$ for $\frac{1}{2}\leq t\leq \frac{3}{4}$,
where $k\in \bZ$. We must have
$$
\tilde{f}(t) = 2k+1 -\tilde{f}(\frac{5}{4}-t), \quad \frac{3}{4}\leq t\leq 1, 
$$
so $\tilde{f}(1)= 2k+1 \neq 0 =\tilde{f}(0)$, which contradicts the contractibility
of the loop defined by $f:[0,1]\to \U(1)$. 

We conclude that $i_*:\pi_0(\O(1))\to \pi_0(\cGt)$ must be an injection.
\end{proof}

\noindent This completes the proof of Theorem \ref{comparison_of_Poincare_series_in_good_cases}.

\begin{remark}
Let us take a look at the remaining cases for $r$ and $d$, i.e.\,those not covered by Theorem \ref{comparison_of_Poincare_series_in_good_cases}. These are the cases $n>0$ and $r$ and $d$ both even, as well as $n=0$ and $r$ even. If $n>0$ and $r$ and $d$ are both even, then they are not coprime, so the moduli space $\cM_{(g,n,a)}^{\ \tau}(r,d)$ is \textit{not} the orbit space $\Csst / \ov{\cG_\C^\tau}$ and we would not get the actual cohomology of the moduli space even if we were able to compute $P_t^{\ov{\cG_\C^{\, \tau}}}(\Csst;\Z_2)$. So we content ourselves with the equivariant Poincar\'e series $P_t^{\cG_\C^{\, \tau}}(\Csst;\Z_2)$ in this case. Likewise, if $n=0$ and $r$ is even, then for $\tau=\tauR$, $n=0$ implies $d$ even, and for $\tau=\tauH$, $r$ even implies $d$ even. So, either way, $r$ and $d$ would again not be coprime in this case and we would not get the actual cohomology of the moduli space $\cM_{(g,n,a)}^{\ \tau}(r,d)$ even if we could compute  $P_t^{\ov{\cG_\C^{\, \tau}}}(\Csst;\Z_2)$.
\end{remark}

\subsection{Strange Poincar\'e duality}

As mentioned at the beginning of Subsection \ref{smooth_case}, it is a consequence of Corollary \ref{cohom_of_moduli_space} that, when $r\wedge d=1$, the formal power series $$P_t := (1-t) \, P_{(g,n,a)}^{\ \tau}(r,d)$$ satisfies the Poincar\'e duality relation 
\begin{equation}\label{mod_2_Poincare_duality_bis}
t^{r^2(g-1) + 1} P_{\frac{1}{t}} = P_t
\end{equation} because in this case it is actually the $\mod{2}$ Poincar\'e polynomial of a compact connected smooth manifold of dimension $r^2(g-1)+1$ (note that, as we have pointed out before, if $n>0$, the assumption $r\wedge d=1$ actually implies that $\tau=\tauR$). We now show that relation \eqref{mod_2_Poincare_duality_bis} in fact holds in some other cases too, without any obvious geometric explanation. We begin with the case $n>0$ and $\tau=\tauH$.

\begin{theorem}\label{strange_Poincare_duality_tauH_with_real_points} Assume that $n>0$ and $\tau=\tauH$, in which case one necessarily has $r=2r'$ and $d=2d'$. Assume additionally that $r'\wedge d'=1$. Then the formal power series 
$$
P_t= (1-t) \, P_{(g,n,a)}^{\ \tauH}(2r',2d')
$$ 
satisfies the relation 
$$
t^{(2r')^2(g-1) + 1} P_{\frac{1}{t}} = P_t.
$$
\end{theorem}

\noindent Theorem \ref{strange_Poincare_duality_tauH_with_real_points} is proved below. This Poincar\'e duality type of relation is a rather strange fact because, if $n>0$ and $r$ and $d$ are both even, $(1-t) \, P_{(g,n,a)}^{\ \tauH}(2r',2d')$
is not a polynomial of the expected degree $(2r')^2(g-1)+1$. 
For example, when $r=2$ (see Section \ref{tauH-n-positive}),
$$
P_t =(1-t) \frac{(1+t)^g(1+t^3)^g}{1-t^4} =\frac{(1+t)^{g-1}(1+t^3)^g}{1+t^2}
$$
which is a power series (with infinitely many non-zero terms) that satisfies \eqref{mod_2_Poincare_duality_bis}. In particular, when $r=2$ and $g=2$,
$$
P_t = \frac{(1+t)(1+t^3)^2}{1+t^2} = 1+t-t^2 + t^3 +\, \cdots\, .
$$
Note that the coefficient of $t^2$ is negative.

The $n=0$ case also gives instances of this strange Poincar\'e duality.

\begin{theorem}\label{strange_Poincare_duality_no_real_points}
Assume that $n=0$.
\begin{enumerate}
\item If $\tau=\tauR$, then necessarily $d=2d'$. Assume additionally that $r\wedge d'=1$. 
Then the formal power series 
$$
P_t= (1-t) \, P_{(g,0,1)}^{\ \tauR}(r,2d')
$$ 
satisfies the relation 
$$
t^{r^2(g-1) + 1} P_{\frac{1}{t}} = P_t.
$$
\item If $\tau=\tauH$ and $g=2g'$, then $d=2d'+r$. Assume additionally that $r\wedge d'=1$. 
Then the formal power series 
$$
P_t= (1-t) \, P_{(2g',0,1)}^{\ \tauH}(r,2d'+r)
$$ 
satisfies the relation 
$$
t^{r^2(2g'-1) + 1} P_{\frac{1}{t}} = P_t.
$$
\item If $\tau=\tauH$ and $g=2g'-1$, then $d=2d'$. Assume additionally that $r\wedge d'=1$. 
Then the formal power series 
$$
P_t= (1-t) \, P_{(2g'-1,0,1)}^{\ \tauH}(r,2d')
$$ 
satisfies the relation 
$$
t^{r^2(2g'-2) + 1} P_{\frac{1}{t}} = P_t.
$$
\end{enumerate}
\end{theorem}

To prove Theorems \ref{strange_Poincare_duality_tauH_with_real_points} and \ref{strange_Poincare_duality_no_real_points}, we begin by the following observations, which are straightforward to check. We use the notation of Section \ref{solving_the_recursion}.

\begin{lemma} \label{lm:inverse}
Suppose that 
$$
r_1,\,\cdots, r_l\in \bZ_{>0}, \quad r_1 +\cdots + r_l = r,\quad r\wedge d=1.
$$
If $y=x^{-1}$, then
$$
\frac{y^{M(r_1,\,\cdots, r_l;\frac{d}{r})} }{\prod_{i=1}^{l-1}(1-y^{r_i+r_{i+1}})}
=\frac{(-1)^{l-1} x^{M(r_l,\,\cdots, r_1;\frac{d}{r})} }{\prod_{i=1}^{l-1}(1-x^{r_i+r_{i+1}})}\cdot
$$
\end{lemma}

\begin{lemma}\label{lm:Qt}
For any topological type $(g,n,a)$ and all $r\geq 1$,
$$
\baQ_{(g,n,a)}^{\ \tauR}(r)(\frac{1}{t}) = -\baQ_{(g,n,a)}^{\ \tauR}(r) (t),\quad
\baQ_{(g,n,a)}^{\ \tauH}(r)(\frac{1}{t}) =-\baQ_{(g,n,a)}^{\ \tauH}(r)(t). 
$$
\end{lemma}
 
\begin{proof}[Proof of Theorems \ref{strange_Poincare_duality_tauH_with_real_points} and \ref{strange_Poincare_duality_no_real_points}]
We prove Theorem \ref{strange_Poincare_duality_tauH_with_real_points}. Theorem \ref{strange_Poincare_duality_no_real_points} is similar.
By Theorem \ref{thm:zag}, Theorem \ref{closed-formulae}, Lemma \ref{lm:inverse} and Lemma \ref{lm:Qt},
$$
\bar{P}^{\ \tauH}_{(g,n,a)}(2r',2d')(\frac{1}{t}) = -\bar{P}^{\ \tauH}_{(g,n,a)}(2r',2d')(t)\cdot
$$
Therefore, if $P_t = (1-t)\, P_{(g,n,a)}^{\ \tauH}(2r',2d')(t)$, one has
\begin{eqnarray*}
t^{(2r')^2(g-1)+1}\ P_{1/t} & = & t^{(2r')^2(g-1)+1} \big(1-\frac{1}{t}\big) t^{-\frac{(2r')^2(g-1)}{2}} \bar{P}_{(g,n,a)}^{\ \tauH}(2r',2d')(\frac{1}{t}) \\
& = & t^{\frac{(2r')^2(g-1)}{2}}(t-1)(-\bar{P}_{(g,n,a)}^{\ \tauH}(2r',2d')(t))\\
& = & (1-t)\, P_{(g,n,a)}^{\ \tauH}(2r',2d')(t) \\
& = & P_t\, .
\end{eqnarray*}
\end{proof}

Part (a) of Corollary \ref{equal} gives an alternate proof of part (3) of Theorem \ref{strange_Poincare_duality_no_real_points} (also of part (1) if in addition $g=2g'-1$), since we know that Poincar\'e duality holds for the polynomial $P_{g'}(r,d')$ when $r\wedge d'=1$. We also observe that, when $r$ is odd, the Poincar\'{e} series in Theorem \ref{strange_Poincare_duality_no_real_points} are polynomials (because $r\wedge(2d')=1$ and $r\wedge(2d'+r)=1$ if $r$ is odd and $r\wedge d'=1$), and they satisfy Poincar\'{e} duality. Finally, we note that the proof of Theorems \ref{strange_Poincare_duality_tauH_with_real_points} and \ref{strange_Poincare_duality_no_real_points} also works for $n>0$, $\tau=\tauR$ and $r\wedge d=1$, so Poincar\'e duality again holds for $(1-t) P_{(g,n,a)}^{\ \tauR}(r,d)$, which of course we knew because in this case $(1-t) P_{(g,n,a)}^{\ \tauR}(r,d)$ is, by Corollary \ref{cohom_of_moduli_space}, the $\mod{2}$ Poincar\'e polynomial of the smooth manifold $\cM_{(g,n,a)}^{\ \tauR}(r,d)$.

\subsection{Moduli spaces of vector bundles on maximal real algebraic curves}\label{bundles_on_maximal_curves}

It is a consequence of Smith theory (see, for instance, the exposition in \cite{Borel} or \cite{Wilson}), that, if $Y/\R$ is a smooth, projective variety of dimension $n$ defined over the field of real numbers, one has $$\sum_{i=0}^n b_i\big( Y(\R);\Z/2\Z\big) \leq \sum_{i=0}^{2n} b_i(Y(\C);\Z/2\Z)\, ,$$ where $b_i(*;\Z/2\Z)$ is the dimension of the $\Z/2\Z$-vector space $H^i(*;\Z/2\Z)$. The real algebraic variety $Y$ is then called \textbf{maximal} if this inequality is an equality. For geometrically connected, smooth projective curves defined over the field of real numbers, this amounts to asking that $$b_0\big(Y(\R)\big) + b_1\big(Y(\R)\big) \overset{!}{=} 1 + 2g + 1 = 2(g+1)\, ,$$ which happens exactly when $Y(\R)$ has $(g+1)$ connected components (the maximal possible number, by Harnack's theorem, all of them being copies of $S^1$). In this subsection, we show that $\cM_{M,\si}^{2,2k+1}/\R$ is maximal if and only if the curve $(M,\si)$ is maximal, for all $k\in\Z$. Note that the analogous result is known to hold for $r=1$, for in that case $\mathcal{M}_{M,\si}^{\, 1,d}(\R) = \mathrm{Pic}_{M,\si}^{\ d}(\R)$ is, when $M^{\si}$ has $n$ connected components, a union of $2^{n-1}$ real tori of dimension $g$ (\cite{GH}), so
$$
\sum_{i=0}^g b_i\big(\mathrm{Pic}_{M,\si}^{\ d}(\R)\big) = 2^{n-1} \times \sum_{i=0}^g b_i(\R^g/\Z^g)\ =\  2^{n-1} \times \sum_{i=0}^g\ (^g_i)\ =\  2^{g+n-1}
$$
\noindent and
$$
\sum_{i=0}^{2g} b_i\big(\mathrm{Pic}_{M,\si}^{\ d}(\C)\big)\ =\ \sum_{i=0}^{2g} b_i(\C^g/\Z^{2g})\ =\  \sum_{i=0}^{2g}\ (^{2g}_{\ i})\  =\  2^{2g}.
$$

\noindent So equality holds if and only if $n=g+1$. We note that, on a maximal real algebraic curve and for $r\wedge d=1$, there are no quaternionic bundles of rank $r$ and degree $d$ (as $r$ has to be even when $M^{\si}\neq\emptyset$ and must satisfy $d+r(g-1) \equiv 0\ (\mod{2})$, which implies that $d$ is even when $r$ is even, contradicting coprimality). Moreover, $\Mod(\R)$ has exactly $2^{n-1}$ connected components in this case (\cite{Sch_JSG}) and two stable real bundles of rank $r$ and degree $d$ lie in a same connected component of $\Mod(\R)$ if and only if they have the same Stiefel-Whitney classes (topological types of real bundles, see Theorem \ref{top_types_of_bundles}).

\begin{theorem}\label{maximal_real_alg_var}
Let $k\in \Z$. If the real algebraic curve curve $(M,\si)$ is maximal, then the real algebraic variety $\cM_{M,\si}^{2,2k+1}$ is maximal. If $(M,\si)$ is not maximal, then $\cM_{M,\si}^{2,2k+1}$ is not maximal.
\end{theorem}

\begin{proof}
By tensoring with a suitable power of a real line bundle of degree $1$ over $(M,\si)$, one obtains an isomorphism of real algebraic varieties between $\cM_{M,\si}^{\, 2,1}$ and $\cM_{M,\si}^{2,2k+1}$, so it suffices to prove the result for $\cM_{M,\si}^{\,2,1}$. We need to show that, when $(M,\si)$ is of topological type $(g,g+1,0)$, then $$P_t\big(\cM_{M,\si}^{\, 2,1}(\C)\big)|_{t=1} = P_t\big(\cM_{M,\si}^{\, 2,1}(\R)\big)|_{t=1}\, .$$ Our results show that, when $Y$ is of topological type $(g,g+1,0)$, $\cM_{M,\si}^{\, 2,1}(\R)$ has $2^g$ connected components with the same $\mod{2}$ Poincar\'e series (since this series does not depend on the Stiefel-Whitney classes of real bundles indexing said connected components), so in fact we need to show that $$(1-t^2) P_g(2,1)|_{t=1} = 2^g \times (1-t) P_{(g,g+1,0)}^{\ \tauR}(2,1)|_{t=1}\, .$$ Let us  use the closed formulae of Theorems \ref{thm:zagier} and \ref{thm:closed-formula} (see the Appendix for the explicit formulae in the rank $2$ case).
\begin{eqnarray*}
(1-t^2) P_g(2,1) & = & (1-t^2) \frac{(1+t)^{2g}}{(1-t^2)^2(1-t^4)} \left[ (1+t^3)^{2g} - t^{2g}(1+t)^{2g} \right] \\
& = & \frac{(1+t)^{2g}}{(1-t^2)(1-t^4)} \left[ (1+t^3)^2 - (t+t^2)^2 \right]\\ 
& & \times \sum_{k=0}^{g-1} \big((1+t^3)^{2}\big)^{g-1-k} \big((t+t^2)^2\big)^{k} \\
& = & (1+t)^{2g} \sum_{k=0}^{g-1} \big((1+t^3)^{2}\big)^{g-1-k} \big((t+t^2)^2\big)^{k}\, ,
\end{eqnarray*}

\noindent and 
\begin{eqnarray*}
2^g\times (1-t) P_{(g,g+1,0)}^{\ \tauR}(2,1) & = & 2^g (1-t) \, \frac{(1+t)^{2g-1}}{(1-t)^3}\left[ (1+t^2)^g -(2t)^g \right]\\
& = & 2^g \frac{(1+t)^{2g-1}}{(1-t)^2} [(1+t^2)- 2t] \sum_{k=0}^{g-1} (1+t^2)^{g-1-k}(2t)^k \\
& = & 2^g (1+t)^{2g-1} \sum_{k=0}^{g-1} (1+t^2)^{g-1-k}(2t)^k\, .
\end{eqnarray*}

\noindent So $$(1-t^2) P_g(2,1)|_{t=1} = g\, 2^{4g-2} = 2^g \times (1-t) P_{(g,g+1,0)}^{\ \tauR}(2,1)|_{t=1}\, .$$

For $M$ such that $\Msi$ has $n\leq g+1$ connected components, $\cM_{M,\si}^{\,2,1}$ has $2^{n-1}$ connected components (it is empty if $n=0$, because on a real curve with no real points, the degree of a real bundle must be even), and the Poincar\'e polynomial of a given connected component is 

{\allowdisplaybreaks\begin{eqnarray*}
& & (1-t)\ P_{(g,n,a)}^{\tauR}(2,1)\\
 & = & \frac{(1+t)^{g+n-2}}{(1-t)^2} \left[ (1+t^2)^{n-1}(1+t^3)^{g-n+1} - 2^{n-1}(1+t)^{g-n+1}t^ g\right] \\
& = & \frac{(1+t)^{g+n-2}}{(1-t)^2} \left[ (1+t^2)^{n-1} (1+t^3)^{g-n+1} - (2t)^{n-1}(t+t^2)^{g-n+1}\right]\\
& = &  \frac{(1+t)^{g+n-2}}{(1-t)^2} \left[ (1+t^2)^{n-1} - (2t)^{n-1}\right] (1+t^3)^{g-n+1}\\
& & +  \frac{(1+t)^{g+n-2}}{(1-t)^2} \left[ (1+t^3)^{g-n+1} - (t+t^2)^{g-n+1}\right] (2t)^{n-1} \\
& = & (1+t)^{g+n-2} (1+t^3)^{g-n+1} \sum_{i=0}^{n-2} (1+t^2)^{n-2-i}(2t)^i\\
& & + (1+t)^{g+n-1} (2t)^{n-1}\sum_{i=0}^{g-n} (1+t^3)^{g-n-i} (t+t^2)^{i}
\end{eqnarray*}
}

\noindent since $1-t-t^2+t^3=(1-t)^2(1+t)$. Evaluating at $t=1$ and multiplying by the number of connected components, we see that the total mod $2$ Betti number of $\cM_{M,\si}^{\, 2,1}(\R)$ is $$2^{n-1} \times 2^{2g+n-3} (2g-n+1) = (2g-n+1)2^{2g+2n-4}.$$ The function $x\longmapsto (2g-x+1)2^{2g+2x-4}$, where $g\geq 2$, is strictly increasing on $[0;g+1]$ and the value at $g+1$ is $g\,2^{4g-2}$, so $$P_t\big(\cM_{M,\si}^{\, 2,1}(\R)\big)|_{t=1} = (2g-n+1) 2^{2g+2n-4} < g\, 2^{4g-2} = P_t\big(\cM_{M,\si}^{\, 2,1}(\C)\big)|_{t=1}$$ when $n<g+1$ and $\cM_{M,\si}^{\, 2,1}$ is therefore not maximal in that case.
\end{proof}

Theorem \ref{maximal_real_alg_var} very likely holds in arbitrary rank $r$, for any choice of $d$ coprime to $r$. Using  computer programming to evaluate the formulae in Theorem \ref{thm:zagier} and Theorem \ref{thm:closed-formula} at $t=1$ by Taylor expansion, Erwan Brugall\'{e} has been able to verify that the real algebraic varieties $\cM_{M,\si}^{r, rk+d}$ were indeed maximal whenever $(M,\si)$ is maximal, $r\leq 6$, $d\in\{1;\cdots;r-1\}$ is coprime to $r$ and $k\in\Z$.

\appendix

\section{Computations in Low Rank}

Recall that for any real number $x$, $\langle x\rangle$ is the unique $t\in (0,1]$ with
$x+t\in \bZ$. Given a positive integer $r$ and an integer $d$, there exists
unique $m\in \bZ$ and $k\in \{0,1,\ldots, r-1\}$ such that
$$
d= mr +k.
$$
Then
$$
\langle \frac{d}{r}\rangle = 1-\frac{k}{r},\quad \langle - \frac{d}{r} \rangle = \delta_{0,k}+ \frac{k}{r}.
$$

\subsection{Complex case}\label{sec:complex-case}
We use Zagier's formula (Theorem \ref{thm:zagier}) to compute
$P_g(r,d)$ for $1\leq r\leq 4$. 
{\small \begin{eqnarray*}
P_g(1,d) &=&\frac{(1+t)^{2g}}{1-t^2} \cdot \\
&& \\
P_g(2,d) &=&  \frac{(1+t)^{2g} (1+t^3)^{2g}}{(1-t^2)^2(1-t^4)} 
-\frac{(1+t)^{4g} t^{2g-2+4\langle \frac{d}{2}\rangle }}{(1-t^2)^2(1-t^4)} \cdot 
\end{eqnarray*}
\begin{eqnarray*}
P_g(3,d) & = & \frac{(1+t)^{2g} (1+t^3)^{2g}(1+t^5)^{2g}}{(1-t^2)^2(1-t^4)^2(1-t^6)} \\
& & -\frac{(1+t)^{4g}(1+t^3)^{2g} t^{4g-4} (t^{6\langle \frac{d}{3}\rangle} + t^{6\langle -\frac{d}{3}\rangle})}{(1-t^2)^3 (1-t^4)(1-t^6)}  \\
&&+ \frac{(1+t)^{6g}t^{6g-6 + 4\langle \frac{d}{3}\rangle + 4\langle -\frac{d}{3}\rangle }}{(1-t^2)^3(1-t^4)^2} \cdot 
\end{eqnarray*}
\begin{eqnarray*}
P_g(4,d) & = &  \frac{(1+t)^{2g}(1+t^3)^{2g}(1+t^5)^{2g}(1+t^7)^{2g}}
{(1-t^2)^2(1-t^4)^2(1-t^6)^2(1-t^8)} \\
& & -\frac{(1+t)^{4g}(1+t^3)^{2g}(1+t^5)^{2g}t^{6g-6}(t^{8\langle \frac{d}{4}\rangle} +t^{8\langle -\frac{d}{4}\rangle})}
{(1-t^2)^3(1-t^4)^2(1-t^6)(1-t^8)}\\
& & -\frac{(1+t)^{4g}(1+t^3)^{4g}t^{8g-8 + 8\langle \frac{d}{2}\rangle}}
{(1-t^2)^4(1-t^4)^2(1-t^8)}\\
& & +\frac{(1+t)^{6g}(1+t^3)^{2g}t^{10g-10+ 6\langle\frac{d}{2}\rangle}(t^{4\langle\frac{d}{4}\rangle}  + t^{4 \langle -\frac{d}{4}\rangle})}
{(1-t^2)^4(1-t^4)^2(1-t^6)}\\
& & +\frac{(1+t)^{6g}(1+t^3)^{2g}t^{10g-10+6\langle\frac{d}{4}\rangle + 6\langle -\frac{d}{4}\rangle }}
{(1-t^2)^4(1-t^4)(1-t^6)^2}\\
& & -\frac{(1+t)^{8g}t^{12g-12 +4\langle\frac{d}{2}\rangle +4\langle\frac{d}{4}\rangle + 4\langle-\frac{d}{4}\rangle}}{(1-t^2)^4(1-t^4)^3}\cdot 
\end{eqnarray*} }
\subsection{Real case} \label{sec:real-case}
We use the closed formula (Theorem \ref{thm:closed-formula}) to 
compute $P_{(g,n,a)}^{\ \tauR}(r,d)$ for $1\leq r\leq 4$. 

\subsubsection{The $n=0$ case}
When $n=0$, the degree must be even. 
{\small \begin{eqnarray*}
P^{\ \tauR}_{(g,0,1)}(1,2d) &=&\frac{(1+t)^{g+1}}{1-t^2}  =\frac{(1+t)^g}{1-t} \cdot\\
& & \\
P^{\ \tauR}_{(g,0,1)}(2,2d) &=&  \frac{(1+t)^{g+1} (1+t^3)^{g+1}}{(1-t^2)^2(1-t^4)} 
- \frac{(1+t)^{2g+2} t^{g-1+4\langle\frac{d}{2}\rangle}}{(1-t^2)^2(1-t^4)} \cdot 
\end{eqnarray*}
\begin{eqnarray*}
P^{\ \tauR}_{(g,0,1)}(3,2d)  & = &\frac{(1+t)^{g+1} (1+t^3)^{g+1}(1+t^5)^{g+1}}{(1-t^2)^2(1-t^4)^2(1-t^6)} \\
&& - \frac{(1+t)^{2g+2}(1+t^3)^{g+1} t^{2g-2}(t^{6\langle \frac{d}{3}\rangle} + t^{6\langle -\frac{d}{3}\rangle})}{(1-t^2)^3 (1-t^4)(1-t^6)}  \\
&&+ \frac{(1+t)^{3g+3}t^{3g-3+4\langle \frac{d}{3}\rangle + 4 \langle -\frac{d}{3}\rangle }}{(1-t^2)^3(1-t^4)^2} \cdot
\end{eqnarray*}
\begin{eqnarray*}
P^{\ \tauR}_{(g,0,1)}(4,2d) 
&=& \frac{(1+t)^{g+1}(1+t^3)^{g+1}(1+t^5)^{g+1}(1+t^7)^{g+1}}
{(1-t^2)^2(1-t^4)^2(1-t^6)^2(1-t^8)} \\
&& -\frac{(1+t)^{2g+2}(1+t^3)^{g+1}(1+t^5)^{g+1}t^{3g-3}(t^{8\langle\frac{d}{4}\rangle} + t^{8\langle-\frac{d}{4}})}
{(1-t^2)^3(1-t^4)^2(1-t^6)(1-t^8)}\\
&& -\frac{(1+t)^{2g+2}(1+t^3)^{2g+2}t^{4g-4+8\langle \frac{d}{2}\rangle}}
{(1-t^2)^4(1-t^4)^2(1-t^8)}\\
&&+\frac{(1+t)^{3g+3}(1+t^3)^{g+1}t^{5g-5+ 6\langle \frac{d}{2}\rangle}(t^{4\langle \frac{d}{4}\rangle} + t^{4\langle -\frac{d}{4}\rangle}) }
{(1-t^2)^4(1-t^4)^2(1-t^6)}\\
&& +\frac{(1+t)^{3g+3}(1+t^3)^{g+1}t^{5g-5+12\langle \frac{d}{2}\rangle }}
{(1-t^2)^4(1-t^4)(1-t^6)^2}\\
&&-\frac{(1+t)^{4g+4}t^{6g-6 + 4 \langle\frac{d}{2}\rangle + 4 \langle \frac{d}{4}\rangle + 4 \langle -\frac{d}{4}\rangle} }{(1-t^2)^4(1-t^4)^3} \cdot
\end{eqnarray*} }

\subsubsection{The $n>0$ case}
~
{\small
\begin{eqnarray*}
P^{\ \tauR}_{(g,n,a)}(1,d) &=&\frac{(1+t)^{g+1}}{1-t^2}  =\frac{(1+t)^g}{1-t} \cdot\\
& &\\
P^{\ \tauR}_{(g,n,a)}(2,d) &=&  \frac{(1+t)^{g+n+1} (1+t^2)^n (1+t^3)^{g-n+1}}{(1-t^2)^2(1-t^4)} 
-2^{n-1} \frac{(1+t)^{2g+2} t^{g-1+2\langle \frac{d}{2}\rangle}}{(1-t^2)^3} \cdot 
\end{eqnarray*}
\begin{eqnarray*}
P^{\ \tauR}_{(g,n,a)}(3,d) &=&  \frac{(1+t)^{g+n+1}(1+t^2)^{2n} (1+t^3)^{g+1} (1+t^5)^{g-n+1}}{(1-t^2)^2(1-t^4)^2(1-t^6)} \\
&& -2^{n-1} \frac{(1+t)^{2g+n+2}(1+t^2)^n (1+t^3)^{g-n+1} t^{2g-2}(t^{3\langle \frac{d}{3}\rangle} + t^{3\langle -\frac{d}{3}\rangle}) }{(1-t^2)^3 (1-t^3)(1-t^4)}\\
& &+ 2^{2n-2}\frac{(1+t)^{3g+3} t^{3g-3 + 2 \langle \frac{d}{3}\rangle  + 2 \langle -\frac{d}{3}\rangle }}{(1-t^2)^5} \cdot 
\end{eqnarray*}
\begin{eqnarray*}
& & P^{\ \tauR}_{(g,n,a)}(4,d)\\
 &=& \frac{(1+t)^{g+n+1}(1+t^2)^{2n} (1+t^3)^{g+n+1}(1+t^4)^n (1+t^5)^{g-n+1}(1+t^7)^{g-n+1}}{(1-t^2)^2(1-t^4)^2(1-t^6)^2(1-t^8)} \\
&& -2^{n-1} \frac{(1+t)^{2g+n+2}(1+t^2)^{2n}(1+t^3)^{g+1}(1+t^5)^{g-n+1}t^{3g-3}(t^{4\langle\frac{d}{4}\rangle} + t^{4\langle- \frac{d}{4}\rangle}  )}{(1-t^2)^3(1-t^4)^3(1-t^6)}\\
&& -2^{n-1} \frac{(1+t)^{2g+2n+2} (1+t^2)^{2n} (1+t^3)^{2g-2n+2} t^{4g-4+ 4\langle\frac{d}{2}\rangle} }{(1-t^2)^4(1-t^4)^3} \\ 
&& +2^{2n-2} \frac{(1+t)^{3g+n+3} (1+t^2)^n(1+t^3)^{g-n+1} t^{5g-5}(t^{3 \langle\frac{d}{2}\rangle + 2 \langle\frac{d}{4}\rangle}+ (t^{3 \langle\frac{d}{2}\rangle + 2 \langle- \frac{d}{4}\rangle})}
{(1-t^2)^5(1-t^3)(1-t^4)} \\
&& +2^{2n-2} \frac{(1+t)^{3g+n+3} (1+t^2)^n(1+t^3)^{g-n+1}t^{5g-5+3 \langle\frac{d}{4}\rangle + 3 \langle-\frac{d}{4}\rangle}}{(1-t^2)^4(1-t^3)^2(1-t^4)} \\
&& -2^{3n-3} \frac{(1+t)^{4g+4}t^{6g-6 + 2\langle\frac{d}{2}\rangle + 2\langle\frac{d}{4}\rangle + 2\langle -\frac{d}{4}\rangle }}{(1-t^2)^7}\cdot
\end{eqnarray*}
}

In particular, 
\begin{eqnarray*}
P^{\ \tauR}_{(g,g+1,0)}(2,1) & = & \frac{(1+t)^{2g+2}(1+t^2)^{g+1}}{(1-t^2)^2(1-t^4)} 
- 2^g\frac{(1+t)^{2g+2} t^{g}}{(1-t^2)^3}\\ & = & \frac{(1+t)^{2g-1}}{(1-t)^3}((1+t^2)^g - (2t)^g). 
\end{eqnarray*}
Let $\widehat{\cM}_{(g,g+1,0)}^{\ \tauR}(2,1)$ denote the moduli space
of semi-stable real holomorphic vector bundle of rank $2$, degree $1$, with {\em fixed determinant}, on a {\em maximal} real algebraic curve of
genus $g$. Then
$$
P_t(\widehat{\cM}_{(g,g+1,0)}^{\ \tauR}(2,1);\bZ/2\bZ) 
= \frac{1-t}{(1+t)^g} P^{\ \tauR}_{(g,g+1,0)}(2,1).
$$
So \begin{equation}\label{eqn:Mcurve}
P_t(\widehat{\cM}_{(g,g+1,0)}^{\ \tauR}(2,1);\bZ/2\bZ) 
=\frac{(1+t)^{g-1}}{(1-t)^2} ((1+t^2)^g -(2t)^g).
\end{equation}
The above formula \eqref{eqn:Mcurve} was conjectured by  Saveliev and Wang in \cite{SW} where
they proved the case $g=2$:
$$
P_t(\widehat{\cM}_{(2,3,0)}^{\ \tauR}(2,1);\bZ/2\bZ)  =(1+t)^3.
$$
It implies, in particular, that the moduli space $\widehat{\cM}_{g}(2,1)$ of semi-stable holomorphic vector bundles of rank $2$ and degree $1$ with fixed determinant on a maximal real algebraic curve is a maximal real algebraic variety (recall from \cite{BHH} that $\widehat{\cM}_{(g,g+1,0)}^{\ \tauR}(2,1)$ is connected):
\begin{eqnarray*}
&& P_t(\widehat{\cM}_{g}(2,1);\Z/2\Z)|_{t=1}\\
 & = & \frac{1-t^2}{(1+t)^{2g}} P_g(2,1)|_{t=1} \\
& = & \Bigl(\frac{1}{(1-t^2)^2(1+t^2)} (1-t - t^2 +t^3) \sum_{k=0}^{2g-1} (1+t^3)^k (t+t^2)^{2g-1-k}\Bigr)\Big|_{t=1} \\
& = & g\, 2^{2g-2}\, .
\end{eqnarray*}
is equal to
\begin{eqnarray*}
&& P_t(\widehat{\cM}_{(g,g+1,0)}^{\ \tauR}(2,1);\Z/2\Z)|_{t=1}\\
 & = & \frac{1-t}{(1+t)^{g}} P_{(g,g+1,0)}(2,1)|_{t=1} \\
& = & \Bigl(\frac{(1+t)^{g-1}}{(1-t)^2} (1-2t + t^2) \sum_{k=0}^{g-1} (1+t^2)^k (2t)^{g-1-k}\Bigr)\Big|_{t=1} \\
& = & g\, 2^{2g-2}\, .
\end{eqnarray*}

\subsection{Quaternionic case}

Finally, we have the following formulae in the quaternionic case.

\subsubsection{The $n=0$ case}
If $g=2g'-1$ is odd, then the degree $d$ must be even. By 
Corollary \ref{equal} (a), 
$$
P_{(2g'-1,0,1)}^{\ \tauH}(r,2d') = P_{(2g'-1,0,1)}^{\ \tauR}(r,2d').  
$$
If $g=2g'$ is even, then the degree $d$ is of the
form $d=2d'+r$, where $d'$ is an integer.  By Corollary \ref{equal} (b),
$$
P_{(2g',0,1)}^{\ \tauH}(r,2d'+r) = P_{(2g',0,1)}^{\ \tauR}(r,2d').  
$$
So Section \ref{sec:real-case} contains explicit formulae for
$$
\{ P_{(g,0,1)}^{\ \tauH}(r,d)\mid  d+r(g-1)\equiv 0\ (\mod \ 2),  1\leq r\leq 4\}.
$$

\subsubsection{The $n>0$ case} \label{tauH-n-positive}
In this case, the rank and the degree must both be even.
{\small \begin{eqnarray*}
P^{\ \tauH}_{(g,n,a)}(2,2d) &=&  \frac{(1+t)^g(1+t^3)^g}{1-t^4}\cdot\\
P^{\ \tauH}_{(g,n,a)}(4,2d) &=& \frac{(1+t)^g(1+t^3)^g(1+t^5)^g (1+t^7)^g}{(1-t^4)^2(1-t^8)}\\
&& -\frac{(1+t)^{2g}(1+t^3)^{2g} t^{4g-4 + 8\langle\frac{d}{2}\rangle} }{(1-t^4)^2(1-t^8)} \cdot
\end{eqnarray*} }


\end{document}